\documentclass[reqno,openright,dvips,11pt,myheadings,twoside]{amsart}

\usepackage{ifthen}
\usepackage{amssymb}

\usepackage[pagebackref=true, colorlinks=true, citecolor=blue]{hyperref}

\usepackage[usenames,dvipsnames]{pstricks}
\usepackage{epsfig}
\usepackage{pst-grad} 
\usepackage{pst-plot} 

\newcounter{HaveBBM} 

\ifthenelse{\value{HaveBBM}=1}{
    \usepackage{bbm}
    }{}

\newcounter{dtlForSubmission} 
\newcounter{dtlMarginComments} 
\newcounter{dtlSomeDetail} 
\newcounter{dtlFullDetails} 

\newcounter{DetailLevel} 

\newcommand{\DetailMarginNote}[1]{
    \ifthenelse{\value{DetailLevel}=\value{dtlMarginComments} \or \value{DetailLevel}>\value{dtlMarginComments}}
        {{\small #1}}{}
    }

\newcommand{\DetailSome}[1]{
    \ifthenelse{\value{DetailLevel}=\value{dtlSomeDetail} \or \value{DetailLevel}>\value{dtlSomeDetail}}
        {{\small \textbf{Detailed compile only}: #1}}{}
    }

\newcommand{\DetailFull}[1]{
    \ifthenelse{\value{DetailLevel}=\value{dtlFullDetails} \or \value{DetailLevel}>\value{dtlFullDetails}}
        {{\small \textbf{Detailed compile only}: #1}}{}
    }

\newcommand{\NotDetailSome}[1]{
    \ifthenelse{\value{DetailLevel}=\value{dtlSomeDetail} \or \value{DetailLevel}>\value{dtlSomeDetail}}
        {}{#1}
    }

\newcommand{\NotDetailFull}[1]{
    \ifthenelse{\value{DetailLevel}=\value{dtlFullDetails} \or \value{DetailLevel}>\value{dtlFullDetails}}
        {}{#1}
    }

\newcommand{\DetailSomeElse}[2]{
    \ifthenelse{\value{DetailLevel}=\value{dtlSomeDetail} \or \value{DetailLevel}>\value{dtlSomeDetail}}
        {{\small \textbf{Detailed compile only}: #1}}{#2}
    }

\newcommand{\DetailFullElse}[2]{
    \ifthenelse{\value{DetailLevel}=\value{dtlFullDetails} \or \value{DetailLevel}>\value{dtlFullDetails}}
        {{\small \textbf{Detailed compile only}: #1}}{#2}
    }

\newcommand{\DetailSomeInline}[1]{
    \ifthenelse{\value{DetailLevel}=\value{dtlSomeDetail} \or \value{DetailLevel}>\value{dtlSomeDetail}}
        {{\small #1}}{}
    }

\newcommand{\DetailFullInline}[1]{
    \ifthenelse{\value{DetailLevel}=\value{dtlFullDetails} \or \value{DetailLevel}>\value{dtlFullDetails}}
        {{\small #1}}{}
    }

\newcommand{\DetailSomeElseInline}[2]{
    \ifthenelse{\value{DetailLevel}=\value{dtlSomeDetail} \or \value{DetailLevel}>\value{dtlSomeDetail}}
        {{\small #1}}{#2}
    }

\newcommand{\DetailFullElseInline}[2]{
    \ifthenelse{\value{DetailLevel}=\value{dtlFullD	etails} \or \value{DetailLevel}>\value{dtlFullDetails}}
        {{\small #1}}{#2}
    }

\newcommand{\ExplainDetailLevel}{
    Detail level is
    \ifthenelse{\value{DetailLevel}=\value{dtlForSubmission}}
        {0: for submission}
        {\ifthenelse{\value{DetailLevel}=\value{dtlMarginComments}}
            {1: as for submission but with margin comments}
            {\ifthenelse{\value{DetailLevel}=\value{dtlSomeDetail}}
                {2: some proofs not intended for submission}
               {\ifthenelse{\value{DetailLevel}=\value{dtlFullDetails}}
                   {3: full details}
                   {invalid}
                }
            }
        }
    }

\newcounter{DoAdditionalConstraint} 

\newcommand{\AdditionalConstraint}[2]{
    \ifthenelse{\value{DoAdditionalConstraint}=1}
        {{\small #1}}{#2}
    }

\newtheorem{theorem}{Theorem}[section]
\newtheorem{prop}[theorem]{Proposition}
\newtheorem{lemma}[theorem]{Lemma}
\newtheorem{cor}[theorem]{Corollary}

\theoremstyle{definition}
\newtheorem{definition}[theorem]{Definition}

\newtheorem{remark}[theorem]{Remark}

\numberwithin{equation}{section}

\newcommand{\abs}[1]{\left\vert#1\right\vert}

\newcommand{\skipline}{\vspace{11pt}}

\newcommand{\BB}[1]{\ensuremath{\mathbb{#1}}}

\newcommand{\R}{\ensuremath{\BB{R}}}
\newcommand{\Z}{\ensuremath{\BB{Z}}}
\newcommand{\iny}{\ensuremath{\infty}}
\newcommand{\grad}{\ensuremath{\nabla}}
\ifthenelse{\value{HaveBBM}=1}{
    \newcommand{\CharFunc}{\ensuremath{\mathbbm{1}}}
    }{
    \newcommand{\CharFunc}{\ensuremath{\mathbf{1}}}
    }
\DeclareMathOperator{\dv}{div} %
\DeclareMathOperator{\RE}{Re} %
\newcommand{\prt}{\ensuremath{\partial}}
\newcommand{\brac}[1]{\ensuremath{\left[ #1 \right]}}
\newcommand{\pr}[1]{\ensuremath{\left( #1 \right) }}
\newcommand{\set}[1]{\ensuremath{\left\{ #1 \right\}}}
\newcommand{\smallset}[1]{\ensuremath{\{ #1 \}}}
\newcommand{\norm}[1]{\ensuremath{\left\Vert #1 \right\Vert}}
\newcommand{\smallnorm}[1]{\ensuremath{\Vert #1 \Vert}}

\newcommand{\refS}[1]{Section~\ref{S:#1}}

\newcommand{\refT}[1]{Theorem~\ref{T:#1}}
\newcommand{\refTAnd}[2]{Theorems~\ref{T:#1} and \ref{T:#2}}
\newcommand{\refP}[1]{Proposition~\ref{P:#1}}
\newcommand{\refL}[1]{Lemma~\ref{L:#1}}

\newcommand{\refD}[1]{Definition~\ref{D:#1}}
\newcommand{\refDAnd}[2]{Definitions~\ref{D:#1} and \ref{D:#2}}
\newcommand{\refC}[1]{Corollary~\ref{C:#1}}
\newcommand{\refE}[1]{(\ref{e:#1})}
\newcommand{\refEAnd}[2]{(\ref{e:#1}, \ref{e:#2})}

\newcommand{\refR}[1]{Remark~(\ref{R:#1})}
\newcommand{\eps}{\ensuremath{\epsilon}}
\newcommand{\Cal}[1]{\ensuremath{\mathcal{#1}}}
\newcommand{\al}{\ensuremath{\alpha}}
\newcommand{\la}{\ensuremath{\lambda}}
\newcommand{\pdx}[2]{\frac{\prt #1}{\prt #2}}
\newcommand{\diff}[2]{\frac{ d#1}{d#2}}
\newcommand{\diffn}[3]{\frac{ d^#3#1}{d#2^#3}}
\newcommand{\ol}{\overline}
\newcommand{\smallabs}[1]{\ensuremath{\vert #1 \vert}}
\newcommand{\Y}{\ensuremath{\BB{Y}}}

\newcommand{\ReallyConvex}{convex }
\newcommand{\ReallyConvexPeriod}{convex}

\begin{document}

\raggedbottom

\numberwithin{equation}{section}

%
%
\newcommand{\MarginNote}[1]{
    \marginpar{
        \begin{flushleft}
            \footnotesize #1
        \end{flushleft}
        }
    }

%
%
\newcommand{\NoteToSelf}[1]{
    }

\newcommand{\Detail}[1]{
    \MarginNote{Detail}
    \skipline
    \hspace{+0.25in}\fbox{\parbox{4.25in}{\small #1}}
    \skipline
    }

\newcommand{\Todo}[1]{
    \skipline \noindent \textbf{TODO}:
    #1
    \skipline
    }

\newcommand{\Comment}[1] {
    \skipline
    \hspace{+0.25in}\fbox{\parbox{4.25in}{\small \textbf{Comment}: #1}}
    \skipline
    }

%
%

\newcommand{\IntTR}
    {\int_{t_0}^{t_1} \int_{\R^d}}

\newcommand{\IntAll}
    {\int_{-\iny}^\iny}

\newcommand{\Schwartz}
    {\ensuremath \Cal{S}}

\newcommand{\SchwartzR}
    {\ensuremath \Schwartz (\R)}

\newcommand{\SchwartzRd}
    {\ensuremath \Schwartz (\R^d)}

\newcommand{\SchwartzDual}
    {\ensuremath \Cal{S}'}

\newcommand{\SchwartzRDual}
    {\ensuremath \Schwartz' (\R)}

\newcommand{\SchwartzRdDual}
    {\ensuremath \Schwartz' (\R^d)}

\newcommand{\HSNorm}[1]
    {\norm{#1}_{H^s(\R^2)}}

\newcommand{\HSNormA}[2]
    {\norm{#1}_{H^{#2}(\R^2)}}

\newcommand{\Holder}
    {H\"{o}lder }

\newcommand{\Holders}
    {H\"{o}lder's }

\newcommand{\Holderian}
    {H\"{o}lderian }

\newcommand{\HolderRNorm}[1]
    {\widetilde{\Vert}{#1}\Vert_r}

\newcommand{\LInfNorm}[1]
    {\norm{#1}_{L^\iny(\Omega)}}

\newcommand{\SmallLInfNorm}[1]
    {\smallnorm{#1}_{L^\iny}}

\newcommand{\LOneNorm}[1]
    {\norm{#1}_{L^1}}

\newcommand{\SmallLOneNorm}[1]
    {\smallnorm{#1}_{L^1}}

\newcommand{\LTwoNorm}[1]
    {\norm{#1}_{L^2(\Omega)}}

\newcommand{\SmallLTwoNorm}[1]
    {\smallnorm{#1}_{L^2}}

\newcommand{\LpNorm}[2]
    {\norm{#1}_{L^{#2}(\R^2)}}

\newcommand{\SmallLpNorm}[2]
    {\smallnorm{#1}_{L^{#2}}}

\newcommand{\lOneNorm}[1]
    {\norm{#1}_{l^1}}

\newcommand{\lTwoNorm}[1]
    {\norm{#1}_{l^2}}

\newcommand{\MsrNorm}[1]
    {\norm{#1}_{\Cal{M}}}

\newcommand{\wh}{\widehat}

\newcommand{\FTF}
    {\Cal{F}}

\newcommand{\FTR}
    {\Cal{F}^{-1}}

\newcommand{\InvLaplacian}
    {\ensuremath{\widetilde{\Delta}^{-1}}}

\newcommand{\EqDef}
    {\stackrel{\text{\tiny def}}{=}}


\newcommand{\Ignore}[1] {}

\newcommand{\Obsolete}[1] {}

\newcommand{\OverAgressiveButWorthKeeping}[1] {}

%
%

%
%

\title
    [The flow for {E}uler equations]
    {On the flow map for 2{D} Euler equations with unbounded vorticity}

\author{James P. Kelliher}
\address{Department of Mathematics, University of California, Riverside, 900 University Ave.,
Riverside, CA 92521}
\email{kelliher@math.ucr.edu}

\date{} 



\begin{abstract}

In Part I, we construct a class of examples of initial velocities for which the unique solution
to the Euler equations in the plane has an associated flow map that lies in no
\Holder space of positive exponent for any positive time.
In Part II, we explore inverse problems that arise in attempting to construct an example of an initial velocity producing an arbitrarily poor modulus of continuity of the flow map.
\end{abstract}

\date{7 July 2011 (compiled on \today)}

\maketitle

\Ignore{ 
\begin{small}
    \begin{flushright}
        Compiled on \textit{\textbf{\today}}
    \end{flushright}
\end{small}
} 

\DetailMarginNote{
    \begin{small}
        \begin{flushright}
            Compiled on \textit{\textbf{\today}}

            \ExplainDetailLevel

        \end{flushright}
    \end{small}
    }

\tableofcontents

\newpage

%
%
\section*{Overview}

\noindent In \cite{Y1995}, V. I. Yudovich described what is still the weakest class of initial velocities for which solutions to the 2D Euler equations are known to be well-posed. (This extended his bounded vorticity result of \cite{Y1963}.) Any solution with initial velocity in this class, which we will call $\Y$, has a unique velocity field and flow map that are continuous  in space and time with explicit upper bounds on their spatial moduli of continuity (MOC). The velocity field is both Osgood- and Dini-continuous (\refEAnd{muOsgood}{Smu}).


In this paper we initiate the investigation of the fundamental question of how ``bad'' the MOC of the flow can be:

\begin{quote}
	\textbf{Fundamental question}: Given any strictly increasing concave MOC, $f$, and positive time,
	$t$, does there exist
	an initial velocity in the class $\Y$ for which any MOC of the flow map at time $t$ is at
	least as large as $f$ on some nonempty open interval, $(0, a)$?	
\end{quote}

Past studies of properties of the flow map tend to ask the opposite question: ``How smooth is the flow map?'' (A notable exception is an example by Bahouri and Chemin in \cite{BC1994}, upon which we build in Part I.) In particular, \cite{Sueur2010} is an important recent study of the smoothness of flow maps for initial velocities in $\Y$.

If the answer to our fundamental question is ``yes,'' it would support the idea that the class $\Y$ is near the edge of uniqueness for solutions to the Euler equations. What is meant by this unavoidably imprecise statement is that, although uniqueness of a solution to the Euler equations does not necessarily require the uniqueness (or even existence) of a (classical) flow (for instance, see \cite{CK2006}), the two ideas seem to be closely entwined. This is evident in the observations of Yudovich in \cite{Y1995} as well as in the approach of Vishik in \cite{V1999}, which relies on properties of the flow.

\Ignore{ 
Also, the classical proof of existence of a flow for a vector field (which goes back at least as far as the ideas of Osgood in \cite{O1898}) also yields uniqueness. Hence, if one is near the edge of having uniqueness of the flow one is presumably near the edge of having uniqueness of the vector field.
} 

On the other hand, if the answer to this question is ``no,'' then it means that there is some hope of extending $\Y$ to obtain a larger class of initial velocities for which both existence and uniqueness in 2D can be proven. And if the answer is ``no'' there still remains the question of characterizing those MOC that \textit{can} be achieved.

To answer our fundamental question ``yes'' we have little choice but to specially construct an initial velocity for which we can obtain a lower bound on the MOC of its flow that is greater than a given $f$.
To answer our fundamental question ``no'' we have little choice but to show that the upper bound on a MOC of the flow that results from the classical theory cannot be arbitrarily large because of some underlying property of the Euler equations.

These two opposing approaches involve very different kinds of techniques, the former more closely tied to classical fluid mechanics the latter more closely tied to the theory of functional equations. In Part I we explore the first approach while in Part II we explore the second approach, in each case obtaining partial answers, but leaving the final answer unresolved.

\Ignore{ 
Yudovich showed in \cite{Y1963} that there exists a unique solution
to the Euler equations in a bounded domain of the plane with bounded initial
vorticity. Also, there exists a unique continuous flow, and this flow lies in
the \Holder space of exponent $e^{-Ct}$ for all positive time $t$. Yudovich's
result is easily modified to apply to solutions in the whole plane. Bahouri and
Chemin in \cite{BC1994} showed that this regularity of the flow was in a sense
optimal by constructing an example for which the flow lies in no \Holder space
of exponent higher than $e^{-t}$.

In \cite{Y1995} Yudovich extended his uniqueness result to a certain
class $\Y$ of unbounded vorticities and showed that there exists as
well a unique flow. There is an upper bound on the modulus of
continuity (MOC) of this flow that depends on how unbounded the initial vorticity is.
} 

This paper is organized as follows:

\textbf{Part I}: Yudovich showed in \cite{Y1963} that for bounded initial vorticity the flow map lies in
the \Holder space of exponent $e^{-Ct}$ for all positive time $t$. Bahouri and
Chemin in \cite{BC1994} showed that this regularity of the flow was in a sense
optimal by constructing an example for which the flow lies in no \Holder space
of exponent higher than $e^{-t}$. We extend the example of Bahouri and Chemin in \cite{BC1994}
to a class of initial vorticities in $\Y$ having a point singularity. We show
that for some such initial vorticities the flow lies in no \Holder space of
positive exponent for any positive time (\refC{MainResultPartI}). In \refS{RemarksPartI} we indicate a possible approach to extending this result to obtain still poorer MOC, a subject of future work.

\textbf{Part II}: A MOC of the flow map can be derived in terms of a MOC of the vector field, as long as the vector field's MOC satisfies an Osgood condition (see \refE{muOsgood}). We examine the inverse problem: given a MOC of the flow, obtain a MOC of a vector field from which it can be derived. We do this first for a general flow and vector field, in a manner that has application beyond solutions to the Euler equations, then specialize to solutions to the Euler equations with Yudovich velocity, where there are further restrictions on both MOC.

We will show that if the MOC, $\Gamma(t, x)$, of the flow map is concave for all $t > 0$ then one can find a necessarily concave MOC, $\mu$, of the velocity field that yields an upper bound on the MOC of the flow map at least as large as $\Gamma(t_0, x)$ at any fixed time $t_0 > 0$ (\refT{ConcaveCIG}, \refR{AnswerToQuestion1}). We identify additional constraints (\refE{YudoCond}) required on the MOC, $\mu$, however, and it is left as an open problem whether such constraints can be satisfied.
We show that $\mu$ is Dini-continuous in \refS{Dini} and explore a useful implication of this property.
Given the constraints on $\mu$ identified in \refE{YudoCond}, we show in \refS{InvertingEulerToGetVectorField} how the $L^p$-norms of the Yudovich vorticity field can be recovered from $\mu$ (\refE{alLegendre} and \refT{Invertmu}). Finally, in \refS{Recoveringomega0}, we discuss how to obtain a vorticity field having, asymptotically in large $p$,  the given $L^p$-norms.

\Ignore{ 
It would be useful to know
whether this upper bound is achieved for certain initial vorticities; that is, whether one can find examples
for which the upper bound is also a lower bound (to within a constant factor).
The hope is that this might cast
some light on how near to the ``edge of uniqueness'' the class $\Y$ has
brought the Euler equations.

We give some partial information by 
} 

\bigskip
\begin{center}
\textbf{Part I: Yudovich velocities}
\phantomsection   
\addcontentsline{toc}{chapter}{Part I: Yudovich velocities}
\end{center}

\section{Introduction}\label{S:YudoTheorem}

\noindent 	The Euler equations describe the flow of an incompressible, constant-density, zero-viscosity fluid in a stationary frame, with the effects of temperature and other physical factors ignored. Letting $v$ be the velocity field and $p$ the pressure field for the fluid, we can write these equations in dimensionless form as
\begin{align*}
    \begin{matrix}
        \left\{
            \begin{array}{rll}
                \prt_t v + v \cdot \grad v + \grad p &= f
                	&  \text{ on } [0, T] \times \Omega, \\
                    \dv v &= 0 &  \text{ on } [0, T] \times \Omega, \\
		v & = v^0 & \text{ on } \set{0} \times \Omega.
            \end{array}
            \right.
    \end{matrix}
\end{align*}
Here, $f$ is the external force, $v^0$ the initial velocity, $T > 0$, and $\Omega \subseteq \R^d$, $d \ge 2$, the domain in which the fluid lies. When $\Omega$ is not all of $\R^d$, we impose the \textit{no-penetration} boundary conditions,
$
	v \cdot \mathbf{n} = 0 \text{ on } \prt \Omega,
$
where $\mathbf{n}$ is the outward unit normal to the boundary.

In this paper neither the external force nor the effect of the boundary will play an important role, though the dimension will. Thus, we will assume that $\Omega = \R^2$, which means that we can also write the Euler equations in their vorticity formulation and can allow $T$ to be arbitrarily large:
\begin{align*}
    \begin{matrix}
        \left\{
            \begin{array}{rll}
                \prt_t \omega+ v \cdot \grad \omega &= 0
                	&  \text{ on } [0, T] \times \Omega, \\
                    v &= K * \omega &  \text{ on } [0, T] \times \Omega, \\
		\omega & = \omega^0 & \text{ on } \set{0} \times \Omega.
            \end{array}
            \right.
    \end{matrix}
\end{align*}
Here, $\omega = \omega(v) = \prt_1 v^2 - \prt_2 v^1$ is the \textit{vorticity} (scalar curl) of the velocity with $\omega^0 = \omega(v^0)$ and $K$ is the Biot-Savart kernel,
\begin{align}\label{e:K}
	 K(x) = \frac{1}{2 \pi} \frac{x^\perp}{\abs{x}^2},
\end{align}
where $x^\perp = (-x_2, x_1)$. (We are following the common convention of giving $\omega$ a dual meaning both as a function and as a variable.)

Suppose that $\omega$ is a scalar field lying in $L^p$ for all $p$ in $[p_0, \iny)$ for some $p_0$ in $[1, 2)$. Let
\begin{align}\label{e:thetaal}
	\theta(p) = \norm{\omega}_{L^p}, \quad \al(\eps) = \eps^{-1} \theta(\eps^{-1}),
\end{align}
and define
\begin{align}\label{e:mu}
	\mu(x)
		&= \inf \set{x^{1 - 2 \eps} \al(\eps) \colon \eps \text{ in } (0, 1/2]}.
\end{align}
A classical result of measure theory is that $p \log \theta(p)$ is convex; this fact will play an important role in Part II.

\begin{definition}\label{D:YudovichVorticity}
We say that $\omega$ is a \textit{Yudovich vorticity} if it is compactly supported (this is not essential, but will simplify our presentation) that satisfies the Osgood condition,
\begin{align}\label{e:muOsgood}
	\int_0^1 \frac{dx}{\mu(x)} = \iny.
\end{align}
\end{definition}
Examples of Yudovich vorticities are
\begin{align}\label{e:YudovichExamples}
    \theta_0(p) = 1,
    \theta_1(p) = \log p, \dots,
    \theta_m(p) = \log p \cdot \log^2 p \cdots \log^m p,
\end{align}
where $\log^m$ is $\log$ composed with itself $m$ times. These examples are described in \cite{Y1995} (see also \cite{K2003}.) Roughly speaking,
the $L^p$--norm of a Yudovich vorticity can grow in $p$ only slightly faster
than $\log p$. Such growth in the $L^p$--norms arises,
for example, from a point singularity of the type $\log \log (1/\abs{x})$.

We define the class, $\Y$, of Yudovich velocities to be
\begin{align*}
	\Y = \set{K * \omega \colon \omega \text{ is a Yudovich vorticity}}.
\end{align*}
A Yudovich velocity will always lie in a space, $E_m$, as defined in \cite{C1998}: Let $\sigma$ be a
\textit{stationary vector field}, meaning that $\sigma$ is of the form
\begin{align}\label{e:Stationary}
    \sigma = \pr{-\frac{x_2}{r^2} \int_0^r \rho g(\rho) \, d \rho,
                \;\frac{x_1}{r^2} \int_0^r \rho g(\rho) \, d \rho}
\end{align}
for some $g$ in $C_C^\iny(\R)$ with $\int_{\R^2} g = 1$.
For any real number $m$, a vector $v$
belongs to $E_m$ if it is divergence-free and can be written in the form $v =
m \sigma + v'$, where $v'$ is in $L^2(\R^2)$. 
$E_m$ is an affine space; having fixed the origin,
$m \sigma$, in $E_m$, we can define a norm by $\norm{m \sigma + v'}_{E_m} =
\LTwoNorm{v'}$. Convergence in $E_m$ is equivalent to convergence in the
$L^2$--norm to a vector in $E_m$.

\Ignore{ 
We will further restrict our initial velocities to have vorticities that are
``only slightly unbounded'' in a sense we now make precise.

\begin{definition}\label{D:Admissible}
    Let $\theta: [p_0, \iny) \to \R^+$ for some $p_0$ in $[1, 2)$ with
    $\theta$ being $C^\iny$ on $(p_0, \iny)$ and $p \log \theta(p)$
    convex. We say that
    $\theta$ is \textit{admissible} if the function $\beta_M:(0, \iny) \to
    [0, \iny)$ defined, for some $M > 0$, by\footnote{The definition of
    $\beta$ in \refE{Beta} differs from that in \cite{K2003} in that it
    directly incorporates the factor of $p$ that appears in the
    Calder\'on-Zygmund inequality;
    in \cite{K2003} this factor is
    included in the equivalent of \refE{betaOsgood}.}
    \begin{align}\label{e:Beta}
        \beta_M(x)
            := C \inf \set{M^\eps x^{1-\eps} \eps^{-1} \theta(\eps^{-1}):
                \eps \text{ in } (0, 1/p_0]},
    \end{align}
    where $C$ is a fixed absolute constant, satisfies the Osgood condition,
    \begin{align}\label{e:betaOsgood}
        \int_0^1 \frac{dx}{\beta_M(x)}
            = \iny.
    \end{align}
\end{definition}

\begin{remark}
	We will not need the requirement that $\theta$ be $C^\iny$ with $p \log \theta(p)$ convex
	until Part II. This restriction appears because of \refL{logThetaConvex}.
\end{remark}

Because
\begin{align}\label{e:BetaCOV}
    \begin{split}
        \beta_M(x)
        &= C M \frac{x}{M} \inf \set{(x/M) ^{-\eps} \eps^{-1} \theta(\eps^{-1}):
                    \eps \text{ in } (0, 1/p_0]} \\
        &= M \beta_1(x/M),
    \end{split}
\end{align}
this definition is independent of the value of $M$.

\begin{lemma}\label{L:betaConcave}
The function $\beta_M$ is
strictly increasing and continuous with $\lim_{x \to 0^+} \beta_M(x)
= 0$ and $\beta_M(x)/x$ strictly decreases 
with $\lim_{x \to \iny} \beta_M(x)/x = 0$. Also, the function $\beta_M$ is concave.
\end{lemma}
\begin{proof}
Because $\beta_M(x) = M \beta_1(x/M)$ it suffices to assume that $M = 1$. 
By \refE{Beta}, since $x^{1 -\eps}$ strictly increases, $\beta_1$ is a
strictly increasing continuous function and $\lim_{x \to 0^+} \beta_1(x)
= 0$. By \refE{BetaCOV}, since $x^{-\eps}$ strictly decreases to zero, $\beta_1(x)/x$ is a strictly decreasing continuous function with $\lim_{x \to \iny} \beta_1(x)/x = 0$.

Since the function $x^{1 - \eps}$ is concave, for any $\la$ in $[0, 1]$ and $x, y$ in $(0, \iny)$,
\begin{align*}
	\beta_1(\la x &+ (1 - \la) y)
		= C \inf_{\eps \in A} \set{(\la x + (1 - \la) y)^{1-\eps} \eps^{-1} \theta(\eps^{-1})} \\
		&\ge C \inf_{\eps \in A} \set{(\la x^{1 - \eps} + (1 - \la) y^{1-\eps}) \eps^{-1} \theta(\eps^{-1})} \\
		&\ge C \la \inf_{\eps \in A} \set{x^{1 - \eps} \eps^{-1} \theta(\eps^{-1})}
			+ C (1 - \la) \inf_{\eps \in A} \set{y^{1 - \eps} \eps^{-1} \theta(1/\eps)} \\
		&= \la \beta_1(x) + (1 - \la) \beta_1(y),
\end{align*}
where $A = (0, 1/p_0]$. It follows that $\beta_1$ is concave.
\end{proof}

In fact, the function $\beta_M$has all of the properties that the function $\mu$ has in \refT{muConcave}, the proof being essentially the same as the proof of that lemma. We will not need any of these properties, however, until Part II.

\begin{definition}\label{D:YudovichVorticity}
    We say that a vector field $v$ has \textit{Yudovich vorticity} if the function
    $\theta: [p_0, \iny) \to \R^+$ defined by $\theta(p) = \norm{\omega(v)}_{L^p}$
    is admissible, where $p_0$ lies in $[1, 2)$.
\end{definition}

\Obsolete{
    \begin{remark}\label{R:p0}
        The proof of the existence of a solution to the Euler equations requires
        that $p_0$ be less than 2. The proof of uniqueness requires only that $p_0$
        be finite.
    \end{remark}
    }

Examples of admissible bounds on vorticity are
\begin{align}\label{e:YudovichExamples}
    \theta_0(p) = 1,
    \theta_1(p) = \log p, \dots,
    \theta_m(p) = \log p \cdot \log^2 p \cdots \log^m p,
\end{align}
where $\log^m$ is $\log$ composed with itself $m$ times. These admissible
bounds are described in \cite{Y1995} (see also \cite{K2003}.) Roughly speaking,
the $L^p$--norm of a Yudovich vorticity can grow in $p$ only slightly faster
than $\log p$ and still be admissible. Such growth in the $L^p$--norms arises,
for example, from a point singularity of the type $\log \log (1/\abs{x})$ (see
\refL{lnLpnorm}).
} 

Given $\theta$ as above, we define the function space,
\begin{align}\label{e:Ytheta}
    \Y_\theta &= \set{v \in E_m: \norm{\omega(v)}_{L^p}
                \le C \theta(p) \text{ for all } p
                    \text{ in } [p_0, \iny)},
\end{align}
for some constant $C$. We define the norm on $\Y_\theta$ to be
\begin{align}\label{e:YNorm}
    \smallnorm{v}_{\Y_\theta}
        = \norm{v}_{E_m}
            + \sup_{p \in [p_0, \iny)} \norm{\omega(v)}_{L^p}/\theta(p).
\end{align}

\Ignore{ 
We also define the space
\begin{align*}
    \Y &= \set{v \in Y_\theta: \theta \text{ is admissible}},
\end{align*}
but place no norm on this space.
} 

\begin{definition}\label{D:MOC}
We say that a continuous function $f \colon [0, \iny) \to [0, \iny)$ with $f(0) = 0$ is a modulus of continuity (MOC). When we say that a MOC, $f$, is $C^k$, $k \ge 0$, we mean that it is continuous on $[0, \iny)$ and $C^k$ on $(0, \iny)$.
\end{definition}
\noindent A real-valued function or vector field, $v$, on a normed linear space, $X$, admits $f$ as a MOC if $\abs{v(x) - v(y)} \le f(\norm{x - y}_X)$ for all $x, y$ in $X$.

In \refD{MOC} we do not require $f$ to be concave: the MOC we will work with in Part I will each have that property (see \refT{muConcave}), but we will not need this until Part II.

The final thing we must do before stating Yudovich's theorem is to define what
we mean by a weak solution to the Euler equations.

\begin{definition}[Weak Euler Solutions]\label{D:WeakSolutionE}
    Given an initial velocity $v^0$ in $\Y_\theta$, $v$
    in $L^\iny([0, T]; \Y_\theta)$ is a weak solution to the Euler
    equations (without forcing) if $v(0) = v^0$ and
    \begin{align*}
        \mathbf{(E)}
        \qquad\diff{}{t} \int_\Omega v \cdot \varphi
            + \int_\Omega (v \cdot \grad v) \cdot \varphi
            = 0
    \end{align*}
    for all divergence-free $\varphi$ in $(H^1(\R^2))^2$.
\end{definition}

Our form of the statement of Yudovich's theorem is a generalization of the
statement of Theorem 5.1.1 of \cite{C1998} from bounded to unbounded vorticity.

\begin{theorem}[Yudovich's Theorem for Unbounded Vorticity]\label{T:YudovichTheorem}
    \textbf{First part}:
    For any $v^0$ in $\Y$ there exists a unique weak solution $v$ of ($E$).
    Moreover, $v$ is in $C(\R; E_m) \cap L_{loc}^\iny(\R; L^\iny(\R^2))$ and
    \begin{align}\label{e:ConservationOfVorticity}
        \LpNorm{\omega(t)}{p} = \SmallLpNorm{\omega^0}{p}
            \text{ for all } p_0 \le p < \iny.
    \end{align}

    \noindent \textbf{Second part}:
    The vector field has a unique continuous flow. More precisely,
    if $v^0$ is in $\Y_\theta$ then there exists
    a unique mapping $\psi$, continuous from $\R \times \R^2$ to
    $\R^2$, such that
    \begin{align*}
        \psi(t, x) = x + \int_0^t v(s, \psi(s, x)) \, ds.
    \end{align*}
    Let $\Gamma_t \colon [0, \iny) \to [0, \iny)$ be defined by
    $\Gamma_t(0) = 0$ and for $s > 0$ by
    \begin{align}\label{e:MOCFlow}
        \int_{s}^{\Gamma_t(s)}  \frac{dr}{\mu (r)} = t.
    \end{align}
    Then $\delta \mapsto \Gamma_t(\delta)$ is an upper bound on the
    MOC of the flow at time $t > 0$; that is,
    for all $x$ and $y$ in $\R^2$
    \begin{align}\label{e:psiMOC}
        \abs{\psi(t, x) - \psi(t, y)}
            \le \Gamma_t(\abs{x - y}).
    \end{align}
    Also, for all $x$ and $y$ in $\R^2$
    \begin{align}\label{e:vMuBound}
        \abs{v(t, x) - v(t, y)}
            \le \mu(\abs{x - y}).
    \end{align}
\end{theorem}

\begin{remark}
As we shall see in \refT{muConcave}, $\mu$ is concave, giving it sublinear growth; hence, $\int_1^\iny \mu(r)^{-1} \, dr = \iny$. This makes $\Gamma_t(s)$ well-defined by \refE{MOCFlow} for all $s > 0$. Because $\mu$ is Osgood, $\Gamma_t(s)$ decreases to 0 as $s \to 0^+$.
\end{remark}

Existence in the first part of Yudovich's theorem can be established, for
instance, by modifying the approach on p. 311-319 of \cite{MB2002}, which
establishes existence under the assumption of bounded vorticity; the uniqueness
argument is given by Yudovich in \cite{Y1995}. The second part is Theorem 2 of
\cite{Y1995}, the bound on the MOC of the flow following from
working out the details of Yudovich's proof (see Sections 5.2 through 5.4 of
\cite{K2005Thesis}).

\begin{remark}
More properly, \refEAnd{MOCFlow}{vMuBound} should be
\begin{align*}
	\int_{2s}^{\Gamma_t(2s)}  \frac{dr}{\mu (r)} = \frac{C t}{2} \text{ and }
	\abs{v(t, x) - v(t, y)}
            \le C \mu(\abs{x - y}/2),
\end{align*}
where $C$ is an absolute constant. We use the simpler forms in \refEAnd{MOCFlow}{vMuBound}, however, because they only result in changes in insignificant constants.
\end{remark}

%
%
\section{Square-symmetric vorticity}\label{S:SquareSymmetric}

\noindent Ignoring for the moment the Euler equations, we will assume that the
vorticity has certain symmetries, and from these symmetries deduce some useful
properties of the divergence-free velocity having the given vorticity. In
\refS{ExampleOverTime}, we will then consider a solution to the
Euler equations whose initial vorticity possesses these symmetries.

For convenience, we number the quadrants in the plane $Q_1$ through $Q_4$,
starting with
\[
    Q_1 = \set{(x_1, x_2): x_1 \ge 0, x_2 \ge 0}
\]
and moving counterclockwise through the quadrants.

\begin{definition}\label{D:SBQ}
    We say that a Yudovich vorticity (vorticity as in
    \refD{YudovichVorticity}) is \textit{symmetric by quadrant}, or SBQ, if
    $\omega$ is compactly supported and
            $\omega(x) = \omega(x_1, x_2)$ is odd in $x_1$ and $x_2$;
            that is, $\omega(-x_1, x_2) = - \omega(x_1, x_2)$ for $x_1 \ne 0$ and
            $\omega(x_1, -x_2) = -\omega(x_1, x_2)$ for $x_2 \ne 0$---so also
            $\omega(-x) = \omega(x)$ when $x$ lies on neither axis.
\end{definition}

\begin{lemma}\label{L:SBQ}
    Let $\omega$ be SBQ. Then there exists a unique vector field $v$ in
    $E_0 \cap \Y$ with $\omega(v) = \omega$, and $v$ satisfies
    the following:
    \begin{enumerate}
        \item
            $v_2(x_1, 0) = 0$ for all $x_1$ in $\R$;

        \item
            $v_1(0, x_2) = 0$ for all $x_2$ in $\R$;

        \item
            $v(0, 0) = 0$.
    \end{enumerate}
    If, in addition, $\omega \ge 0$ in $Q_1$, then
    \begin{itemize}
        \item[(4)]
            $v_1(x_1, 0) \ge 0$ for all $x_1 \ge 0$.
    \end{itemize}
\end{lemma}
\begin{proof}
    Let $p$ be in $[1, 2)$ and let
    $q > 2p/(2 - p)$. By Proposition 3.1.1 p. 44 of \cite{C1998}, for
    any vorticity $\omega$ in $L^p$ there exists a unique divergence-free
    vector field $v$ in $L^p + L^q$ whose
    curl is $\omega$, with $v$ being given by the Biot-Savart law,
    \begin{align}\label{e:BSLaw}
        v = K*\omega.
    \end{align}
    Here, $K$ is the Biot-Savart kernel of \refE{K},
    which decays like $C/\abs{x}$ with a singularity of order $C/\abs{x}$
    at the origin.

    Because $\omega$ is compactly supported and lies in
    $L^2(\R^2)$, $\omega$ is in $L^p(\R^2)$, and \refE{BSLaw} gives our
    velocity $v$, unique in all the spaces $L^p + L^q$. Also, because
    $\int_{\R^2} \omega = 0$, $v$ is in $(L^2)^2 = E_0$ (see the comment
    following Definition 1.3.3 of \cite{C1998}, for instance).

    Then
    \begin{align*}
        v_1(x_1, 0)
        &= \frac{1}{2 \pi} \int_{\R^2}
                \frac{y_2}{\abs{x - y}^2}
                \omega(y) \, dy
            = \frac{1}{2 \pi} \int_{\R^2}
                \frac{y_2}{\pr{x_1 - y_1}^2 + y_2^2}
                \omega(y) \, dy \\
        &= \frac{1}{2 \pi} \sum_{j=1}^4
                \int_{Q_j} \frac{y_2}{\pr{x_1 - y_1}^2 + y_2^2}
                \omega(y) \, dy.
    \end{align*}

    Making the changes of variables, $u = (-y_1, y_2)$, $u = -y$, and $u =
    (y_1, -y_2)$ on $Q_2$, $Q_3$, and $Q_4$, respectively, in all cases the
    determinant of the Jacobian is $\pm 1$, and we obtain
    \begin{align*}
        v_1(x_1, 0)
        &= \frac{1}{2 \pi} \left[
                \int_{Q_1} \frac{y_2}{\pr{x_1 - y_1}^2 + y_2^2}
                \omega(y) \, dy \right.
            - \int_{Q_1} \frac{u_2}{\pr{x_1 + u_1}^2 + u_2^2}
                \omega(u) \, du \\
        &\quad+\left. \int_{Q_1} \frac{u_2}{\pr{x_1 + u_1}^2 + u_2^2}
                \omega(u) \, du
            - \int_{Q_1} \frac{u_2}{\pr{x_1 - u_1}^2 + u_2^2}
                \omega(u) \, du \right]
    \end{align*}
    or
    \begin{align}\label{e:v1Int}
        v_1(x_1, 0)
            = \frac{1}{\pi} \int_{Q_1}
                \pr{f_1(x_1, y) - f_2(x_1, y)} \omega(y) \, dy,
    \end{align}
    where
    \begin{align}\label{e:f1f2}
        f_1(x_1, y)
            = \frac{y_2}{\pr{x_1 - y_1}^2 + y_2^2},
            \quad
        f_2(x_1, y)
            = \frac{y_2}{\pr{x_1 + y_1}^2 + y_2^2}.
    \end{align}

    It follows from $(x_1 - y_1)^2 + y_2^2 \le (x_1 + y_1)^2 + y_2^2$ on
    $Q_1$ that $f_1(x_1, y) > f_2(x_1, y)$ for all $x_1, y_1 > 0$. Conclusion
    (4) then follows from \refE{v1Int}.

    By the Biot-Savart law of \refE{BSLaw},
    \begin{align*}
        v_2(x_1, -x_2)
        &= (K_2*\omega)(x_1, -x_2) \\
        &= \int_{\R^2} K_2(x_1 - y_1, -x_2 - y_2) \omega(y_1, y_2) \, dy \\
        &= \int_{\R^2} K_2(x_1 - y_1, x_2 + y_2) \omega(y_1, -y_2) \, dy  \\
        &= \int_{\R^2} K_2(x_1 - y_1, x_2 - (-y_2)) \omega(y_1, -y_2) \, dy  \\
        &= -v_2( x_1, x_2).
    \end{align*}
    Here we used $K_2(x_1, -x_2) = - K_2(x_1, x_2)$ and the symmetry of $\omega$. A
    similar calculation shows that $v_1( -x_1, x_2) = -v_1( x_1, x_2)$. Thus, the
    velocity along the $x$-axis is directed along the $x$-axis and the velocity
    along the $y$-axis is directed along the $y$-axis, so the axes are preserved by
    the flow. In particular, the origin is fixed. This gives conclusions
    (1)-(3).
\end{proof}

\refL{BCBound} is Proposition 2.1 of \cite{BC1994} (see also Proposition 5.3.1 of \cite{C1998}).

\begin{lemma}\label{L:BCBound}
    Let $\omega$ be SBQ with
    \begin{align}\label{e:omegaOnASquare1}
        \omega = 2 \pi \CharFunc_{[0, 1] \times [0, 1]}
    \end{align}
    in $Q_1$. Then there exists a constant $C > 0$ such that
    \begin{align}\label{e:v1SquareBound1}
        v_1(x_1, 0) \ge 2 x_1 \log (1/x_1)
    \end{align}
    for all $x_1$ in $(0, C]$.
\end{lemma}

The following lemma is a slight generalization of \refL{BCBound} that will give us our key inequality.

\begin{lemma}\label{L:BC21Generalized}
    Let $\omega$ be SBQ with
    \begin{align}\label{e:omegaOnASquare}
        \omega = 2 \pi \CharFunc_{[0, r] \times [0, r]}
    \end{align}
    on $Q_1$ for some $r$ in $(0, 1)$. Then for any $\la$ in $(0, 1)$ there exists a right neighborhood
    of the origin, $\Cal{N}$, such that
    \begin{align}\label{e:v1SquareBound}
        v_1(x_1, 0) \ge 2(1 - \la) x_1 \log (1/x_1)
    \end{align}
    for all $x_1$ in $(0, r^{1/\la}] \cap \Cal{N}$.
\end{lemma}
\begin{remark}
    The neighborhood $\Cal{N}$ depends only upon $\la$; in particular, it
    is independent of $r$.
\end{remark}
\begin{proof}
The result follows from scaling the result in \refL{BCBound}. Indeed, if we write $\omega^r(x)$ for the function $\omega$ defined by \refE{omegaOnASquare} then $\omega^1$ is the function defined by \refE{omegaOnASquare1} and
$
	\omega^r(\cdot) = \omega^1(\cdot/r).
$
Letting $v^r = K*\omega^r$ we see that
\Ignore{ 
\begin{align*}
	v^r(r x)
		&= \int_{\R^2} K(r x - y) \omega^1(y/r) \, dy
		= r^2 \int_{\R^2} K(r x - rz) \omega^1(z) \, dz \\
		&= r^2 \int_{\R^2} \frac{1}{r} K(x - z) \omega^1(z) \, dz 
		= r v^1(x).
\end{align*}
} 
$v^r(x) = r v^1(x/r)$, since then $\omega(v^r(x)) = r (1/r) \omega(v^1) (x/r) = \omega^1(x/r) = \omega^r(x)$ and $v^r(x)$ is divergence-free. It follows from \refL{BCBound} that for all $x_1$ such that $x_1/r$ lies in $[0, C]$,
\begin{align*}
	v^r_1(x_1, 0)
		&= r v^1_1(x_1/r, 0)
		> r 2 (x_1/r) \log(1/(x_1/r)) \\
		&= 2 x_1 \brac{\log (1/x_1) + \log r}.
\end{align*}
Thus, if $x_1^\la \le r$ then $\log r \ge \la \log x_1 = - \la \log (1/x_1)$ so
\begin{align*}
	v^r_1(x_1, 0)
		&> 2 x_1 (1 - \la) \log (1/x_1).
\end{align*}
Thus, \refE{v1SquareBound} holds for all $x_1$ in $[0, r^{1/\la}] \cap [0, rC]$. But $r^{1/\la} \le rC$ if and only if $r \le C^{\la/(1 - \la)}$, which gives us the right neighborhood, $\Cal{N} = (0, C^{1/(1 - \la)})$.
\end{proof}

\DetailSome { 
The following is a direct proof of \refL{BC21Generalized}, as I did it in my thesis, before I thought about rescaling the Bahouri and Chemin result.
\begin{proof}
    \refE{v1Int} gives
    \begin{align*}
        &v_1(x_1, 0)
            = 2 \pi \frac{1}{\pi} \int_{[0, r] \times [0, r]}
                \pr{f_1(x_1, y) - f_2(x_1, y)} \, dy \\
        &\qquad= 2 \int_0^r \int_0^r
                \frac{y_2}{\pr{x_1 - y_1}^2 + y_2^2} \, dy_2 \, dy_1
            - 2 \int_0^r \int_0^r
                \frac{y_2}{\pr{x_1 + y_1}^2 + y_2^2} \, dy_2 \, dy_1.
    \end{align*}
    Both of the inner integrals above are of the form
    \begin{align*}
        \int_0^r \frac{y_2}{a^2 + y_2^2} \, dy_2
        &= \frac{1}{2} \int_0^{r^2} \frac{du}{a^2 + u} \, du
            = \frac{1}{2} \brac{\log (a^2 + u)}_0^{r^2} \\
        &= \frac{1}{2} \log(a^2 + r^2) - \frac{1}{2} \log(a^2),
    \end{align*}
    where $a$ is $x_1 - y_1$ and $x_1 + y_1$ in the two integrals. This
    gives
    \begin{align}\label{e:v}
        \begin{split}
            v_1(x_1, 0)
            &= 2 \int_0^r
                    \left[
                        \frac{1}{2} \log((x_1 - y_1)^2 + r^2) - \frac{1}{2}\log((x_1 - y_1)^2)
                        \right.
                        \\
                &\qquad\qquad\left.- \frac{1}{2} \log((x_1 + y_1)^2 + r^2)
                            + \frac{1}{2}\log((x_1 + y_1)^2)
                        \right]
                    \, dy_1 \\
            &= \widetilde{v}_1(x_1, 0) + \ol{v}_1(x_1, 0),
        \end{split}
    \end{align}
    where
    \begin{align}\label{e:vv}
        \begin{split}
            \widetilde{v}_1(x_1, 0)
            &= \int_0^r \log((x_1 + y_1)^2) \, d y_1
                - \int_0^r \log((x_1 - y_1)^2) \, d y_1, \\
                \ol{v}_1(x_1, 0)
            &= \psi_r(x_1)
                = \int_0^r \log\frac{(x_1 - y_1)^2 + r^2}{(x_1 + y_1)^2 + r^2}
                    \, d y_1.
        \end{split}
    \end{align}

    But,
    \begin{align*}
        \int_0^r \log((x_1 - y_1)^2) \, dy_1
            = 2(r - x_1) \log \abs{r - x_1} + 2 x_1 \log x_1 - 2r,
    \end{align*}
    which follows by adding the integral from 0 to
    $x_1$ to the integral from $x_1$ to $r$. Also,
    \begin{align*}
        \int_0^r \log((x_1 + y_1)^2) \, dy_1
            = 2(r + x_1) \log(r + x_1) - 2 x_1 \log x_1 - 2r,
    \end{align*}
    so
    \begin{align*}
        \widetilde{v}_1(x_1, 0)
        &= -4 x_1 \log x_1
                + 2(r + x_1) \log(r + x_1) - 2(r - x_1) \log \abs{r - x_1} \\
        &= 4 x_1 \log (1/x_1) + \phi_r(x_1),
    \end{align*}
    where
    \begin{align}\label{e:phiTheFlow}
        \phi_r(x_1)
            = 2(r + x_1) \log(r + x_1) - 2(r - x_1) \log \abs{r - x_1}.
    \end{align}

    (So far, our analysis is much as in the proof of Proposition Proposition 2.1 of
    \cite{BC1994}, the significant difference being that for Bahouri and Chemin $r
    = 1$. This allows them to conclude that $v_1(x_1, 0) \ge -2 x_1 \log x_1$ in
    some right-neighborhood of the origin, but we will find that for an arbitrary
    $r$ in $(0, 1)$, the size of this neighborhood shrinks with $r$. This will be
    of importance in \refL{GeneralOmegaStatic} where we consider an initial
    vorticity that is a sum of vorticities like those in \refL{BC21Generalized}.
    We seek both to refine Bahouri's and Chemin's bound slightly and to establish a
    lower bound on the size of the right-neighborhood.)

    By \refE{v}, \refE{vv}, and \refE{phiTheFlow},
    \begin{align}\label{e:vAsSum}
        v_1(x_1, 0)
            = - 4 x_1 \log x_1 + \phi_r(x_1) + \psi_r(x_1).
    \end{align}
    Suppose we can show for $x_1$ lying in $(0, 1)$ but also bounded by
    some function of $r$, that
    \begin{align}\label{e:phipsiBound}
        \phi_r(x_1) + \psi_r(x_1) \ge -C x_1 \log (1/x_1),
    \end{align}
    where $0 < C < 4$. This will insure that
    \begin{align}\label{e:vBound}
        v_1(x_1, 0) \ge C_1 x_1 \log (1/x_1)
    \end{align}
    for the given range of $x_1$, where $C_1 = 4 - C$ is a positive constant.

    We will first show that $\psi_r(x_1)$ is negligible compared to the
    other two terms in \refE{vAsSum}. Then we will show that $\phi_r(x_1)$,
    while comparable in magnitude to $- 4 x_1 \log x_1$, can be made
    sufficiently smaller than it so that \refE{phipsiBound} will hold, as
    long as we restrict $x_1$ so that $x_1 < r^\la$ for some $\la$ in $(0,
    1)$.

    We have,
    \begin{align*}
        \pdx{}{r} \psi_r(x_1)
        &= -2 \tan^{-1}\pr{\frac{x_1 - r}{r}}
            -2 \tan^{-1}\pr{\frac{x_1 + r}{r}}
            + 4 \tan^{-1}\pr{\frac{x_1}{r}} \\
        &\qquad
                + \log \pr{\frac{x_1^2 + 2 r^2 - 2 x_1 r}
                            {x_1^2 + 2 r^2 + 2 x_1 r}} \\
        &= -2 \tan^{-1}\pr{w - 1}
            -2 \tan^{-1}\pr{w + 1}
            + 4 \tan^{-1}\pr{w} \\
        &\qquad
                + \log \pr{\frac{w^2 + 2  - 2 w}
                            {w^2 + 2 + 2 w}} \\
        & = : h(w),
    \end{align*}
    where $w = x_1/r < 1$, since $ 0 \le x_1 < r$. Also,
    \begin{align*}
        h'(w)
            = 4 \frac
                {w^2(w^2 - 4)}
                {(w^2 + 2 - 2w)(w^2 + 2 + 2w)(1 + w^2)},
    \end{align*}
    which is is negative for $w$ in $(0, 2)$, is zero at $w = 2$, and is
    positive for $w$ in $(2, \iny)$. This means that $h(w)$ decreases from
    a value of $0$ at $0$ to a negative value at $w = 2$, then
    monotonically increases. However, it is also true that $\lim_{w \to
    \iny} h(w) = 0$, which means that $h(w) < 0$ for all $w > 0$.

    What we have shown is that that for all $x_1, r > 0$,
    \begin{align*}
        \pdx{}{r} \psi_r(x_1)
            = h(w) < 0,
    \end{align*}
    so for a fixed value of $x_1$, $\psi_r(x_1)$ is a decreasing function
    of $r$. Restricting $r$ to lie in the range $(0, 1]$, we can therefore
    bound $\psi_r(x_1)$ from below by $\psi_1(x_1)$. But $\psi_1(x_1)$ is a
    smooth function equal to zero at the origin, so
    \begin{align*}
        \psi_r(x_1)
            \ge \psi_1(x_1)
            \ge -C_3 x_1
    \end{align*}
    for some positive constant $C_3$ in a fixed neighborhood of the the
    origin that is independent of the value of $r$. This means that by
    restricting $x_1$ to a possibly smaller neighborhood of the origin, we
    can insure that the contribution of $\psi_r(x_1)$ to the left-hand side
    of \refE{phipsiBound} is as small as required. Hence, we can replace
    \refE{phipsiBound} by the requirement that
    \begin{align}\label{e:phiBound}
        \phi_r(x_1) \ge -C x_1 \log (1/x_1)
    \end{align}
    for all $x_1 < r$, where $0 < C < 4$, and from this \refE{vBound} will follow.

    We have,
    \begin{align*}
    &\frac{\phi_{r}(x_1)}{2}
            = (r + x_1) \log(r + x_1) - (r - x_1) \log(r - x_1) \\
        &\qquad = r(1 + w) \log(r(1 + w)) - r(1 - w) \log(r(1 - w)) \\
        &\qquad= r(1 + w) \log r - r(1 - w) \log r
                + r(1 + w) \log(1 + w)  \\
        &\qquad\qquad - r(1 - w) \log( 1 - w) \\
        &\qquad= r \log r + rw \log r - r \log r + rw \log r \\
        &\qquad\qquad + r \brac{\log(1 + w) - \log(1 - w)}
                + rw \brac{\log(1 + w) + \log(1 - w)} \\
        &\qquad= 2rw \log r + r \log \pr{\frac{1+w}{1 - w}}
                + rw \log(1 - w^2).
    \end{align*}

    Since
    \begin{align*}
        r \log &\pr{\frac{1+w}{1 - w}} + rw \log(1 - w^2) \\
        &\ge r \log \pr{\frac{1+w}{1 - w}} + r \log(1 - w^2)
            = r \log \pr{\frac{1+w}{1 - w} \, (1 + w)(1 - w)} \\
        &= 2 r \log (1 + w) \ge 0,
    \end{align*}
    it follows that
    \begin{align*}
        \phi_{r}(x_1)
            &\ge 4rw \log r
            = 4x_1 \log r.
    \end{align*}

    We now add the restriction that
    \begin{align*}
        x_1^\la \le r
    \end{align*}
    for some $\la$ in $(0, 1)$. Then
    \begin{align*}
        \phi_{r}(x_1)
            \ge 4 x_1 \log r
            \ge 4 x_1 \log x_1^\la
            = -4 \la x_1 \log (1/x_1),
    \end{align*}
    which is \refE{phiBound} with $C = 4 \la$. It follows from \refE{vBound} that
    \begin{align}\label{e:v01BoundFrom2Symmetries}
        v_1(x_1, 0) \ge (4 - 4 \la)x_1 \log (1/x_1) - C_3 x_1.
    \end{align}
    \refE{v1SquareBound} then follows, since
    choosing $\la' < \la$ allows us to ignore the term $-C_3
    x_1$.
\end{proof}
} 

\begin{definition}\label{D:SquareSymmetric}
    We say that $\omega$ is square-symmetric if $\omega$ is SBQ
    and $\omega(x_1, x_2) = \omega(\max\smallset{x_1, x_2}, 0)$
    on $Q_1$.
\end{definition}

Being square-symmetric means that a vorticity is SBQ and is constant in
absolute value along the boundary of any square centered at the origin.

\begin{lemma}\label{L:GeneralOmegaStatic}
    Assume that $\omega$ is square-symmetric, finite except possibly at the
    origin, and $\omega(x_1, 0)$ is non-increasing for $x_1 > 0$.
    Then for any $\la$ in $(0, 1)$
    \begin{align}\label{e:vFinalBound}
        v_1(x_1, 0)
            \ge C_\la \omega(x_1^\la, 0) x_1 \log (1/x_1)
    \end{align}
    for all $x_1$ in the neighborhood $\Cal{N}$ of \refL{BC21Generalized}, where $C_\la = (1 - \la)/\pi$.
\end{lemma}
\begin{proof}
    We can write $\omega$ on $Q_1$ as
    \begin{align}\label{e:omega0Integral}
        \omega(x) = 2 \pi \int_0^1 \al(s)
            \CharFunc_{[0, s] \times [0, s]}(x) \, ds,
    \end{align}
    for some measurable, nonnegative function $\al \colon (0, 1) \to [0, \iny)$.
    This means that
    \begin{align}\label{e:omega0x1}
        \omega(x_1, 0) = 2 \pi \int_{x_1}^1 \al(s) \, ds.
    \end{align}

    Let $V(s)$ be the value of $v_1(x_1, 0)$ that results from assuming that
    $\omega$ is given by \refE{omegaOnASquare}. By \refL{SBQ}, $V(s) > 0$.
    Then because the Biot-Savart law of \refE{BSLaw} is linear, and using
    \refL{BC21Generalized}, for all $x_1$ in the neighborhood $\Cal{N}$,
    \begin{align*}
        v_1(x_1, 0)
           &= \int_0^1 \al(s) V(s) \, ds \\
           &= \int_0^{x_1^\la} \al(s) V(s) \, ds
             + \int_{x_1^\la}^1 \al(s) V(s) \, ds \\
           &\ge \int_{x_1^\la}^1 \al(s) V(s) \, ds \\
           &\ge 2 \pi
                \pr{\int_{x_1^\la}^1 \al(s) \, ds}
                    \frac{2 (1 - \la)}{2 \pi} x_1 \log (1/x_1) \\
           &= \frac{1 - \la}{\pi} \omega(x_1^\la, 0) x_1 \log (1/x_1).
    \end{align*}
In the final inequality, $V(s)$ is bounded as in \refL{BC21Generalized} because $x_1^\la \le s$ in the integrand.
\end{proof}

\begin{remark}\label{R:omeg0Smooth}
    Properly speaking, we must allow the function $\al$ of \refE{omega0Integral} to
    be a distribution since, for instance, to obtain $\omega$ of
    \refL{BC21Generalized}, we would need $\al = \delta_r$. We could avoid this
    complication, however, by assuming that $\omega$ is strictly decreasing and
    that $\omega(x_1, 0)$ is sufficiently smooth as a function of $x_1 > 0$.
\end{remark}

%
%
\section{Evolution of square-symmetric vorticity}\label{S:ExampleOverTime}

\noindent We now assume that our initial vorticity is square-symmetric, and
consider what happens to the solution to ($E$) over time.

\begin{theorem}\label{T:UpperBoundAllTime}
    Assume that $\omega^0$ is square-symmetric, finite except possibly at the
    origin, and $\omega^0(x_1, 0)$ is nonnegative and non-increasing for $x_1 > 0$.
    Then for any $\la$ in $(0, 1)$,
    \begin{align}\label{e:vFinalBoundOverTime}
        v_1(t, x_1, 0)
            \ge C_\al \omega^0(\Gamma_t(2^{\la/2} x_1^\la), 0) x_1 \log(1/x_1)
            := L(t, x_1)
    \end{align}
    for all $x_1$ in the neighborhood $\Cal{N}$ of
    \refL{BC21Generalized} and all time $t \ge 0$, where $\Gamma_t$ is defined as in
    \refT{YudovichTheorem}. The constant $C_\la = (1 - \la)/\pi$.

    Further, let $x_1(t)$ be the minimal solution to
    \begin{align}\label{e:dx1t}
        \diff{x_1(t)}{t}
            = L(t, x_1), \quad x_1(0) = a 
    \end{align}
    with $a > 0$ in $\Cal{N}$, and $(0, t_a)$ being the time of existence.
    Then $\psi^1(t, a, 0) \ge x_1(t)$ for all $t$ in $[0, t_a)$.
\end{theorem}

\begin{remark}\label{R:RemUpperBoundAllTime}
In our applications of \refT{UpperBoundAllTime} in the next two sections, $L$ will be Osgood continuous in space, so that a unique (and explicit) solution to \refE{dx1t} exists for all time. Hence, there will be no need to determine a minimal solution and we will have $t_a = \iny$.
\end{remark}

\begin{proof}[Proof of \refT{UpperBoundAllTime}]
    Since $\omega^0(x_1, x_2) = -\omega^0(x_1, -x_2)$, if $\omega(t, x_1, x_2)$ is
    a solution to ($E$) then $-\omega(t, x_1, -x_2)$ is also a solution. But the
    solution to ($E$) is unique by \refT{YudovichTheorem},
    so we conclude that $\omega(t, x_1, x_2) =
    -\omega(t, x_1, -x_2)$. Similarly, $\omega(t, x_1, x_2) = -\omega(t, -x_1,
    x_2)$, and we see that $\omega$ is SBQ. By \refL{SBQ}, then, it follows that
    the flow transports vorticity in $Q_k$, $k = 1, \dots, 4$, only within $Q_k$,
    because of the direction of $v$ along the axes for all $t \ge 0$. Therefore,
    $\omega(t)$ is also nonnegative in $Q_1$ for all time.

    Our approach then will be to produce a point-by-point lower bound
    $\ol{\omega}(t)$ on $\omega(t)$ that satisfies all the requirements of
    \refL{GeneralOmegaStatic}. In particular, it is SBQ, so $\omega(t) -
    \ol{\omega}(t)$ is SBQ and nonnegative in $Q_1$. It follows from \refL{SBQ}
    that $v_1(t, x_1, 0) - \ol{v}_1(t, x_1, 0) \ge 0$ for all $t \ge 0$,
    where $\omega(\ol{v}(t)) =
    \ol{\omega}(t)$. Thus, the lower bound on $\ol{v}_1(t, x_1, 0)$ coming from
    \refL{GeneralOmegaStatic} will also be a lower bound on $v_1(t, x_1, 0)$.
    We now determine $\ol{\omega}(t)$.

    Because conclusion (3) of \refL{SBQ} holds for all time, $\omega$ being SBQ for
    all time, the origin is fixed by the flow; that is $\psi(t, 0) = \psi^{-1}(t,
    0) = 0$ for all $t$. Also, the Euler equations are time reversible, and the
    function $\Gamma_t$ of \refE{MOCFlow} depends only upon the Lebesgue norms of
    the vorticity, which are preserved by the flow; therefore, $\Gamma_t$ is a
    bound on the modulus of continuity of $\psi^{-1}(t, \cdot)$ as well. Thus,
    \begin{align*}
        \abs{\psi^{-1}(t, x)}
            = \abs{\psi^{-1}(t, x) - \psi^{-1}(t, 0)}
            \le \Gamma_t(\abs{x}).
    \end{align*}

    In $Q_1$, the value of $\omega(t, x)$, then, is bounded below by using the
    minimum value of $\omega^0$ on the circle of radius $\Gamma_t(\abs{x})$
    centered at the origin, since this is the furthest away from the origin that
    $\psi^{-1}(t, x)$ can lie, and $\omega^0$ decreases with the distance from the
    origin. That is,
    \begin{align*}
        \omega(t, x)
           &= \omega^0(\psi^{-1}(t, x))
            \ge \omega^0(\Gamma_t(\abs{x}), 0)
    \end{align*}
    because $\omega^0$ is square-symmetric.

    Since $\sqrt{2} \max\smallset{x_1, x_2} \ge \abs{x}$, $\Gamma_t$ is
    nondecreasing, and $\omega_0$ is nonincreasing on $Q_1$,
    $\omega^0(\Gamma_t(\sqrt{2} \max\smallset{x_1, x_2}), 0)
    \le \omega^0(\Gamma_t(\abs{x}), 0)$ on $Q_1$. Letting
    \begin{align*}
        \ol{\omega}(t, x_1, x_2)
            = \omega^0(\Gamma_t(\sqrt{2} \max\smallset{x_1, x_2}), 0)
    \end{align*}
    we see that $\ol{\omega}$ is square-symmetric, and on $Q_1$,
    $\ol{\omega}(t, x) \le \omega(t, x)$,
    so $\ol{\omega}$ is our desired lower bound on $\omega$.

    Then from \refE{vFinalBound},
    \begin{align*}
        v^1(x_1, 0)
           &\ge C_\la \ol{\omega}(t, x_1^\la, 0) x_1 \log (1/x_1)
                \\
           &= C_\la  \omega^0(\Gamma_t((\sqrt{2} \max\smallset{x_1, 0})^\la), 0)
                     x_1 \log (1/x_1) \\
           &= C_\la   \omega^0(\Gamma_t(2^{\la/2} x_1^\la), 0) x_1 \log
                    (1/x_1).	
    \end{align*}
    
    That the minimal solution exists to \refE{dx1t} on $[0, t_a)$ for some
    $t_a > 0$ and the inequality, $\psi^1(t, a, 0) \ge x_1(t)$ for all $t$ in $[0, t_a)$, are classical
    results;
    see, for instance, Theorems 2.1 and 4.2 Chapter III of \cite{Hartman2002}.
\end{proof}

\section{Bounded vorticity}\label{S:BoundedVorticity}

\noindent We now apply \refT{UpperBoundAllTime} to the first in the sequence of
Yudovich's vorticity bounds in \refE{YudovichExamples} in which we have bounded
vorticity. We assume that $\omega$ is square-symmetric with $\omega^0 =
\CharFunc_{[0, 1/2] \times [0, 1/2]}$ in $Q_1$ so that
$\smallnorm{\omega^0}_{L^1 \cap L^\iny} = 1$. We have,
\begin{align*}
        \mu(r)
        &= \inf \set{r^{1 - 2 \eps}/\eps: \eps \text{ in } (0, 1/2]}
            = \inf \set{g(\eps): \eps \text{ in } (0, 1/2]},
\end{align*}
where $g(\eps) = r^{1-2 \eps}/\eps$. Then
\begin{align*}
    g'(\eps)
        = C \frac{r^{1 - 2 \eps}(2 \eps \log(1/r) - 1)}{\eps^2},
\end{align*}
which is zero when $\eps = \eps_0 := 1/(2 \log(1/r))$ if $r < 1$ and $\eps_0 <
1/2$, and never zero otherwise. But
\begin{align*}
    \eps_0 < 1/2
    &\iff \frac{1}{2 \log(1/r)} < 1/2
        \iff \log(1/r) > 1 \\
    &\iff \frac{1}{r} > e
        \iff r < e^{-1},
\end{align*}
so the condition $r < 1$ is redundant.

Assume that $r < e^{-1}$. Then $g(\eps)$ approaches infinity as $\eps$
approaches either zero or infinity; hence, $\eps_0$ minimizes $g$. Thus,
\begin{align*}
    \mu(r)
    &=r^{1-2 \eps_0}/\eps_0
        = 2 r (1/r)^{2 \eps_0} \log(1/r) \\
    &= 2 r e^{2 \log(1/r)\eps_0} \log(1/r)
        = -2 e r \log(r).
\end{align*}
Then from \refE{MOCFlow},
\begin{align*}
    \int_{x_1}^{\Gamma_t(x_1)} &\frac{dr}{\mu(r)}
        = -(2 e)^{-1} \brac{\log (-\log r)}_{x_1}^{\Gamma_t(x_1)}
        = t \\
    &\implies \log (-\log (x_1) - \log (-\log (\Gamma_t(x_1))) = 2 e t \\
    &\implies \Gamma_t(x_1) =  x_1 {e^{-2et}}
\end{align*}
as long as $\Gamma_t(x_1) < e^{-1}$.

Thus, \refT{UpperBoundAllTime} gives
\begin{align*}
	v^1(t, x_1, 0)
		&\ge C_\la  \omega^0(2^{\la/2} x_1^{\la e^{-2et}}, 0) x_1 \log (1/x_1) \\	
		&\ge C_\la  x_1 \log (1/x_1)
\end{align*}
as long as $2^{\la/2} x_1^{\la e^{-2et}} < {1/2}$.

Solving $d x_1(t)/dt = C_\la  x_1 \log (1/x_1)$ with $x_1(0) = a$ gives
\begin{align*}
    \psi^1(t, a, 0) \ge x_1(t) = a^{\exp(-C_\la  t)},
\end{align*}
which applies for sufficiently small $a$.

Since $\psi(t, 0, 0) = 0$,
\begin{align*}
    \frac{\abs{\psi(t, a, 0) - \psi(t, 0, 0)}}{{a^\al}}
        = \frac{\abs{\psi^1(t, a, 0)}}{{a^\al}}
        \ge a^{\exp(- C_\la  t) - \al},
\end{align*}
which is infinite for any $\al > e^{-C_\la t}$. This shows that the flow
can be in no \Holder space $C_\la ^\al$ for $\al > e^{-C_\la t}$, reproducing, up to a constant, the result of \cite{BC1994} (or see Theorem 5.3.1 of \cite{C1998}.)

\Obsolete{ 
    \skipline \noindent \textbf{$\mathbf{\text{log} \, \text{log} (1/\abs{x})}$
    point singularity:} The second example bound in Yudovich's sequence is
    $\theta(p) = \log p$, which is produced asymptotically when $\omega^0$ has a
    point singularity like $\log \log (1/\abs{x})$. (This is by \refL{lnLpnorm} and
    the comment following it.) Thus, we assume that $\omega$ is square-symmetric
    with $\omega^0(x_1, 0) = \log \log (1/x_1)$ for $x_1 > 0$.

    Using an observation of Yudovich's in \cite{Y1995}, if $\beta_M$ is the
    function of \refD{Admissible} associated with the admissible function
    $\theta_m$ of \refE{YudovichExamples}, then letting $\eps_0 = -1/\log r$ for $r
    < e^{-p_0}$,
    \begin{align*}
        \beta_M(r)
            &\le C (M^{\eps_0} r^{1-\eps_0}/\eps_0) \theta_m(1/\eps_0) \\
            &= -C M^{\eps_0} r r^{1/\log r} \log r \theta_m(\log (1/r)) \\
            &= C M^{\eps_0} r \theta_{m + 1}(1/r)
            \le C r \theta_{m + 1}(1/r),
    \end{align*}
    since $M^{\eps_0} \le \max\smallset{1, M^{1/p_0}}$. (In fact, the reverse
    inequality---with a different constant---can also be proven.)

    For our choice of $\omega^0$, for which $m = 1$,
    \begin{align*}
        \beta_1(r)
            \le C r \theta_2(1/r)
            = C r \log (1/r) \log \log (1/r).
    \end{align*}
    Then, if we define the upper bound on the modulus of continuity of the flow by
    \refE{MOCFlowLoose} instead of \refE{MOCFlow}, we have
    \begin{align*}
        -C &\brac{\log \log (-\log r)}_{s}^{\Gamma_t(s)}
            = \int_{s}^{\Gamma_t(s)}
                \frac{dr}{C r \log (1/r) \log \log (1/r)} \\
        &\qquad\le \int_{s}^{\Gamma_t(s)} \frac{dr}{\beta_1(r)}
            = t \\
        &\implies \log (-\log (\Gamma_t(s))) \ge e^{-Ct} \log (-\log s).
    \end{align*}

    From \refE{vFinalBoundOverTime}, for $x_1 > 0$,
    \begin{align*}
        v^1(t, x_1, 0)
        &\ge - \frac{2(1 - \la')}{\pi} \log (-\log (\Gamma_t(x_1)^\la)) x_1 \log
                x_1 \\
        &= - C \log (-\la \log (\Gamma_t(x_1))) x_1 \log
                x_1 \\
        &= - C \log \la \log x_1 \log x_1 - C \log (\log(- \Gamma_t(x_1))) x_1 \log x_1 \\
        &\ge - C \log (\log(- \Gamma_t(x_1))) x_1 \log x_1 \\
        &\ge - Ce^{-Ct} \log (-\log x_1) x_1 \log x_1
    \end{align*}
    as long as $x_1$ is sufficiently small that the second-to-last inequality
    holds.

    Solving for
    \begin{align*}
        \diff{x_1(t)}{t}
            = - Ce^{-Ct} \log (-\log x_1) x_1 \log x_1
    \end{align*}
    with $x_1(0) = a$, we get
    \begin{align*}
        \log \log (-\log x_1(t))
            = \log \log (-\log a)
            + C \pr{e^{-Ct} - 1},
    \end{align*}
    so
    \begin{align*}
        \psi^1(t, a, 0)
        &\ge x_1(t)
            = \exp\pr{-(-\log a)^{\exp\pr{C (e^{-Ct} - 1)}}} \\
        &= e^{-(-\log a)^\gamma},
    \end{align*}
    where $\gamma = \exp\pr{C (e^{-Ct} - 1)}$.

    Observe that $\gamma < 1$ for all $t > 0$. Thus, for any $\al$ in $(0,
    1)$ and all $t > 0$,
    \begin{align*}
        \norm{\psi - \text{Id}}_{C^\al}
        &\ge \lim_{a \to 0^+} \frac{\psi^1(t, a, 0) - \psi^1(t, 0, 0)}
                    {a^\al}
            \ge \lim_{a \to 0^+} \frac{x_1(t)}{a^\al} \\
        &= \lim_{a \to 0^+} \frac{e^{-(-\log a)^\gamma}}
                    {e^{(-\log a) \al}}
            = \lim_{u \to \iny} \frac{e^{-u^\gamma}}{e^{-\al u}}
            = \lim_{u \to \iny} e^{\al u -u^\gamma}
            = \iny.
    \end{align*}
    We conclude that the flow lies in no \Holder space of positive exponent
    for all positive time, a result that we state explicitly as a corollary
    of \refT{UpperBoundAllTime}.

    \begin{cor}
        There exists initial velocities satisfying the conditions of
        \refT{YudovichTheorem} for which the unique solution to ($E$)
        has an associated flow lying, for all positive time, in no \Holder
        space of positive exponent.
    \end{cor}
    } 

\section{Yudovich's higher examples}\label{S:YudoExamples}

\noindent Assume that $m \ge 2$ and let
$\omega^0$ have the symmetry described in \refT{UpperBoundAllTime} with
\begin{align} \label{e:omega0ForBound}
	\omega^0(x_1, 0)
		= \log^2 (1/x_1) \cdots \log^m (1/x_1)
		= \theta_m(1/x_1)/\log(1/x_1),
\end{align}
for $x_1$ in $(0, 1/\exp^m(0))$, and $\omega^0$ equal to zero elsewhere
in the first quadrant. Then by \refL{lnLpnorm},
\begin{align*}
    \theta(p) = \SmallLpNorm{\omega^0}{p}
    	\le C \log p \cdots \log^{m-1} p
	= \theta_{m - 1}(p)
\end{align*}
for all $p$ larger than some $p^*$, with $\theta_{m - 1}$ given by \refE{YudovichExamples}.

Adapting an observation of Yudovich's in \cite{Y1995}, if $\mu$ is the
function of \refE{mu} associated with the admissible function $\theta$,
then letting $\eps_0 =1/ 2 \log (1/r)$ for $r < e^{-p^*}$,
\begin{align*}
    \mu(r)
        &\le (r^{1-2 \eps_0}/\eps_0) \theta(1/\eps_0)
         = -C r r^{1/\log r} \log r \theta(\log (1/r)) \\
        &= C r \log(1/r) \log^2(1/r) \cdots \log^m(1/r) \\
        &= C r \theta_m(1/r).
\end{align*}

Then, if we define the upper bound on the modulus of continuity of the flow by
\refE{MOCFlow}, we have
\begin{align*}
    -C &\brac{\log^{m + 1}(1/r)}_{s}^{\Gamma_t(s)}
        = \int_{s}^{\Gamma_t(s)}
            \frac{dr}{C r \theta_m(1/r)} \\
       &\qquad\le \int_{s}^{\Gamma_t(s)} \frac{dr}{\mu(r)}
        = t \\
       &\implies -C \log^{m + 1}(1/\Gamma_t(s))
            \le -C \log^{m + 1}(1/s) + t \\
       &\implies \log^{m + 1}(1/\Gamma_t(s))
            \ge \log^{m + 1}(1/s) - C t \\
       &\implies \log^m(1/\Gamma_t(s))
            \ge e^{-Ct} \log^m(1/s).
\end{align*}

Using this bound, we have, from \refE{vFinalBoundOverTime},
\begin{align}\label{e:v1BoundYudo}
    \begin{split}
        v^1(t, &x_1, 0)
            \ge C_\la  \omega^0(\Gamma_t(2^{\la/2} x_1^\la), 0) x_1 \log
                        (1/x_1) \\
            &= C_\la  \log^2(1/\Gamma_t(2^{\la/2} x_1^\la)) \cdots  \log^m(1/\Gamma_t(2^{\la/2} x_1^\la)) x_1 \log
                        (1/x_1) \\
             &\ge C_\la e^{-Ct} \log^2(1/\Gamma_t(2^{\la/2} x_1^\la)) \cdots \log^m(1/2^{\la/2} x_1^\la)
             				x_1 \log (1/x_1) \\
             &\ge C_\la' e^{-Ct} \log^2(1/\Gamma_t(2^{\la/2} x_1^\la)) \cdots \log^m(1/x_1) x_1 \log (1/x_1),
    \end{split}
\end{align}
as long as $x_1 > 0$ is sufficiently small, where $C_\al'$ depends on $\la$. (When $m > 2$, the argument $1/\Gamma_t(2^{\la/2} x_1^\la)$ appears in each of the $\log^2, \dots, \log^{m - 1}$ factors above, but not in the $\log$ factor. When $m = 2$ it appears in none of the factors.)

Specializing to the case $m = 2$ and combining the previous two inequalities, the explicit dependence of the bound in \refE{v1BoundYudo} on $\Gamma_t$ disappears, and we have
\begin{align*}
    v^1(t, x_1, 0)
       &\ge C_\la' e^{-Ct} \log^2(1/x_1) x_1 \log (1/x_1)
        = C_\la' e^{-Ct} x_1 \theta_2(1/x_1).
\end{align*}
Solving for
\begin{align}\label{e:ODEm2}
    \diff{x_1(t)}{t}
        = C_\la' e^{-Ct} x_1 \theta_2(1/x_1)
\end{align}
with $x_1(0) = a$, we get
\begin{align*}
    \log^3(1/x_1(t))
        = \log^3(1/a)
          + C_\la' \pr{e^{-Ct} - 1},
\end{align*}
so
\begin{align*}
    \psi^1(t, a, 0)
       &\ge x_1(t)
        = \exp\pr{-(-\log a)^{\exp\pr{C_\la' (e^{-Ct} - 1)}}} \\
       &= e^{-(-\log a)^\gamma},
\end{align*}
where $\gamma = \exp\pr{C_\la' (e^{-Ct} - 1)}$.

Observe that $\gamma < 1$ for all $t > 0$. Thus, for any $\al$ in $(0, 1)$ and
all $t > 0$,
\begin{align}\label{e:psiCalpha}
    \begin{split}
        \norm{\psi}_{C^\al}
        &\ge \lim_{a \to 0^+} \frac{\psi^1(t, a, 0) - \psi^1(t, 0, 0)}
                    {a^\al}
            \ge \lim_{a \to 0^+} \frac{x_1(t)}{a^\al} \\
        &= \lim_{a \to 0^+} \frac{e^{-(-\log a)^\gamma}}
                    {e^{-(-\log a) \al}}
            = \lim_{u \to \iny} \frac{e^{-u^\gamma}}{e^{-\al u}}
            = \lim_{u \to \iny} e^{\al u -u^\gamma}
            = \iny.
    \end{split}
\end{align}
We conclude that the flow lies in no \Holder space of positive exponent for all
positive time, a result that we state explicitly as a corollary of \refT{UpperBoundAllTime}.

\begin{cor}\label{C:MainResultPartI}
    There exists initial velocities satisfying the conditions of
    \refT{YudovichTheorem} for which the unique solution to ($E$)
    has an associated flow lying, for all positive time, in no \Holder
    space of positive exponent.
\end{cor}

\Ignore{ 

We used the following lemma above:

\begin{lemma}\label{L:logmInequality}
    Let $m$ be a positive integer. Then for sufficiently small positive
    $x$,
    \begin{align*}
        \log^m(1/x) \ge (1/2) \log^m(1/x^2).
    \end{align*}
\end{lemma}
\begin{proof}
    The proof is by induction. For $x < 1$,
    \begin{align*}
        \log(1/x^2)
            = 2 \log(1/x),
    \end{align*}
    which establishes the inequality (in fact, equality) for $m = 1$.
    So assume the inequality holds for all positive integers up to $m$. Then
    \begin{align*}
        \log^{m + 1}(1/x^2)
           &= \log \log^{m - 1} (1/x^2)
            \le \log(2 \log^{m - 1}(1/x)) \\
           &= \log 2 + \log^m(1/x)
            \le 2 \log^m(1/x)
    \end{align*}
    as long as $x$ is sufficiently small that $\log 2 \le \log^m(1/x)$.
\end{proof}
} 

We used  \refL{lnLpnorm} above, and \refL{loglogpIneq} is used in its proof.

\begin{lemma}\label{L:lnLpnorm}
	Let $m \ge 2$ and let $\omega^0$ have the symmetry described in \refT{UpperBoundAllTime} with
    \begin{align*}
        \omega^0(x_1, 0)
            = \log^2 (1/x_1) \cdots \log^m (1/x_1)
            = \theta_m(1/x_1)/\log(1/x_1),
    \end{align*}
    for $x_1$ in $(0, 1/\exp^m(0))$, and $\omega^0$ equal to zero elsewhere
    in the first quadrant. Then
    \begin{align*}
        \SmallLpNorm{\omega^0}{p}
            \sim \log p \cdots \log^{m-1} p
            = \theta_{m - 1}(p)
    \end{align*}
    for large $p$.
\end{lemma}
\begin{proof}
    Because of the symmetry of $\omega^0$,
    \begin{align}\label{e:OrigIntegration}
        \begin{split}
            \SmallLpNorm{\omega^0}{p}^p
            &= 4 \int_0^{1/\exp^m(0)}
                    2 \int_0^{x_1} (\omega^0(x_1, 0))^p \, dx_2 \, dx_1 \\
            &= 8 \int_0^{1/\exp^m(0)}
                        x_1 \brac{\log^2 (1/x_1) \cdots \log^m (1/x_1)}^p \, dx_1.
        \end{split}
    \end{align}
    Making the change of variables, $u = \log(1/x_1) = -\log x_1$, it
    follows that $x_1 = e^{-u}$ and $du = -(1/x_1) \, dx_1$ so $dx_1 =
    -e^{-u} \, du$. Thus,
    \begin{align*}
        \SmallLpNorm{\omega^0}{p}^p
           &= 8 \int_\iny^{\exp^{m - 1}(0)} e^{-u} \brac{\log u \cdots \log^{m-1} u}^p
                (-e^{-u}) \, du \\
           &= 8 \int_{\exp^{m - 1}(0)}^\iny e^{-2u} \brac{\log u \cdots \log^{m-1} u}^p
                \, du \\
           &\ge 8 \int_p^\iny e^{-2u} \brac{\log p \cdots \log^{m-1} p}^p
                \, du \\
	&= 4 e^{-2p}  \brac{\log p \cdots \log^{m-1} p}^p,
    \end{align*}
    the inequality holding for all sufficiently large $p$.
    Asymptotically, then, $\SmallLpNorm{\omega^0}{p} \ge e^{-2} \log p \cdots \log^{m-1} p$.
    
    For the upper upper bound on $\SmallLpNorm{\omega^0}{p}$, we use \refL{loglogpIneq}
    to obtain, for all sufficiently large $p$,
    \begin{align*}
        \SmallLpNorm{\omega^0}{p}^p
           &\le 8 \pr{\int_0^p  + \int_p^\iny}
                e^{-2u} \brac{\log u \cdots \log^{m-1} u}^p \, du \\
           &\le 8 \int_0^p
                e^{-2u} \brac{\log p \cdots \log^{m-1} p}^p \, du \\
              &\qquad\qquad+ 8 \int_p^\iny
                    e^{-2u} \brac{\log p \cdots \log^{m-1} p \, e^{\frac{u}{p} - 1}}^p \, du \\
           &\le 8 \pr{\log p \cdots \log^{m-1} p}^p
                \brac{\int_0^p e^{-2u} \, du
                    + e^{-p}\int_p^\iny e^{-u} \, du} \\
           &= 8 \pr{\frac{1 + e^{-2p}}{2}} \pr{\log p \cdots \log^{m-1} p}^p.
    \end{align*}
    Thus, asymptotically, $\SmallLpNorm{\omega^0}{p}
    \le \log p \cdots \log^{m-1} p$, completing the proof.
\end{proof}

\Ignore{ 
\begin{remark}
Essentially the same proof gives the same result for radial or other symmetries
(except that $e^{-2}$ might become a different constant), as well as for any
spatial dimension. In dimension $n$ with radial symmetry, for example, the
factor of $x_1$ in the original integration in \refE{OrigIntegration} would
become $r^n$ and the factor of $8$ would become the surface area of the unit
sphere. The $r^n$ factor would become $e^{-nxp}$ (even for $n = 1$) in
\refE{LpNorm}, which does not materially affect the rest of the argument.
\end{remark}
} 

\begin{lemma}\label{L:loglogpIneq}
    Let $m$ be a positive integer. Then for sufficiently large $p$,
    \begin{align}\label{e:LogpInequality}
        \log (xp) \cdots \log^{m-1} (xp)
            \le (\log p \cdots \log^{m - 1} p) e^{x - 1}
    \end{align}
    for all $x \ge 1$.
\end{lemma}
\begin{proof}
    We prove this for $m = 3$, the proof for other values of $m$ being
    entirely analogous. First, by taking the logarithm of both sides of
    \refE{LogpInequality}, that equation holds if and only if
    \begin{align*}
        f(x) &:= \log \log (xp) + \log \log \log (xp) \\
            &\le g(x) := \log(\log p \log \log p) + x - 1.
    \end{align*}
    Because equality holds for $x = 1$, our result will follow if we
    can show that $f' \le g'$ for all $x \ge 1$ and sufficiently large
    $p$. This is, in fact, true, since
    \begin{align*}
        f'
        		= \frac{1}{x \log(xp)}
        		+ \frac{1}{x \log(xp) \log \log (xp)} \le 1 = g'
    \end{align*}
    for all $x \ge 1$ and $p \ge e^e$.
\end{proof}

\section{Remarks}\label{S:RemarksPartI}

\noindent It is natural to try to extend the analysis of \refS{YudoExamples} to Yudovich initial vorticities for $m > 2$. But this is, in fact, quite difficult, because when $m > 2$ the explicit dependence of the bound in \refE{v1BoundYudo} on $\Gamma_t$ remains, so we must also bound $\log^k(1/\Gamma_t(s))$ for
$k = 2, \dots, m - 1$. Doing so makes the analog of
\refE{ODEm2} no longer exactly integrable. It is clear that one obtains a worse bound on the modulus of continuity (MOC) than for $m = 2$, but it is not at all clear what happens as we take $m$ to infinity.

\Ignore{ 
Another approach to is to assume that $\omega^0$ is SBQ, but is constant along
vertical line segments (in each quadrant), though supported in a square. This
would still allow the inner integral in the first expression for $v_1(x_1, 0)$\MarginNote{Remember that you no longer give a direct proof of this.}in the proof of \refL{BC21Generalized} to be calculated exactly. One would then
need to choose the dependence of $\omega^0$ on the $x_1$ coordinate carefully
to produce the equivalent of \refE{vFinalBoundOverTime}. It is not at all clear
that this approach has any advantages, however.
} 

We chose to give the initial vorticity $SBQ$ symmetry because such symmetry works well with the symmetry of the Biot-Savart law to produce a lower bound on the velocity along the $x$- or $y$-axes. Having made this choice, the rest of our choices concerning the vorticity were inevitable, up to interchanging the roles of $x$ and $y$ or changing the sign of the vorticity. Because the initial vorticity is $SBQ$, the function $f = f_1 -
f_2$, where $f_1$ and $f_2$ are defined in \refE{f1f2}, controls the bound on the velocity, and
$f$ is continuous except for a singularity at $y = (x_1, 0)$, where
it goes to positive infinity (for $x_1 > 0$). Thus, whatever lower bound we derive
on $v_1(x_1, 0)$, it will increase the fastest at a singularity of
$\smallabs{\omega(t)}$ that lies along the $x$-axis and this effect is most pronounced when $\omega$ is of a single sign in $Q1$ (this follows from
\refE{v1Int}). The lower bound on the MOC of the flow
then follows from allowing a point $a = (a_1, 0)$ to approach the singularity
and looking at how large the appropriate difference quotient becomes, as in
\refE{psiCalpha}. But to do this, we need control on the position of the
singularity of $\smallabs{\omega(t)}$, and, when assuming $SBQ$, the origin is
the one point at which we have the most control---the singularity doesn't move
at all.

Thus, the assumption of $SBQ$ naturally leads us to assume a point singularity
at the origin. Then, because it appears that we can only bound from below the
effect on $v_1(x_1, 0)$ of the vorticity outside of the square on which a point
lies (actually, an even larger square because of the exponent $\la$
in \refL{GeneralOmegaStatic}), we are naturally led to the assumption that
$\smallabs{\omega^0}$ decreases with the distance from the origin, which leads
to \refL{GeneralOmegaStatic}.

A possible way to around these difficulties is to maintain symmetry by quadrant of the initial vorticity, but to drop the constraint of square symmetry, pinching or cutting out the singularity as in Figure 1. We also require that $\abs{\omega^0}$ be a decreasing function of $\abs{x_1}$ alone.

\medskip

\begin{center}
\scalebox{1} 
{
\begin{pspicture}(0,-1.05)(9.08,1.05)
\psline[linewidth=0.02cm,linestyle=dotted,dotsep=0.16cm](0.5,0.03)(6.6,-0.01)
\usefont{T1}{ptm}{m}{n}
\rput(6.16,0.175){+}
\usefont{T1}{ptm}{m}{n}
\rput(0.94,-0.205){+}
\usefont{T1}{ptm}{m}{n}
\rput(0.96,0.255){-}
\usefont{T1}{ptm}{m}{n}
\rput(6.16,-0.205){-}
\usefont{T1}{ptm}{m}{n}
\rput(3.57,0.375){0}
\usefont{T1}{ptm}{m}{n}
\rput(3.57,-0.405){0}
\pscustom[linewidth=0.02]
{
\newpath
\moveto(0.68,0.55)
\lineto(1.12,0.39)
\curveto(1.34,0.31)(1.765,0.2)(1.97,0.17)
\curveto(2.175,0.14)(2.595,0.085)(2.81,0.06)
\curveto(3.025,0.035)(3.32,0.015)(3.56,0.03)
}
\pscustom[linewidth=0.02]
{
\newpath
\moveto(6.5,0.55)
\lineto(6.06,0.39)
\curveto(5.84,0.31)(5.415,0.2)(5.21,0.17)
\curveto(5.005,0.14)(4.585,0.085)(4.37,0.06)
\curveto(4.155,0.035)(3.86,0.015)(3.62,0.03)
}
\pscustom[linewidth=0.02]
{
\newpath
\moveto(0.68,-0.53)
\lineto(1.12,-0.37)
\curveto(1.34,-0.29)(1.765,-0.18)(1.97,-0.15)
\curveto(2.175,-0.12)(2.595,-0.065)(2.81,-0.04)
\curveto(3.025,-0.015)(3.32,0.0050)(3.56,-0.01)
}
\pscustom[linewidth=0.02]
{
\newpath
\moveto(6.5,-0.53)
\lineto(6.06,-0.37)
\curveto(5.84,-0.29)(5.415,-0.18)(5.21,-0.15)
\curveto(5.005,-0.12)(4.585,-0.065)(4.37,-0.04)
\curveto(4.155,-0.015)(3.86,0.0050)(3.62,-0.01)
}
\usefont{T1}{ptm}{m}{n}
\rput(7.18,0.855){$x_2 = \gamma(x_1)$}
\usefont{T1}{ptm}{m}{n}
\rput(7.31,-0.825){$x_2 = - \gamma(x_1)$}
\psline[linewidth=0.02cm](0.68,0.57)(0.68,-0.55)
\psline[linewidth=0.02cm](6.5,0.55)(6.5,-0.51)
\usefont{T1}{ptm}{m}{n}
\rput(6.92,-0.505){$+r$}
\usefont{T1}{ptm}{m}{n}
\rput(0.3,-0.525){$-r$}
\end{pspicture} 
}

\end{center}

\begin{center}
\textit{Figure 1}
\end{center}

\medskip

The support of the initial vorticity is now $\Omega_0$, the region lying between the curves $x_2 = \pm \gamma(x_1)$ and the vertical lines $x_1 = \pm r$. We require that $\gamma$ be even and that $\gamma'(0) = 0$ (else we gain little over our existing example). One can show that \refEAnd{vFinalBound}{vFinalBoundOverTime} continue to hold, though now the neighborhood on which they hold shrinks rapidly as we make $\gamma$ pinch more tightly (that is, vanish more quickly as the origin is approached). But by pinching the initial vorticity this way we can sharpen the singularity of $\omega^0$ at the origin while leaving its $L^p$-norms and so $\Gamma_t$ unchanged. This then would increase the lower bound in \refE{vFinalBoundOverTime}---if, that is, we can insure that the geometry of the support of the vorticity is not changed too radically over a short time. This is the subject of a future work.


\bigskip
\begin{center}
\textbf{Part II: Inverse problems}
\phantomsection   
\addcontentsline{toc}{chapter}{Part II: Inverse problems}
\end{center}

%
%
\section{Introduction}\label{S:IntroPartII}

\noindent 
The various mappings in Part I can be described schematically as (see \refS{YudoTheorem} for definitions)
\begin{align}
	&v^0 \mapsto v(t, x) \mapsto \psi(t, x), \label{e:Maps1} \\
	&v^0 \mapsto \omega^0 \mapsto \theta(p) \mapsto \mu(x) \mapsto \Gamma_t(x). \label{e:Maps2}
\end{align}
That is, the initial velocity gives the velocity at all time which gives the flow at all time, and the initial velocity also gives the initial vorticity, whose $L^p$-norms, $\theta(p)$, give the MOC of the the velocity field which gives the MOC of the flow. The Osgood condition on $\mu$ insures that both mappings in \refE{Maps1} and the last mapping in \refE{Maps2} are well-defined.

The mappings in \refE{Maps1} are trivial to invert: $v(t, x) = \prt_t \psi(t, \psi^{-1}(t, x))$ and $v^0 = v(0, \cdot)$. The first mapping in \refE{Maps2} is easily inverted as well using the Biot-Savart law, \refE{BSLaw}. The remaining three mappings in \refE{Maps2}, which are the topic of this part of the paper, are another matter.

Our interest in inverting $\theta(p) \mapsto \mu(x)$ and $\mu(x) \mapsto \Gamma_t(x)$ stems from an attempt to answer the question, ``Can the upper bound on the MOC of the flow map given by \refE{MOCFlow} be arbitrarily poor?'' More precisely, we have the following two questions:

\begin{quote}
	\textbf{Question 1}: Given a fixed time $t_0 > 0$ and any (concave) MOC, $f$, does there exist
	a (concave) MOC, $\mu$, such that the associated $\Gamma_{t_0} \ge f$, at least near the origin?
\end{quote}
\begin{quote}
	\textbf{Question 2}: If we obtain $\mu$ from $f$ as in Question 1, can we find a function $\theta$
	that inverts the map, $\theta(p) \mapsto \mu(x)$?
\end{quote}
\Ignore{ 
\begin{quote}
	\textbf{Question 3}: If we obtain $\theta$ from $\mu$, can we find
	a Yudovich vector field that inverts the map, $\omega^0 \mapsto \theta(p)$?
\end{quote}
} 

If we place no restrictions on the MOC, $\mu$, other than the almost minimal ones that $\mu$ is $C^1$ and strictly increasing then we can answer Question 1 affirmatively fairly easily (see \refT{InverseMOC}). In this form, Question 1 is equivalent to some results in functional equations from the early 1960s due to Kordylewski and Kuczma (\cite{KK1960, KK1962}) and Choczewski (\cite{Choczewski1963}).

We will find, however, in  \refS{InvertingEulerMOC} that $\mu$ must be concave (and have further properties as well). Adding only the fundamental constraint that $\mu$ be concave (a common constraint for flows associated with transport equations) we give an affirmative answer to Question 1 in \refR{AnswerToQuestion1}, but only after a fairly lengthy digression into results, due primarily to Zdun \cite{Zdun1979} and Targo{\'n}ski and Zdun \cite{TargZdun1987}, on iteration (sub)groups. In \refS{BasicMOC}, we give some of these results along with proofs in a form more directly applicable to our purposes. In \refS{GammaSingleTime} we consider Question 1, extending in a small way the results of Zdun and using them to answer Question 1. 

After characterizing the required properties of $\mu$ in \refS{InvertingEulerMOC}, and describing some useful implications of these properties in \refS{Dini}, we give a complete answer to Question 2 in \refT{Invertmu}.

Remaining open is whether Question 1 can be answered affirmatively when the additional properties of $\mu$ given in the three equivalent conditions of \refE{YudoCond} are required to hold.

Finally, in \refS{Recoveringomega0}, we give one approach to inverting (approximately) the map, $\omega^0 \mapsto \theta(p)$. This approach yields an additional constraint on the third derivative of $\mu$; it is unclear, however, whether this constraint is an artifact of the method of inversion or whether it, or some similar constraint, is an essential requirement.

\bigskip

\Ignore{ 
``Can one find a Yudovich initial velocity such that the resulting MOC of the flow is arbitrarily poor?'' More precisely, given a MOC, $f$, can one find a $v^0$ in $\Y$ such that $\Gamma_t \ge f$ for some $t > 0$? This would require us to invert the map, $\mu(x) \mapsto \Gamma_t(x)$ given $\Gamma_t$ only at a single time, hence we should expect, and will find, a lack of uniqueness in the inverse. (Inverting this map given $\Gamma_t$ for all $t > 0$ is quite simple by virtue of \refT{InverseMOCAllTime}, the key fact being given in \refT{MOCsAreFlows}.) Inverting the map,  $\theta \mapsto \mu$, also has it difficulties, and is the subject of \refS{InvertingEulerToGetVectorField}.

In both cases, we will find that we can (usually) invert the last two mappings in \refE{Maps2}. We will not, however, be able to do so in a way that insures that certain properties of the resulting inverses required of solutions to the Euler equations (namely, $\mu$ being concave and $\log \theta$ being convex) are assured. Obtaining such properties will remain an open problem.
} 

In what follows, we will have need to distinguish among the following three degrees of concavity (or, similarly, convexity):
\begin{definition}\label{D:Concavity}
	Assume that $f$ is a twice differentiable function on some open interval of $\R$. Then we
	say that
	\begin{enumerate}
		\item
			$f$ is \textit{concave} if $f'' \le 0$;
			
		\item
			$f$ is \textit{strictly concave} if $f'$ is strictly decreasing;
			
		\item
			$f$ is \textit{strongly strictly concave} if $f'' < 0$.
	\end{enumerate}
\end{definition}
Observe that $(3) \implies (2) \implies (1)$.

%
%
\section{Properties of the MOC of the flow map}\label{S:BasicMOC}

\noindent The relation between the (time-independent) MOC of a vector field and of its flow is given in the following, classical lemma (which was used in the proof of \refT{YudovichTheorem}):

\begin{lemma}\label{L:ClassicalFlow}
Let $\Omega$ be a domain in $\R^n$ and $v$ be a vector field on $\Omega$ that admits a MOC, $\mu$, that satisfies the Osgood condition, \refE{muOsgood}.
\Ignore{ 
\begin{align}\label{e:muOsgood}
	\int_0^1 \frac{ds}{\mu(s)} = \iny.
\end{align}
} 
Then $v$ has a unique associated flow, $\psi$, continuous from $\R \times \Omega$ to $\R^n$, such that
    \begin{align}\label{e:psiFlow}
        \psi(t, x) = x + \int_0^t v(\psi(s, x)) \, ds
    \end{align}
for all $x$ in $\Omega$. For $t \ge 0$ define $\Gamma_t \colon [0, \iny) \to [0, \iny)$ by $\Gamma_t(0) = 0$ and for $x > 0$ by
\begin{align}\label{e:mufGamma}
        I_t(x)
		:= \int_x^{\Gamma_t(x)}  \frac{dr}{\mu (r)} = t
\end{align}
for all $t \ge 0$.
Then $\Gamma_t$ is a modulus of continuity for the flow, $\psi$, in the sense of \refE{psiMOC}.
\end{lemma}
\begin{proof}
The proof is classical. See, for instance, Theorem 5.2.1 of \cite{C1998}. (Chemin's theorem is stated for log-Lipschitz vector fields, but the proof applies for any vector field having a MOC satisfying Osgood's condition.)
\end{proof}

\Ignore{ 
The MOC, $\Gamma$, of the flow map is defined by
\begin{align}\label{e:mufGamma}
	I_t(h)
		:= \int_h^{\Gamma_t(h)} \frac{ds}{\mu(s)} = t
		\text{ for all } t, h > 0.
\end{align}
} 
We will alternately write $\Gamma(t, x)$ and $\Gamma_t(x)$.

\begin{theorem}\label{T:MOCsAreFlows}
	Assume  that $\mu$ is a MOC with $\mu > 0$ on $(0, \iny)$ and that it satisfies the Osgood
	condition, \refE{muOsgood}. Then \refE{mufGamma} uniquely defines 
	$\Gamma \colon [0, \iny)^2 \to [0, \iny)$ with $\Gamma_0 = \textit{identity map}$
	and, for all $t > 0$, $\Gamma_t$ strictly increasing, $\Gamma_t(0) = 0$, and $\Gamma_t(x) > x$
	for all $x > 0$.
	Also, $\Gamma$ is continuously differentiable.
	
	 If $\mu$ is strictly increasing on $(0, a)$ for some $0 < a \le \iny$ then $\Gamma_t' > 1$
	 on $(0, \Gamma_t^{-1}(a))$ for all $t > 0$.
	
	Moreover, viewing $\mu$ as a $1$-vector (velocity) field on $[0, \iny)$, $\Gamma$ satisfies the
	transport equation,
	\begin{align}\label{e:transport}
	\left\{
		\begin{array}{rl}
			\prt_t \Gamma(t, x) - \mu(x) \prt_x \Gamma(t, x)	 = 0,
				& (t, x) \text{ in } [0, \iny) \times [0, \iny), \\
			\Gamma(0, x) = x, & x \text{ in } [0, \iny).
		\end{array}
	\right.
	\end{align}
	Also viewing each $\Gamma_t$, $t \ge 0$, as $1$-vector fields on $[0, \iny)$,
	$\Gamma$ is the flow associated with the velocity field, $\mu$. That is, $\Gamma$ is
	its own flow:
	\begin{align}\label{e:flow}
	\left\{
		\begin{array}{rl}
			\prt_t \Gamma(t, x) = \mu(\Gamma(t, x)),
				& (t, x) \text{ in } [0, \iny) \times [0, \iny), \\
			\Gamma(0, x) = x, & x \text{ in } [0, \iny).
		\end{array}
	\right.
	\end{align}
	Finally, the following identity holds for all $(t, x)$ in $[0, \iny) \times [0, \iny)$:
	\begin{align}\label{e:muFunc}
		\mu(\Gamma_t(x)) = \Gamma_t'(x) \mu(x).
	\end{align}
\end{theorem}
\begin{proof}
The conclusions in the first paragraph of this lemma, with the exception of the last sentence, follow immediately from \refE{mufGamma}.

Taking the derivative with respect to $x$ of \refE{mufGamma},
\begin{align*}
	I_t'(x)
		= \frac{\Gamma_t'(x)}{\mu(\Gamma_t(x))} - \frac{1}{\mu(x)}
		= 0
\end{align*}
for all $t, x > 0$. This gives \refE{muFunc} and shows that $\Gamma$ is continuously differentiable in space and that if $\mu$ is strictly increasing on $(0, a)$ then $\Gamma_t' > 1$ on $(0, \Gamma_t^{-1}(a))$ for all $t > 0$.

Taking the derivative of $I_t(x)$ with respect to $t$, we have
\begin{align*}
	\prt_t I_t(x)
		= \frac{\prt_t \Gamma_t(x)}{\mu(\Gamma_t(x))}
		= 1
\end{align*}
so $\mu(\Gamma_t(x)) = \prt_t \Gamma_t(x)$ for all $t, x > 0$, which is \refE{flow}. It follows, then, that $\Gamma$ is continuously differentiable in time and that $\prt_t \Gamma_t(x) = \Gamma_t'(x) \mu(x)$, or,
\begin{align}\label{e:muh}
	\mu(x)
		= \frac{\prt_t \Gamma_t(x)}{\Gamma_t'(x)},
\end{align}
an equality that hold for all $x$ and is independent of $t \ge 0$. This is \refE{transport}.
\end{proof}

Moreover, we have the following simple but important lemma:

\begin{lemma}\label{L:muGammaConcave}
Let $\mu$ be as in \refT{MOCsAreFlows}.
The following are equivalent:
\begin{enumerate}
	\item
		The MOC, $\mu = \prt_t \Gamma_t|_{t = 0}$, is differentiable and concave on $(0, a)$.
		
	\item
		$\Gamma_t$ is twice differentiable on $(0, \iny)$ and concave on $(0, \Gamma_t^{-1}(a))$
		for all $t > 0$.
		
	\item
		$\Gamma_t$ is twice differentiable on $(0, \iny)$ and concave on $(0, \Gamma_t^{-1}(a))$
		for all $t$ in $(0, \delta)$ for some $\delta > 0$.
\end{enumerate}
Here, $0 < a \le \iny$. Furthermore, strong strict concavity of $\Gamma_t$ on $(0, \Gamma_t^{-1}(a))$ for all $t > 0$ implies strict concavity of $\mu$ on $(0, a)$.
\end{lemma}
\begin{proof}
That $\mu = \prt_t \Gamma_t|_{t = 0}$ follows from \refE{muh}, since $\Gamma_0(x) = x$.
Taking the derivative of \refE{muFunc} with respect to $x$ gives
\begin{align*}
	\begin{split}
	\mu'(\Gamma_t(x))
		&= \frac{\pr{\mu(\Gamma_t(x))}'}{\Gamma_t'(x)}
		= \frac{\Gamma_t'(x) \mu'(x) + \Gamma_t''(x) \mu(x)}{\Gamma_t'(x)} \\
		&= \mu'(x) + \Gamma_t''(x) \frac{\mu(x)}{\Gamma_t'(x)},
	\end{split}
\end{align*}
or,
\begin{align}\label{e:GammatConcave}
	\Gamma_t''(h)
		= \pr{\mu'(\Gamma_t(h)) - \mu'(h)} \frac{\Gamma_t'(h)}{\mu(h)}.
\end{align}
Since $\Gamma_t(h) > h$ for all $t, h > 0$, $\Gamma_t''(h)$ and $\mu'(\Gamma_t(h)) - \mu'(h)$ have the same sign for all $t, h > 0$.

Suppose that (3) holds and let $h$ lie in $(0, a)$. Then for all sufficiently small $t > 0$, $h$ lies in $(0, \Gamma_t^{-1}(a))$ which shows by \refE{GammatConcave} that $\mu'(\Gamma_t(h)) \le \mu'(h)$. But $\Gamma_t(h)$ decreases to $h$ as $t \to 0^+$, so $\mu'$ is increasing at $h$. Hence, $\mu'$ is increasing for all $h$ in $(0, a)$ meaning that $\mu$ is concave on $(0, a)$. That is, $(3) \implies (1)$.

Now assume that $(1)$ holds. This shows directly from \refE{GammatConcave} that $\Gamma_t''(h)$ is concave; that is, $(1) \implies (2)$. That  $(2) \implies (3)$ is immediate.

The statement involving strong strict concavity is a small modification of the argument above.
\end{proof}

\begin{remark}
Other than the implication involving strict concavity, one need only assume in \refL{muGammaConcave} that $\mu$ is continuous and that $\Gamma_t$ is continuously differentiable, as is shown in the results of Zdun \cite{Zdun1979} that we discuss in the next section.
\end{remark}

It follows from \refE{flow} that $\prt_t \Gamma_t(x) > 0$, so $\Gamma_t(x)$ is an increasing function of $t$ for fixed $x > 0$; that is, the MOC of the flow gets worse with time.
This observation leads to the following:

\begin{lemma}\label{L:tGammatIncreasing}
If $\mu$ is (strictly) increasing then the map, $\Gamma(\cdot, x)$ is a (strictly) increasing \ReallyConvex function for all $x > 0$ and $\prt_x \Gamma(\cdot, x)$ is (strictly) increasing.
\end{lemma}
\begin{proof}
From \refE{flow},
$
	\prt_t \Gamma_t(x) = \mu(\Gamma_t(x)),
$
and as observed above, $\Gamma_t(x)$ increases with $t$. Since $\mu$ also increases by assumption, it follows that $\prt_t \Gamma_t(x)$ increases with $t$; that is, $t \mapsto \Gamma_t(x)$ is \ReallyConvexPeriod. From \refE{transport}, 
$
	\prt_t \Gamma_t(x) = \mu(x) \Gamma_t'(x),
$
and we conclude that $t \mapsto \Gamma_t'(x)$ is strictly increasing.
\end{proof}

In the context of iteration (semi)groups, all of the results in this section, with the possible exception of \refL{tGammatIncreasing}, are known, in broader generality, and are due primarily to Zdun \cite{Zdun1979} and Targo{\'n}ski and Zdun \cite{TargZdun1987} (also see Section 3.3 of \cite{Targ1981}, which summarizes many of the key results in \cite{Zdun1979}).  We make use of these connections, along with some of the deeper results in \cite{Zdun1979, TargZdun1987}, in the next section.

\Ignore{ 

These depictions of $\Gamma$ above leads us to the following definition:

\begin{definition}\label{D:MOCFlowFamily}
	Let $\Gamma \colon [0, \iny) \times [0, \iny) \to [0, \iny)$ and write $\Gamma(t, x)$ alternately
	as $\Gamma_t(x)$. Assume that
	$\Gamma$ is continuously differentiable on $(0, \iny)^2$, $\Gamma_0$ is the identity, 
	and for each $t > 0$, $\Gamma_t$ is a strictly increasing MOC with
	$\Gamma_t(x) > x$ for all $x > 0$.
	Assume that for all $t, x > 0$, $\prt_t \Gamma_t(x) > 0$ and that
	$\mu(x) := \prt_t \Gamma_t(x) / \Gamma_t'(x)$ is a function only
	of $x$ with $\mu(\Gamma_t(x)) = \Gamma_t'(x) \mu(x)$.
	Then we say that $\Gamma$ is an \textit{acceptable family of moduli
	of continuity (MOC) for a flow}.
\end{definition}

In \refD{MOCFlowFamily}, we do not need to include the requirement that $\mu = \prt_t \Gamma_t/\Gamma_t'$ satisfies the Osgood condition, since this condition automatically follows once we know that \refE{mufGamma} holds. Also, \refT{MOCsAreFlows} through \refL{tGammatIncreasing} continue to hold for $\Gamma$ an acceptable family of MOC for a flow.

Unsurprisingly, since being an acceptable family of MOC includes such strong properties, it is easy to obtain $\mu$ from such a family and show that \refE{mufGamma} is satisfied.

\begin{theorem}\label{T:InverseMOCAllTime}
	Suppose that $\Gamma$ is an acceptable family of MOC for a flow. Then defining $\mu$
	as in \refE{muh} for any $t = t_0 > 0$, $\mu$ is a MOC with $\mu(h) > 0$ for all $h > 0$,
	$\mu$ satisfies \refE{mufGamma}, and $\mu$ satisfies the Osgood condition, \refE{muOsgood}.
\end{theorem}
\begin{proof}
By \refD{MOCFlowFamily}, $\Gamma$ is continuously differentiable with $\Gamma_t' > 0$ so $\mu$ as given by \refE{muh} is continuous with $\mu(0) = 0$ and $\mu(h) > 0$ for all $h > 0$.

To show that \refE{mufGamma} holds, observe that
\begin{align*}
	I_t(h)
		= \int_h^{\Gamma_t(h)} \frac{ds}{\mu(s)} 
\end{align*}
satisfies,
\begin{align*}
	\prt_t I_t(h)
		= \frac{\prt_t \Gamma_t(h)}{\mu(\Gamma_t(h))}
		= \frac{\mu(h) \Gamma_t'(h)}{\mu(h) \Gamma_t'(h)}
		= 1,
\end{align*}
since $\prt_t \Gamma_t(h) = \mu(h) \Gamma_t'(h)$, $\mu(\Gamma_t(h)) = \mu(h) \Gamma_t'(h)$, and $\mu(h) > 0$ for all $h > 0$.

We conclude that $I_t(h) = t + C$. But $I_0(h) = h$ for all $h$ so $C = 0$, showing that \refE{mufGamma} holds. The Osgood condition on $\mu$ then follows from \refE{mufGamma}.
\Ignore{ 

Now suppose that $\Gamma_t$ is concave for $t$ in $(0, \delta)$ for some $\delta > 0$. Then $\mu(\Gamma_t(h)) = \Gamma_t'(h) \mu(h)$ for all $t, h > 0$ gives \refE{mupGammat}.
But $\Gamma_t'(h) > 0$ and $\Gamma_t''(h) < 0$ so $\mu'(\Gamma_t(h)) \le \mu'(h)$ for all $h, t > 0$. Since $\Gamma_t(h) > h$ and $\Gamma_t(h) \to h$ as $h \to 0$, it follows that $\mu'(h)$ is decreasing. That is, $\mu$ is concave.
} 
\end{proof}

\Ignore{ 
In our construction of $\mu$ in the proof of \refT{InverseMOC}, $\mu(f(h)) = f'(h) \mu(h)$ on $(0, h_0)$, so, using the relationship in \refE{fConcave}, we have $\mu'(f(h)) < \mu'(h)$ on $(0, h_0])$. Thus, $\mu'$ comes close, in a sense, to being convcave.
} 

\Ignore{ 

Suppose that $\mu$ is a strictly increasing, concave MOC satisfying the Osgood condition, \refE{muOsgood}. Fix $t_0 > 0$ and let $\Gamma_{t_0}$ be an acceptable MOC for a flow as in \refD{AcceptableMOC}. Fix $h > 0$ and consider the initial value problem,
\begin{align}\label{e:ODE}
	\left\{
		\begin{array}{rl}
			\prt_t \Gamma_t(h) = \mu(\Gamma_t(h))	& t \text{ in } (0, \iny), \\
			\Gamma_t = \Gamma_{t_0}, & \text{ for } t = t_0.
		\end{array}
	\right.
\end{align}
Then for all $x, y > 0$,
\begin{align*}
	\abs{\mu(x) - \mu(y)}
		= \abs{\int_x^y \mu'(s) \, ds}
		\le \int_0^{\abs{x - y}} \mu'(s) \, ds
		= \mu(\abs{x - y}).
\end{align*}
That is, $\mu$ has itself as a MOC (this required both $\mu$ increasing and $\mu$ concave). Since $\mu$ satisfies the Osgood condition, by \refL{ClassicalFlow}, \refE{ODE} has a unique solution.
} 

%
%
\Ignore{ 
Another simple but useful observation is the following:

\begin{lemma}\label{L:muConcaveIncreasing}
Let $\mu$ be as in \refT{MOCsAreFlows}.
If $\mu$ is (strictly) concave then $\mu$ must be (strictly) increasing.
\end{lemma}
\begin{proof}
This follows easily if $\mu$ is continuously differentiable from the (strictly) decreasing nature of $\mu$.
If $\mu$ is only assumed to be concave then we can make a similar, though more involved, argument using difference quotients, which decrease in a manner similar to the derivative (see, for example, Chapter 4 Exercise 23 of \cite{RudinPrinciples}).
\end{proof}

We show in \refT{fConcaveGivesUniqueness} that a MOC, $f$, for a flow at time $t_0 > 0$ results in a unique family, $\Gamma$, of acceptable MOC for a flow via \refE{mufGamma} if the MOC, $\mu$, for the vector field is concave. It is an open question how close the concavity of $\mu$ comes to being a necessary condition to insure uniqueness.

\begin{theorem}\label{T:fConcaveGivesUniqueness}
Let $\Gamma$ be an acceptable family of MOC for a flow and let $f := \Gamma_{t_0}$
for some $t_0 > 0$.
Suppose that some MOC, $\mu$, (perhaps given by \refT{InverseMOC}) satisfies \refE{muf} for $t = t_0$. Use this $\mu$ to define $\ol{\Gamma}_t$ for all $t > 0$ by \refE{mufGamma} (observe that $\ol{\Gamma}_{t_0} = f$). If $\mu$ is concave on all of $[0, h_0)$ for some $h_0 > 0$ then $\Gamma = \ol{\Gamma}$.
\Ignore{ 
and, moreover, for all $s, t$ in $[0, \iny)$, $ t > s$,
\begin{align}\label{e:GammaMOCInTime}
	0 <
		\Gamma_t(h) - \Gamma_s(h)
		< (t - s) \mu(\Gamma_t(h)).
\end{align}
} 
\end{theorem}
\begin{proof}
Assume first that $\mu$ is concave on all of $[0, \iny)$ so that by \refL{muConcaveIncreasing} it is also increasing. Then, for any $x, y > 0$,
\begin{align*}
	\abs{\mu(x) - \mu(y)}
		= \abs{\int_x^y \mu'(s) \, ds}
		\le \int_0^{\abs{x - y}} \mu'(s) \, ds
		= \mu(\abs{x - y}),
\end{align*}
where the inequality holds because $\mu'$ is decreasing but nonnegative, $\mu$ being concave but increasing. That is, $\mu$ has itself as a MOC.

Consider the initial value problem,
\begin{align}\label{e:ODE}
	\left\{
		\begin{array}{rl}
			\prt_t \widetilde{\Gamma}_t(h) = \mu(\widetilde{\Gamma}_t(h)),	& (t, h) \text{ in } (0, \iny)
				\times[0, \iny), \\
			\widetilde{\Gamma}(0, x) = x, & x \text{ in } [0, \iny).
		\end{array}
	\right.
\end{align}
This is the differential form of \refE{psiFlow}, so because $\mu$ satisfies the Osgood condition and is its own MOC, by \refL{ClassicalFlow} applied with $n = 1$, \refE{ODE} has a unique solution. But by \refT{MOCsAreFlows}, $\ol{\Gamma}$ is also a solution of \refE{flow}, as is $\Gamma$ by \refD{MOCFlowFamily}. Therefore,  $\Gamma = \ol{\Gamma} = \widetilde{\Gamma}$.


Now assume only that $\mu$ is concave on $[0, h_0)$ for some $h_0 > 0$. Reasoning as in the proof of \refL{muConcaveIncreasing}, it follows that $\mu$ must be increasing on some possibly smaller interval, which we relabel as $[0, h_0)$. Choose any $z > h_0$ and let $M$ be the maximum value of $\mu'$ on $[h_0, z]$. Then for all $x, y$ in $[0, z]$ with $\abs{x - y} < h_0$, it is easy to see that
\begin{align*}
	\abs{\mu(x) - \mu(y)}
		&\le \mu (\abs{x - y}) + M \abs{x - y} \\
		&\le \ol{\mu}(\abs{x - y})
		:= 2 \max \set{\mu\abs{x - y}, M \abs{x - y}}.
\end{align*}
This defines a MOC, $\ol{\mu}$, for $\mu$ on $[0, h_0]$, which we extend arbitrarily to all of $(0, z)$. It is easy to verify that $\ol{\mu}$, defined on $[0, h_0]$, satisfies the Osgood condition in the form $\int_0^{h_0} \ol{\mu}(x)^{-1} \, dx = \iny$, and this is enough to apply \refL{ClassicalFlow} to the $1$-vector field, $\mu$, to obtain a unique solution to \refE{ODE} for all time for $h$ in $\Omega =  (0, z)$. (Note that the MOC need only apply near the origin in the application of \refL{ClassicalFlow}.) Since $z$ was arbitrary, this gives a unique solution for all $(t, h)$ in $[0, \iny) \times [0, \iny)$.
\end{proof}
} 

\begin{theorem}\label{T:fConcaveGivesUniqueness}
Let $\Gamma$ be an acceptable family of MOC for a flow and let $f := \Gamma_{t_0}$
for some $t_0 > 0$.
Suppose that some MOC, $\mu$, (perhaps given by \refT{InverseMOC}) satisfies \refE{muf} for $t = t_0$. Use this $\mu$ to define $\ol{\Gamma}_t$ for all $t > 0$ by \refE{mufGamma} (observe that $\ol{\Gamma}_{t_0} = f$). If $\mu$ is strictly increasing and concave then $\Gamma = \ol{\Gamma}$.
\end{theorem}
\begin{proof}
For any $x, y > 0$,
\begin{align*}
	\abs{\mu(x) - \mu(y)}
		= \abs{\int_x^y \mu'(s) \, ds}
		\le \int_0^{\abs{x - y}} \mu'(s) \, ds
		= \mu(\abs{x - y}),
\end{align*}
where the inequality holds because $\mu'$ is decreasing but nonnegative, $\mu$ being concave but increasing. That is, $\mu$ has itself as a MOC.

Consider the initial value problem,
\begin{align}\label{e:ODE}
	\left\{
		\begin{array}{rl}
			\prt_t \widetilde{\Gamma}_t(h) = \mu(\widetilde{\Gamma}_t(h)),	& (t, h) \text{ in } (0, \iny)
				\times[0, \iny), \\
			\widetilde{\Gamma}(0, x) = x, & x \text{ in } [0, \iny).
		\end{array}
	\right.
\end{align}
This is the differential form of \refE{psiFlow}, so because $\mu$ satisfies the Osgood condition and is its own MOC, by \refL{ClassicalFlow} applied with $n = 1$, \refE{ODE} has a unique solution. But by \refT{MOCsAreFlows}, $\ol{\Gamma}$ is also a solution of \refE{flow}, as is $\Gamma$ by \refD{MOCFlowFamily}. Therefore,  $\Gamma = \ol{\Gamma} = \widetilde{\Gamma}$.
\end{proof}

It turns out that an acceptable family of MOC is the same thing as a differentiable \textit{iteration group} (see \refD{CIG}, below). In fact, the standard interpretation of iteration (to use Targonski's terminology in \cite{Targ1981}) is as a dynamical system, or a flow map when considering continuous time. Our point of view is just slightly different, as we view it in terms of the moduli of continuity of the flow maps. As in the proof of \refT{fConcaveGivesUniqueness}, these two interpretations can coincide when the MOC of the vector field is concave and the flow is one-dimensional, but otherwise they are not quite the same.

In the context of iteration (semi)groups, all of the results in this section, with the possible exception of \refL{tGammatIncreasing}, are known, in broader generality, and are due primarily to Zdun \cite{Zdun1979} (or see Section 3.3 of \cite{Targ1981}, which summarizes some of the results in \cite{Zdun1979}).  We make use of these connections, along with some of the deeper results in \cite{Zdun1979}, in the next section.

} 

%
%
\section{Inverting the MOC of the flow map}\label{S:GammaSingleTime}

\noindent In \refS{BasicMOC} we characterized the properties of the MOC, $\Gamma$, of the flow map, in particular as regards concavity properties of $\Gamma_t$, and found that it is easy to obtain the corresponding MOC, $\mu$, of the vector field from \refE{muh}. But this formula requires that we know $\Gamma_t$ for all $t$ in a neighborhood of the origin. In this section we attempt to answer Question 1 of \refS{IntroPartII}, in which we only know $\Gamma_{t_0}$ at one time, $t_0 > 0$. In this case, we do not expect to obtain a unique $\mu$ and so do not expect to obtain a unique $\Gamma$.
We are especially interested in determining whether we can find a $\mu$ that is concave.

Our starting point will be \refT{InverseMOC}, in which we construct a $\mu$ that is strictly increasing, but only on some interval $[0, a)$. This is adequate for our uses, but complicates all subsequent arguments because we have to keep track of the interval on which various functions are guaranteed to be strictly increasing or concave: this is the purpose of introducing \refDAnd{iotaV}{AcceptableMOC}. The essential meaning of the theorems are easier to grasp, however, if one ignores any statements involving $\iota$, $V$, or $J$, and just imagines that $\mu$ is strictly increasing on all of $[0, \iny)$. In any case, these definitions are required along the way, but not in the statement \refT{ConcaveCIG}, the main result of this section.

\begin{definition}\label{D:iotaV}
	For any function, $f$, on $[0, \iny)$ we define
	\begin{align*}
		\iota(f) &= \sup \set{a \in [0, \iny] \colon f \text{ is strictly increasing on } [0, a)}, \\
		V(f) &= \sup \set{a \in [0, \iny] \colon f \text{ is concave on } (0, a)}.
	\end{align*}

\end{definition}

\begin{definition}\label{D:AcceptableMOC}
	Let $f$ be a $C^1$ MOC and let $a = \iota(f(x) - x)$.
	We say that a MOC, $f$, such that $f(x) > x$ for all $x$ in $(0, \iny)$ is \textit{acceptable}
	or \textit{globally acceptable}
	if $a = \iny$ and is \textit{locally acceptable}
	if $a > 0$. We define $J(f) = a$.
\end{definition}

\begin{remark}\label{R:MOCFlow}
	It follows from \refD{AcceptableMOC} that $f' > 1$ on $(0, J(f))$.
	Also, $f$ concave is compatible with $f$ being acceptable.
\end{remark}


\Ignore{ 
In \refT{InverseMOC}, we invert \refE{muf} non-uniquely for a fixed time $t > 0$ (see \refR{InverseMOC1}). The inversion only holds for $h$ up to an arbitrary fixed $h_0 > 0$, but that is all that is important, because it is the germs of $\mu$ and of $f$ at the origin that are relevant to the fluid flow.
} 

We now answer Question 1 of \refS{IntroPartII} affirmatively.

\begin{theorem}\label{T:InverseMOC}
	Fix $t_0 > 0$.
	Given any $f$ that is a globally acceptable MOC
	there exists a continuous MOC, $\mu$, satisfying the Osgood condition,
	\refE{muOsgood}, with $\mu > 0$ on $(0, \iny)$, such that
	\begin{align}\label{e:muf}
        		I(x)
			:= \int_x^{f(x)}  \frac{dr}{\mu (r)}
			= t_0
	\end{align}
	 for all $x > 0$.
	 \Ignore{ 
	 If $\lim_{x \to \iny} x/f(x) = 0$ then
	 $
		\lim_{x \to \iny} \mu(x)/x
			= 0.
	$
	} 
	If $f$ is $C^k$, $k \ge 1$, then $\mu$ can be chosen to be $C^{k - 1}$.
\end{theorem}
\begin{proof}
	Choose $a > 0$ arbitrarily, then choose a smooth $\mu$, strictly
	increasing on the interval
	$[a, f(a)]$ such that the following two conditions are satisfied:
	\begin{align}
		\int_{a}^{f(a)} \frac{ds}{\mu(s)} &= t_0, \label{e:Cond1} \\
		\mu(f(a)) &= f'(a) \mu(a). \label{e:Cond2}
	\end{align}
	It is easy to see that we can find a function $\mu$
	on $[a, f(a)]$ that satisfies \refEAnd{Cond1}{Cond2}
	if and only if we choose $\mu(a)$ so that
	\begin{align}\label{e:fDiffReq}
		\frac{f(a) - a}{f'(a) t_0}
			< \mu(a)
			< \frac{f(a) - a}{t_0}.
	\end{align}
	Since $f'(a) > 1$ this is always possible.
	\Ignore{ 
	Then for \refE{Cond3} to hold we must have
	\begin{align*}
		1
			&= \int_{a}^{f(a)} \frac{ds}{\mu(s)}
			< \int_{a}^{f(a)} \frac{f'(s)}{\mu(f(a))} \, ds
			= \frac{f(f(a)) - f(a)}{\mu(f(a))} \\
			&= \frac{f(f(a)) - f(a)}{f'(a) \mu(a)}
			\le \frac{f(f(a)) - f(a)}{f'(a)}\frac{f'(a)}{f(a) - a} \\
			&= \frac{f(f(a)) - f(a)}{f(a) - a}.
	\end{align*}
	But by \refR{MOCFlow}, $f' > 1$, so by the mean value theorem,
	the final ratio above is always greater than 1; that is,
	this inequality does not represent a restriction.
	
	Also,
	\begin{align*}
		&\frac{f(f(a)) - f(a)}{f'(a) \mu(a)}
			\ge \frac{f(f(a)) - f(a)}{f'(a)}\frac{1}{f(a) - a} \\
			&\qquad
			= \frac{f(f(a)) - f(a)}{f(a) - a} \frac{1}{f'(a)}.
	\end{align*}
	But $f'$ is strictly decreasing so by the mean value theorem,
	\begin{align*}
		f'(a) > \frac{f(f(a)) - f(a)}{f(a) - a}
	\end{align*}
	} 
	
	Next define $\mu(x)$ for $x$ in the interval $[f^{-1}(a), a]$ by
	\begin{align}\label{e:muExtension}
		\mu(x) = \frac{\mu(f(x))}{f'(x)}
	\end{align}
	and note that if $f$ is concave then $\mu$ is strictly increasing on $[f^{-1}(a), a]$ because
	$\mu(f(x))$ is strictly increasing and $f'$ is
	decreasing on that interval. Also, $\mu$ is continuous, in particular at
	$a$ by \refE{Cond2}.
	
	Suppose that $f$ lies in $C^2((0, \iny))$. Then taking the derivative of \refE{muExtension},
	\begin{align*}
		\mu'(x) = \mu'(f(x)) - \mu(f(x)) \frac{f''(x)}{f'(x)}^2.
	\end{align*}
	Hence, $\mu$ is in $C^1((f^{-1}(a), a)$. To insure that $\mu'$ is continuous at $a$, we simply
	require that $\mu$ be chosen on the interval $[a, f(a)]$ such that
	\begin{align}\label{e:diffContCond}
		\mu'(a) = \mu'(f(a)) - \mu(f(a)) \frac{f''(a)}{f'(a)}^2,
	\end{align}
	$\mu'(a)$ being a right-sided derivative and $\mu'(f(a))$ a left-sided derivative. Equality
	in \refE{diffContCond} can be assured by changing the
	definition of $\mu$ on $[a, f(a)]$ an arbitrarily small amount near the endpoints. Hence,
	we can make \refE{diffContCond} hold under the same conditions \refE{fDiffReq}
	while retaining
	the other properties of $\mu$ already established. But once \refE{diffContCond} holds, the continuity
	of $f$, $f'$, and $f''$ makes $\mu'$ continuous on $[f^{-1}(a), f(a)]$. A straightforward
	extension of this argument shows that if $f$ lies in $C^k((0, \iny))$ then $\mu$ can be chosen to lie 
	in $C^{k - 1}([f^{-1}(a), f(a)])$.

	Extending this definition of $\mu$ inductively to $[f^{-n}(a), f^{-n + 1}(a)]$
	we unambiguously define $\mu$ on all of $(0, f(a)]$ and the resulting $\mu$
	is positive, continuous, satisfies \refE{muExtension}
	for all $x$ in $(0, a]$, is strictly increasing if $f$ is concave, and if $f$ lies in
	$C^k((0, \iny))$ then $\mu$ lies in $C^{k - 1}((0, f(a)])$.

	Next, extend $\mu$ to the interval $[f(a), f(f(a))]$ using
	$
		\mu(f(x)) = f'(x) \mu(x)
	$,
	the complement of \refE{muExtension}, and inductively extend $\mu$ to all of $(0, \iny)$.
	Then $\mu$ satisfies \refE{muExtension} for $x > a$ as well and
	if $f$ lies in $C^k((0, \iny))$ then $\mu$ lies in
	$C^{k - 1}((0, \iny))$. (Even if $f$ is concave,
	extending $\mu$ in this way, does not insure that it is an increasing function for $x > f(a)$,
	for though $\mu$ is increasing on the interval $[a, f(a)]$, $f'$ would be decreasing.)
	
	\Ignore{ 
	and similarly define $\mu$ on the interval $[f(a), f(f(a))]$ by
	\begin{align}\label{e:mufpmu}
		\mu(f(x)) = f'(x) \mu(x)
	\end{align}
	where $x$ ranges over the interval $[a, f(a)]$.
	Extending this definition of $\mu$ inductively in both directions we
	unambiguously define $\mu$ on all of $(0, \iny)$ and the resulting $\mu$
	is positive, continuous, and satisfies \refE{mufpmu} on all of $(0, \iny)$.
	
	Now, for $\mu$ to be strictly increasing in $x$, we must have
	$\mu'(f(x) > 0$ for all $x > 0$.
	But by ]\refE{mufpmu}, this becomes
	\begin{align*}
		\mu'(f(x))
			&= \frac{\mu(f(x))'}{f'(x)}
			= \frac{\brac{f'(x) \mu(x)}'}{f'(x)} \\
			&= \frac{f'(x) \mu'(x) + f''(x) \mu(x)}{f'(x)}
			= \mu'(x) + \frac{f''(x)}{f'(x)} \mu(x).
	\end{align*}
	Since we assumed that $f$ is concave this shows that $\mu'(f(x))
	< \mu'(x)$, which is compatible with $\mu$ being concave. Moreover, to have
	$\mu$ strictly increasing we must have,
	\begin{align*}
		\frac{\mu'(x)}{\mu(x)} > - \frac{f''(x)}{f'(x)}.
	\end{align*}
	
	Writing this inequality as
	\begin{align*}
		\brac{\log \mu(x)}' > - \brac{\log(f'(x))}'
	\end{align*}
	and integrating from $a$ to $x > a$ gives
	\begin{align*}
		\log \mu(x) - \log \mu(a)
			> \log f'(a) - \log f'(x).
	\end{align*}
	Exponentiating both sides gives
	\begin{align*}
		\frac{\mu(x)}{\mu(a)}
			> \frac{f'(a)}{f'(x)},
	\end{align*}
	or,
	\begin{align*}
		\mu(x)
			>  \frac{f'(a) \mu(a)}{f'(x)}
			=  \frac{\mu(f(a))}{f'(x)}.
	\end{align*}
	\textbf{This follows immediately, though, from $\mu(f(x)) > \mu(f(a))$.}
	
	HERE IS A PROBLEM!!!
	} 
	
	Defining $I \colon (0, \iny) \to [0, \iny)$ as in \refE{muf},
	it follows from \refE{muExtension} that
	\begin{align*}
		I'(x)
			= \frac{f'(x)}{\mu(f(x))} - \frac{1}{\mu(x)}
			= 0
	\end{align*}
	for all $x > 0$. Therefore, $I$ is a constant function. But by
	construction,
	$I(a) = t_0$ so $I(x) = t_0$ for all $x > 0$. That is, $\mu$ satisfies
	\refE{muf} for all $x > 0$.
	
	It follows by the absolute continuity of the integral and the fact that
	$f$ is \textit{strictly} increasing that
	\begin{align}\label{e:OsgoodHolds}
		\int_0^{a} \frac{ds}{\mu(s)}
			&= \sum_{n = 0}^\iny
				\int_{f^{-(n + 1)}(a)}^{f^{-n}(a)} \frac{ds}{\mu(s)}
				= \sum_{n = 0}^\iny t_0
				= \iny,
	\end{align}
	so that $\mu$ satisfies the Osgood condition, \refE{muOsgood}, and also that we
	must have
	\begin{align}\label{e:muZeroEqualsZero}
		\lim_{s \to 0^+} \mu(s) = 0,
	\end{align}
	so that we can extend $\mu$ continuously to $[0, \iny)$ by setting
	$\mu(0) = 0$.
\end{proof}

\begin{remark}
After expressing the relation in \refE{muf} in the form $\mu(f(x)) = f'(x) \mu(x)$, we can view our construction of the function $\mu$ as an application of Theorem 2.1 of \cite{Kuczma1968}, which is due to Kordylewski and Kuczma (\cite{KK1960, KK1962}). That $\mu$ is $C^{k-1}$ when $f$ is $C^k$ can be seen as an application of Theorem 4.1 of \cite{Kuczma1968}, which is due to Choczewski (\cite{Choczewski1963}). Also see Theorem 6.2 of \cite{Zdun1979} (quoted in Proposition 3.3.45 of \cite{Targ1981}).
\end{remark}

\DetailSome{
In the proof of \refT{InverseMOC}, there is a great deal of flexibility in the construction of the inverse to \refE{muf}, which prevents the inverse from being unique. The fundamental reason for this lack of uniqueness is that given $\mu_1$ and $\mu_2$ both satisfying \refE{muf}, we have
\begin{align*}
	\int_x^{f(x)} \pr{\frac{1}{\mu_1(s)} - \frac{1}{\mu_2(s)}} \, ds = 0,
\end{align*}
but this is insufficient to conclude that the integrand vanishes. Of course, the integrand must oscillate, changing sign more and more rapidly as $s$ goes to zero, but even $\mu_1$ and $\mu_2$ strictly increasing cannot prevent this from happening.
} 

Since $\mu(f(x)) = f'(x) \mu(x)$, we have
\begin{align}\label{e:mupfx}
	\mu'(f(x)) = \mu'(x) + \mu(x) \frac{f''(x)}{f'(x)}.
\end{align}
If $f$ is concave then $\mu'(f(x)) \le \mu'(x)$. It does not, however, follow that $\mu'(x)$ is a decreasing function of $x$---for this, we need more information, as in \refL{muGammaConcave}.

\Ignore { 
Whether there is a \textit{practicable} way to characterize those concave acceptable MOC, $f$, for which some corresponding MOC, $\mu$, is concave is an interesting open question that we investigate below, showing how it is related to results in the theory of iteration semigroups developed primarily by Zdun. But for our purposes we can sidestep this issue, using the following result on moduli of continuity:
} 

We now show how the problem of inverting the relation in \refE{muf} to obtain $\mu$ from $f$ is related to iteration theory. We start with the following definition, adapted to our setting from \cite{Zdun1979} (see also  \cite{Zdun1979Also} and Section 3.3 of \cite{Targ1981}):

\begin{definition}\label{D:CIG}
A \textit{continuous iteration group of MOC}  (CIG) is a family, $G = (f^t)_{t \in \R}$, of MOC such that
\begin{enumerate}
	\item
		For all $t > 0$, $f^t(x) > x$.
				
	\item
		For all $s, t$ in $\R$, $f^s \circ f^t = f^{s + t}$.
	
	\item
		$f^0$ is the identity.
		
	\item
		As a map from $\R$ to $[0, \iny)$, $t \mapsto f^t(x)$ is continuous for all $x$ in $[0, \iny)$.
\end{enumerate}
Furthermore (refer to \refDAnd{iotaV}{AcceptableMOC} for definitions of $J$ and $V$),
\begin{itemize}
	\item
If, for all $t > 0$, $f^t$ is locally acceptable with $J(f^t) \ge f^{-t}(a)$ for some $0 < a < \iny$ then we say that $G$ is \textit{locally acceptable} and define $J(G)$ to be the supremum of all such $a$. If $a = \iny$ then we say that $G$ is \textit{acceptable} or \textit{globally acceptable}.

\item
If, for all $t > 0$, $V(f^t) \ge f^{-t}(a)$ for some $0 < a < \iny$ then we say that $G$ is \textit{locally concave} and define $V(G)$ to be the supremum of all such $a$. If $a = \iny$ then we say that $G$ is \textit{concave} or \textit{globally concave}.

\item
We say that $G$ is $C^k$, $k \ge 0$, if $f^t$ is $C^k$ for all $t$ in $\R$.

\item
We say that$f$ is \textit{embedded} in $G$ if $f^1 = f$.
\end{itemize}
\end{definition}

\Ignore { 
We can, however, differentiate once more, giving
\begin{align} 
	\mu''(f(x)) f'(x)
		= \mu''(x) + \mu'(x) \frac{f''(x)}{f'(x)}
			+ \mu(x) \frac{f'(x) f'''(x) - (f''(x))^2}{(f'(x))^2},
\end{align}
so
\begin{align}\label{e:mupp}
	 \mu''(x) 
	 	= \mu''(f(x)) f'(x)
			- \mu'(x) \frac{f''(x)}{f'(x)}
			+ \mu(x) \frac{(f''(x))^2 - f'(x) f'''(x)}{(f'(x))^2}.
\end{align}
} 

\Ignore{ 
Now assume that $f$ is concave with $f'(x) f'''(x) - (f''(x))^2 < 0$. Choose $\mu$ on $[h_0, f(h_0)]$ to be concave as well as strictly increasing and to satisfy \refEAnd{Cond1}{Cond2}, which is always possible. Then all three terms on the right-hand side of \refE{mufppfp} are negative so $\mu$ is concave on $[f(h_0), f(f(h_0)]$. We can continue this inductively to extend $\mu$ to be concave on the interval $[h_0, f^n(h_0)]$. There is no guarantee, however, that $\mu$ will continue to be increasing, so we cannot use this approach to extend $\mu$ to be concave indefinitely. \textbf{Also, for bounded vorticity, where $f(x) = x^\al$, $0 < \al < 1$, we have, $f'(x) f'''(x) - (f''(x))^2 = \al^2(1 - \al) x^{2 \al - 4} > 0$ and approaches $+\iny$ as $x \to 0^+$. So it is more reasonable to assume the opposite sign.}

To summarize, in constructing $\mu$, the relation $\mu(f(h)) = f'(h) \mu(h)$ is forced on us. Under this relation, when $f$ is concave the property of being strictly increasing is inherited to the left, and with an additional restriction on $f$ the property of being concave is inherited to the right, though not indefinitely.
} 

\Ignore{ 
Returning to \refE{mupfx}, for $\mu$ to continue to be increasing on the interval, $[f(h_0), \iny)$, we must have $\mu'(x) + \mu(x) f''(x)/f(x) \ge 0$ on $[h_0, \iny)$. Assuming that $f$ is concave and that $\mu(h_0) > 0$, this condition becomes
\begin{align*}
	\frac{\mu'(x)}{\mu(x)}
		&\ge - \frac{f''(x)}{f'(x)}
		\text{ or }
		\pr{\log \mu(x)}'
		\ge - \pr{\log(f'(x)}',
\end{align*}
both sides of this inequality being nonnegative. Integrating gives
\begin{align*}
	\log \mu(x) \ge \log \mu(h_0) + \log f'(h_0) - \log f'(x),
\end{align*}
or,
\begin{align*}
	\mu(x)
		\ge \mu(h_0) \frac{f'(h_0)}{f'(x)}
		= \frac{\mu(f(h_0))}{f'(x)},
\end{align*}
or
\begin{align*}
	\mu(f(x)) \ge \mu(f(h_0)),
\end{align*}
or
\begin{align*}
	\mu(x) \ge \mu(h_0).
\end{align*}
\textbf{This is an obvious necessary condition! It is not sufficient, however, and that's the point.}
} 

\Ignore{ 
We know from \refT{fConcaveGivesUniqueness} and its proof that if $\mu$ is concave then $\Gamma_t$ defined by \refE{mufGamma} is both the flow associated with the $1$-vector field, $\mu$, as well as being a MOC on the flow (that is, on itself).
} 

Let $\mu$ be a MOC and let $\Gamma$ be the corresponding MOC of the flow given by \refT{MOCsAreFlows}. If we let $f = \Gamma_1$ then because the flows $(\Gamma_t)_{t \in \R}$ form a group under composition, letting $f^t = \Gamma_t$, $f$ is embedded in the CIG, $(f^t)_{t \in \R}$.

Now suppose, starting with only with an acceptable MOC, $f$, that we can find a $C^1$ CIG, $G = (f^t)_{t \in \R}$, for $f$ with each $f^t$ also strictly increasing. Then
\begin{align*} 
	\prt_t f^t(x)
		&= \lim_{h \to 0} \frac{f^{t + h}(x) - f^t(x)}{h}
		= \lim_{h \to 0} \frac{f^t(f^h(x)) - f^t(x)}{h} \\
		&= \lim_{h \to 0} \frac{f^t(f^h(x)) - f^t(x)}{f^h(x) - x}
			\lim_{h \to 0} \frac{f^h(x) - f^0(x)}{h} \\
		&= (f^t)'(x) \prt_s f^s(x)|_{s = 0}.
\end{align*}
Hence the function, 
\begin{align*}
	\mu(x)
		:= \frac{\prt_t f^t(x)}{(f^t)'(x)}
		= \prt_s f^s(x)|_{s = 0}
\end{align*}
is well-defined, with the first expression independent of $t$ in $\R$.

Also, arguing as in the proof of Lemma 5.2 part II of \cite{Zdun1979},
\begin{align*}
	\mu(f^t(x))
		= \lim_{h \to 0} \frac{f^h(f^t(x)) - f^t(x)}{h}
		= \lim_{h \to 0} \frac{f^{t + h}(x) - f^t(x)}{h}
		= \prt_t f^t(x)
\end{align*}
and, applying the chain rule,
\begin{align*}
	\mu(f^t(x))
		&= \pdx{f^s(f^t(x))}{s} \Big\vert_{s = 0}
		= \pdx{f^t(f^s(x))}{s} \Big\vert_{s = 0}
		= (f^t)'(f^0(x)) \pdx{f^t(x)}{s}\Big\vert_{s = 0} \\
		&= (f^t)'(x) \mu(x).
\end{align*}

Letting $I(t, x) = \int_x^{f^t(x)} \frac{dr}{\mu(r)}$, we conclude from these two relations for $\mu(f^t(x)$ that $\prt_t I(t, x) = 1$ and $\prt_x I(t, x) = 0$, and from this it follows that for all $x > 0$,
\begin{align*}
	\int_x^{f^t(x)} \frac{dr}{\mu(r)} = t.
\end{align*}
This includes \refE{muf} in the special case, $t_0 = 1$. (In other words, \refE{muf} is satisfied, where it is now convenient to set $t_0 = 1$, with no loss of generality.) In fact, by \refL{muGammaConcave}, we can invert \refE{muf} to obtain $\mu$ from $f$ if and only if there \textit{is} some such CIG, $G = (f^t)_{t \in \R}$, in which case $\Gamma_t = f^t$.
Finally, observe that $\iota(\mu) = J(G)$.

Combining the results of \refS{BasicMOC}, \refT{InverseMOC}, and \cite{Zdun1979} with the observations above, we have:
\begin{theorem}\label{T:Combined}
Suppose that $f$ is a $C^k$, $k \ge 1$, globally concave globally  acceptable MOC. Then there exists a (in fact, an infinite number of) $C^k$ locally acceptable CIG, $G$, embedding $f$
with $J(G)$ arbitrarily large.
Let $G = (f^t)_{t \in \R}$ be any such CIG embedding $f$ with $a = J(G)$. Then $\mu := \prt_t f^t|_{t = 0}$ is $C^{k - 1}$, satisfies the Osgood condition, and $\iota(\mu) = a$.
Moreover, the following are equivalent:
\begin{enumerate}
	\item
		$\mu$ is concave (on $(0, b))$;
		
	\item
		$G$ is locally concave with $V(G) \ge b$;
		
	\item
		for some $\delta > 0$, $f^t$ is concave (on $(0, f^{-t}(b))$) for all $t$
		in $(0, \delta)$ .
\end{enumerate}
\end{theorem}

\begin{remark}\label{R:DeferedProof}
That $G$ is $C^k$ in \refT{Combined} is proved below following \refE{muhph}.
\end{remark}

Left open in \refT{Combined} is the question of whether any concave acceptable MOC is embeddable in a concave CIG. We cannot prove this, and indeed it may not be true, but for our purposes, the weaker result in \refT{ConcaveCIG} will suffice (see \refR{AnswerToQuestion1}). Before proceeding to the proof of \refT{ConcaveCIG}, however, let us look at some illustrative examples.

\Ignore{ 
\textbf{At this point bring in the book of Kuczma's, Section 1.7 and see theorems on page 436 and 443, of \cite{Kuczma1990} and the book, \cite{Zdun1979}, by Zdun, if the library ever gets it to me. These give some results for conditions on when one can find an iterative family. It seems that looking at $f^{-1}$ is more appropriate in applying their results. If this works out, replace most of what follows in this section.}
} 

Let $a_m = 1/\exp^{m + 1}(1)$ for for any $m = 0, 1, \ldots$ define $\mu_m \colon [0, a_m) \to [0, \iny)$ with $\mu_m(0) = 0$ and
\begin{align}\label{e:mumDef}
	\mu_m(x)
		&= x \theta_m(1/x) \text{ for } 0 < x < a_m,
\end{align}
where $\theta_m$ is defined in \refE{YudovichExamples}.
It is straightforward to verify that $\mu_m$ is continuous, $C^\iny$ on $(0, a_m)$, concave, and increasing (strictly increasing for $m \ge 1$). Extend $\mu$ arbitrarily to $[0, \iny)$ in such a way as to maintain these properties.

\DetailSome{ 
We now show that $\mu$ as defined satisfies the properties we have claimed for it. It is easy to verify that $\mu_m$ satisfies the Osgood condition and that $\lim_{x \to 0^+} \mu_m(x) = 0$, so setting $\mu_m(0) = 0$ makes $\mu_m$ continuous at the origin. For $m = 0$ we can choose $\mu_0(x) = x$ for all $x > 0$. For $m \ge 1$, each $\theta_m$ is strictly increasing and strictly concave and we we will use this to show that each $\mu_m$ is strictly increasing and strictly concave on the interval $(0, 1/\exp^{m + 1}(1))$.

First we calculate $\theta_m'(p)$. We have,
\begin{align*}
	\theta_m'(p)
		&= (\theta_{m - 1}(p) \log^m p)'
		= \theta_{m - 1}'(p) \log^m p + \theta_{m - 1}(p) (\log^m p)',
\end{align*}
and
\begin{align*}
	(\log^m p)'
		= \frac{1}{p \log p \cdots \log^{m - 1} p}
		= \frac{1}{p \theta_{m - 1}(p)}
\end{align*}
so
\begin{align*}
	\theta_m'(p)
		&= \theta_{m - 1}'(p) \log^m p + \frac{1}{p}.
\end{align*}
Arguing inductively, this leads to
\begin{align}\label{e:thetampg}
	\theta_m'(p)
		&= \frac{1}{p} \gamma_m(p),
\end{align}
where $\gamma_0(p) = 0$, $\gamma_1(p) = 1$, and
\begin{align*}
	\gamma_m(p) = \gamma_{m - 1}(p) \log^m p + 1,
\end{align*}
for $m \ge 1$.

Then for $x$ in $(0, \exp^{- (m + 1)}(1))$,
\begin{align}\label{e:mumPrime}
	\begin{split}
	\mu_m'(x)
		&= \theta_m(1/x) + x (\theta_m(1/x))'
		= \theta_m(1/x) - \frac{1}{x} \theta_m'(1/x) \\
		&= \theta_m(1/x) - \frac{1}{x} x \gamma_m(1/x)
		= \theta_m(1/x) - \gamma_m(1/x)
	\end{split}
\end{align}
and
\begin{align*}
	\mu_m''(x)
		&= -\frac{1}{x^2} \brac{\theta_m'(1/x) - \gamma_m'(1/x)}
		= -\frac{1}{x^2} \brac{x \gamma_m(1/x) - \gamma_m'(1/x)}.
\end{align*}
But from \refE{thetampg}, $\gamma_m'(p) = (p \theta_m'(p))' = \theta_m'(p) + p \theta_m''(p)$ so
\begin{align*}
	\mu_m''(x)
		&= -\frac{1}{x^2} \brac{\theta_m'(1/x) - \theta_m'(1/x) - (1/x) \theta_m''(1/x)}
		= \frac{1}{x^3} \theta_m''(1/x)
		< 0,
\end{align*}
since $\theta_m$ is strictly concave. It follows that $\mu_m$ is concave.

\begin{lemma}\label{L:thetagamapLimit}
	For all $m \ge 1$, $p \ge \exp^k(1)$, $k \ge m + 1$, we have
	\begin{align*}
		\theta_m(p) - \gamma_m(p)
			\ge P_m(k) - 1,
	\end{align*}
	where $P_m(k)$ is a polynomial with $P_m(k) \ge k - m$.
\end{lemma}
\begin{proof}
Assume that $p \ge \exp^k(1)$.
	For $m =1 $, $\theta_1(p) = \log p$ and $\gamma_1(p) = 1$, so
	\begin{align*}
		\theta_1(p) - \gamma_1(p)
			\ge \theta_1(\exp^k(1)) - \gamma_1(\exp^k(1))
			> \theta_1(e^k) - \gamma_1(e^k)
			= k - 1
	\end{align*}
	and the result holds for $m = 1$.
	Assuming it holds for $m - 1$, we have
	\begin{align*}
		\theta_m(p) - \gamma_m(p)
			&\ge (\theta_{m - 1}(p) - \gamma_{m - 1}(p)) \log^m(p) - 1 \\
			&\ge (P_{m - 1}(k) - 1) \log^m(\exp^k(1)) \\
			&\ge (k - m + 1) \exp^{k - m}(1) - 1.
	\end{align*}
	But for $k \ge m + 1$, $\exp^{k - m}(1) \ge e \ge 2$ so
	\begin{align*}
		\theta_m(p) - \gamma_m(p)
			&\ge 2 (k - m + 1) - 1
			\ge k - m.
	\end{align*}
\end{proof}

It follows from \refL{thetagamapLimit} that $\theta_m(p) - \gamma_m(p) \ge 1$ for all $p \ge \exp^{m + 1}(1)$ so that by \refE{mumPrime} $\mu_m$ is strictly increasing for all $m \ge 1$.
} 

Now define $f_m^t$ by
\begin{align*} 
	\int_x^{f_m^t(x)} \frac{ds}{\mu_m(s)} = t.
\end{align*}
We can exactly integrate this to give, for $0 < x < a_m$,
\begin{align*}
	\log^{m + 1}(1/f_m^t(x)) - \log^{m + 1}(1/x) = -t,
\end{align*}
whose solution is
\begin{align*} 
	f_m^t(x)
		&=  1/ \exp^m(e^{-t} \log^m(1/x)))
		= F_m(e^{-t} F_m^{-1}(x)),
\end{align*}
where $F_m(x) = 1/\exp^m(x)$. Or, we can write,
\begin{align}\label{e:fmEquals}
	f_m^t(x) = h_m(t + h_m^{-1}(x)),
\end{align}
where
\begin{align*}
	h_m(x) = \frac{1}{\exp^{m + 1}(-x)},
\end{align*}
which we note is strictly increasing.

It is easy to verify directly from \refE{fmEquals} that $G_m = (f_m^t)_{t \in \R}$ is a CIG embedding $f = f_m^1$ and it follows from \refT{Combined} that $G_m$ is a concave CIG. In fact, in our setting, any $C^1$ CIG, $G = (f^t)_{t \in \R}$, must be of the form
\begin{align}\label{e:ftForm}
	f^t(x) = h(t + h^{-1}(x))
\end{align}
for some $h \colon (-\iny, \iny) \to (0, \iny)$. (The function $h$ is called the \textit{generating function} of $G$.) This follows from Theorem 7.1 Chapter I of \cite{Zdun1979} (quoted in Theorem 3.3.29 of \cite{Targ1981}), where it is also proven that $h$ must be strictly increasing with $h(-\iny) = 0$, $h(\iny) = \iny$ and that $h$ is nearly unique in the sense that if $f^t(x) = h_j(t + h_j^{-1}(x))$, $j = 1, 2$, then there exists some $a$ in $\R$ such that $h_1(\cdot) = h_2(a + \cdot)$. (In the terminology of \cite{Zdun1979}, $f$ satisfies property $P3^\circ$ with $\ol{a} = 0$, $\ol{b} = +\iny$.)

By Theorem 1.2 Chapter II of \cite{Zdun1979}, $h$ is differentiable on $(-\iny, \iny)$ and $h'$ never vanishes so, in fact, $h$ is \textit{strictly} increasing. Also, by \refT{Combined},
\begin{align}\label{e:muhph}
	\begin{split}
	\mu(x)
		&= \prt_t f^t(x)|_{t = 0}
		= h'(t + h^{-1}(x))|_{t = 0} \\
		&= h'(h^{-1}(x))
		= \frac{1}{(h^{-1}(x))'},
	\end{split}
\end{align}
which is Lemma 4.2 of \cite{Zdun1979}. The last expression for $\mu$ shows that if $\mu$ is $C^{k - 1}$ then $h$ is $C^k$. Since, by \refT{InverseMOC}, $\mu$ can be chosen to be $C^{k -1}$ if $f$ is $C^k$, this completes the proof of \refT{Combined} promised in \refR{DeferedProof}.

(The last expression for $\mu$ in \refE{muhph} also leads to a direct expression for $h^{-1}$ in terms of $\mu$; namely, $h^{-1}(x) = h^{-1}(a) + \int_x^a (\mu(s))^{-1} \, ds$ for any fixed choice of $a > 0$ and assigned value of $h^{-1}(a)$. This in turn leads to the relation in \refE{ftForm} and to the statement regarding the uniqueness of $h$.)

For $f_m$, we have
\begin{align*}
	h_m^{-1}(x) = - \log^{m + 1}(1/x)
\end{align*}
and
\begin{align*}
	h_m'(x)
		= - \frac{\exp^{m + 1}(-x) \cdots \exp(-x) (-1)}{\exp^{m + 1}(-x)^2}
		= \frac{\exp^m(-x) \cdots \exp(-x)}{\exp^{m + 1}(-x)}
\end{align*}
so
\begin{align*}
	\mu_m(x)
		&= h_m'(h_m^{-1}(x))
		= \frac{\log x \cdots \log^m(x)}{1/x}
		= x \log x \cdots \log^m(x) \\
		&= x \theta_m(1/x),
\end{align*}
in agreement with \refE{mumDef}.

In more generality, we have \refT{GenSoFar}. (See \refD{Concavity} for our distinctions among degrees of concavity.) 

\begin{theorem}\label{T:GenSoFar}
Suppose that $f$ is a $C^k$, $k \ge 3$, concave acceptable MOC embedded in a $C^k$ locally acceptable CIG, $G = (f^t)_{t \in \R}$, given by \refT{Combined} with $a = J(G)$. There exists a $C^k$ generating function, $h \colon (-\iny, \iny) \to (0, \iny)$, for $f^t$ as in \refE{ftForm}, with $h(-\iny) = 0$ and $h(\iny) = \iny$. Any such $h$ must be strictly increasing on $\R$ and strongly strictly convex on $(-\iny, b)$, where $b = h^{-1}(a)$. If $f^t(x) = h_j(t + h_j^{-1}(x))$, $j = 1, 2$, then there exists some $c$ in $\R$ such that $h_1(\cdot) = h_2(c + \cdot)$. For any $t, x > 0$,
\begin{align*}
	\int_x^{f^t(x)} \frac{ds}{\mu(s)} = t,
\end{align*}
where $\mu(x) = h'(h^{-1}(x)) = 1/(h^{-1}(x))'$ is $C^{k - 1}$ and strictly increasing on $(0, a)$.

Let $0 < \ol{a} \le a$ and let $\ol{b} = h^{-1}(\ol{a})$. $G$ (and hence $\mu$) is locally concave with $V(G) = \ol{a}$ if and only if $\log h'$ is concave on $(-\iny, \ol{b})$, and $\mu$ is strongly strictly concave on $(0, \ol{a})$ if and only if $\log h'$ is concave on $(-\iny, \ol{b})$. Finally, if $\mu$ is (strongly strictly) concave on $(0, \ol{a})$ then  on $(-\iny, \ol{b})$, $\log h$ is strictly increasing and (strongly strictly) concave with $(\log h')' \le (\log h)'$, strict inequality holding when $\mu$ is strongly strictly concave.
\end{theorem}
\begin{proof} 
The existence of a generating function $h$ satisfying \refE{ftForm} and possessing the stated properties follows from Theorem 7.1 Chapter I of \cite{Zdun1979}.
By \refT{Combined}, $\mu$ is $C^{k - 1}$, and by the comment following \refE{muhph}, $h$ is $C^k$.

By \refE{ftForm}, $f^t(h(x - t)) = h(x)$, so
\begin{align*}
	(f^t)'(h(x - t)) h'(x - t) &= h'(x).
\end{align*}
But $(f^t)' > 1$ on $f^{-t}(a)$ by \refT{Combined} so we conclude that $h'$ is increasing on $(-\iny, b)$; that is, the condition that $h$ be (non-strictly) convex is already required simply for the CIG to be any strictly increasing CIG embedding $f$, as given by \refT{Combined}. More important, we conclude that $f^t$ is as differentiable as $h$.

Writing \refE{muhph} as
\begin{align}\label{e:muhhp}
	\mu(h(x)) = h'(x)
\end{align}
(which shows that $h$ is strictly increasing on $\R$) we have
\begin{align}\label{e:muphhphpp}
	\mu'(h(x)) h'(x) = h''(x).
\end{align}

Since $h$ is strictly increasing, $\mu$ will be strictly increasing on $(0, \ol{a})$ if and only if $h$ is strictly convex on $(-\iny, \ol{b})$. Taking another derivative gives
\begin{align*}
	\mu'(h(x)) h''(x) + \mu''(h(x)) (h'(x))^2
		= h'''(x)
\end{align*}
so
\begin{align*}
	 \mu''(h(x)) (h'(x))^2
		= h'''(x) - \mu'(h(x)) h''(x)
		= h'''(x) - \frac{(h''(x))^2}{h'(x)},
\end{align*}
or,
\begin{align}\label{e:mupphp}
	\begin{split}
	 \mu''(h(x)) h'(x)
		&= \frac{h'''(x) h'(x) - (h''(x))^2}{h'(x)^2}
		= \pr{\frac{h''(x)}{h'(x)}}' \\
		&= (\log h')''(x).
	\end{split}
\end{align}
Thus, $\mu$ is strictly increasing and (strongly strictly) concave on $(0, \ol{a})$ if and only if $h$ is strictly convex while $\log h'$ is (strongly strictly) concave on $(-\iny, \ol{b})$.

Since $\mu$ is concave on $(0, \ol{a})$, we have $\mu'(x) \le \mu(x)/x$ (with strict inequality if $\mu$ is strongly strictly convex) on $(0, \ol{a})$.
From \refEAnd{muhhp}{muphhphpp},
\begin{align}\label{e:hLogConnection}
	\begin{split}
	\frac{\mu(h(x))}{h(x)}
		&= \frac{h'(x)}{h(x)}
		= (\log h)'(x), \\
	\mu'(h(x))
		&= \frac{h''(x)}{h'(x)}
		= (\log h')'(x),
	\end{split}
\end{align}
so we conclude that $0 < (\log h')' \le (\log h)'$ on $(-\iny, \ol{b})$ with strict inequality when $\mu$ is strongly strictly convex on $(0, \ol{a})$. Differentiating $\refE{hLogConnection}_1$ then substituting $\refE{hLogConnection}_2$ gives
\begin{align}\label{e:loglogpIneq}
	\begin{split}
		(\log h)''(x)
			&= \frac{h(x) \mu'(h(x)) h'(x) - \mu(h(x)) h'(x)}{(h(x))^2} \\
			&= \frac{h'(x)}{(h(x)^2} \brac{\mu'(h(x)) h(x) - \mu(h(x))}
			\le 0,
	\end{split}
\end{align}
again using $\mu'(x) \le \mu(x)/x$ . Thus, if $\mu$ is (strongly strictly) concave on $(0, \ol{a})$ then $\log h$ must be (strongly strictly) concave on $(-\iny, \ol{b})$.
\end{proof} 

\begin{remark}
Much of \refT{GenSoFar} can be obtained assuming only that $(f^t)_{t \in \R}$ is $C^1$, as in Theorem 3.22 of \cite{TargZdun1987}.
\end{remark}

\Ignore{ 
Also, subtracting $\refE{hLogConnection}_1$ from $\refE{hLogConnection}_2$ and differentiating gives
\begin{align*}
	((\log &\, h')'' - (\log h)')'(x) \\
		&= \mu''(h(x)) h'(x)
			- \frac{h'(x)}{(h(x)^2} \brac{\mu'(h(x)) h(x) - \mu(h(x))} \\
		&= \frac{h'(x)}{(h(x)^2} \brac{\mu''(h(x) (h(x))^2 + \mu(h(x)) - \mu'(h(x)) h(x)}
\end{align*}
} 

\Ignore{ 
\begin{remark}
It is important to remember in applying the conditions in \refT{GenSoFar} that it is only the convexity and concavity in a neighborhood of $- \iny$ that is important, since $h(-\iny) = 0$, and it is only a right neighborhood of the origin that matters in assessing the properties of $\mu$. In particular, $h_m$ is convex and strictly increasing only on the interval $(-\iny, 0]$.
\end{remark}
} 

Let $h$, $a$, and $b$ be as in \refT{GenSoFar}. The conclusion in \refT{GenSoFar} that $h(-\iny) = 0$ is equivalent to $\mu$ satisfying the Osgood condition, for a change of variables gives
\begin{align*}
	\int_0^1 \frac{ds}{\mu(s)}
		&= \int_0^1 \frac{ds}{h'(h^{-1}(s))}
		= \int_{h^{-1}(0)}^{h^{-1}(1)} \frac{h'(u)}{h'(u)} \, du
		= h^{-1}(1) - h^{-1}(0),
\end{align*}
which is infinite if and only if $h^{-1}(0) = - \iny$.

\Ignore{ 
(Although \refT{GenSoFar} implies that $(f^t)_{t \in \R}$ is concave if $h$ is convex while $h'$ is concave, this is too restrictive a constraint, for the sequence of examples $h_m$ above fail this condition: $h_m$ and all of its derivatives are convex on $(-\iny, 0]$.)
} 

Since $h' > 0$, we can write $h'$ uniquely in the form
\begin{align}\label{e:gDef}
	h' = e^g
\end{align}
for some function $g \colon \R \to \R$. Then $g = \log h'$ and $h'' = g' e^g$. But $h'' > 0$ on $(-\iny, b)$ by \refT{GenSoFar}, meaning that $g' >0$ on $(-\iny, b)$ and hence $g$ is a strictly increasing function on $(-\iny, b)$.
Again by \refT{GenSoFar}, $G = (f^t)_{t \in \R}$ is locally concave with $0 < V(G) = \ol{a} \le a$ if and only if $g$ is concave on $(-\iny, h^{-1}(\ol{a}))$. Also, $0 = \mu(0) = \mu(h(-\iny)) = h'(-\iny)$ so $g(-\iny) = - \iny$. Thus, we have the following immediate corollary of \refT{GenSoFar}:

\begin{cor}\label{C:ftConvityg}
Let $f$, $G = (f^t)_{t \in \R}$, and $h$ be as in \refT{GenSoFar} with $a = J(G)$. Then $h' = e^g$ for some $C^{k - 1}$ function, $g \colon \R \to \R$, strictly increasing on $(-\iny, h^{-1}(a))$ with $g(-\iny) = -\iny$. Furthermore, $G$ is locally concave with $0 < V(G) = \ol{a} \le a$ if and only if $g$ is concave on $(-\iny, \ol{b})$, where $\ol{b} = h^{-1}(\ol{a})$, and $\mu$ is strongly strictly concave on $(0, \ol{a})$ if and only if $g$ is strongly strictly concave on $(-\iny, \ol{b})$.
\end{cor}

We are now in a position to prove that given a concave acceptable MOC there always exists a larger concave acceptable MOC that is embeddable in a concave CIG: this is \refT{ConcaveCIG}.

\begin{theorem}\label{T:ConcaveCIG}
	Let $f$ be any $C^k$ globally concave globally acceptable MOC, $k \ge 3$.
	Then for any $a > 0$ there exists a  $C^{k + 1}$ globally concave 
	globally acceptable MOC, $\ol{f}$,
	embedded in a $C^{k + 1}$ globally concave globally acceptable  CIG
	with $\ol{f} > f$
	on $(0, a)$. The associated function $\mu$ is concave and $C^k$ and the generating
	function $h$ is $C^{k + 1}$.
	Furthermore, if $f$ is strongly strictly concave then $\mu$ is strongly strictly concave
	on $(0, a)$.
\end{theorem}
\begin{proof}
	Let $G = (f_t)_{t \in \R}$ be any $C^k$ locally acceptable CIG embedding $f$ as given by
	\refT{Combined} and let $h$ be the corresponding $C^k$ generating function
	given by \refT{GenSoFar}. Let $a' = J(G)$ and let $a = h(h^{-1}(a') - 1)$,
	which we note can be made arbitrarily large.
	
By \refE{ftForm}, $f(h(x - 1)) = f^1(h(x - 1)) = h(x)$, so
\begin{align*} 
	f'(h(x - 1)) h'(x - 1) &= h'(x).
\end{align*}
Taking the logarithm of both sides gives
\begin{align}\label{e:logfph}
	\log f'(h(x - 1)) + g(x - 1) = g(x)
\end{align}
since $g = \log h'$. But $f' > 1$ so
\begin{align}\label{e:gIneq}
	g(x) > g(x - 1) \text{ for all } x \text{ in } \R.
\end{align}

Taking the derivative of \refE{logfph} gives
\begin{align*}
	\frac{f''(h(x - 1)) h'(x - 1)}{f'(h(x - 1)} + g'(x - 1) = g'(x).
\end{align*}
But $f'' \le 0$, $f' > 0$, and $h' > 0$ on all of $\R$ so we conclude that
\begin{align}\label{e:gpIneq}
	g'(x) \le g'(x - 1) \text{ for all } x \text{ in } \R
\end{align}
and we note that strict inequality holds if $f$ is strongly strictly concave.

We emphasize that \refEAnd{gIneq}{gpIneq} hold globally for all $x$ in $\R$.

\Ignore{ 
and
\begin{align*}
	f''(h(x - 1)) (h'(x - 1))^2
		+ f'(h(x - 1)) h''(x - 1) = h''(x)
\end{align*}
and thus
\begin{align*}
	f''(h(x - 1)) (&h'(x - 1))^2
		= h''(x) - \frac{h'(x) h''(x - 1)}{h'(x - 1)} \\
		&= h'(x)
			\brac{\frac{h''(x)}{h'(x)} - \frac{h''(x - 1)}{h'(x - 1)}} \\
		&= h'(x) \brac{(\log h')'(x) - (\log h')'(x - 1)}.
\end{align*}
Since $f'' \le 0$ we conclude that 
\begin{align*}
	(\log h')'(x) \le (\log h')'(x - 1).
\end{align*}
But $g = \log h'$, so we can write this inequality as
\begin{align}\label{e:gpIneq}
	g'(x) \le g'(x - 1) \text{ for all } x \text{ in } \R,
\end{align}
and we note that strict inequality holds if $f$ is strongly strictly concave.
} 
		
	For any $x$ in $\R$ let
	\begin{align}\label{e:Ug}
		\ol{g}(x) = \int_x^{x + 1} g(s) \, ds.
	\end{align}
	Since $\ol{g}$ is the mean value of $g$ on $(x, x + 1)$ and $g$ is strictly increasing on
	$(-\iny, h^{-1}(a'))$,
	\begin{align}\label{e:jgj}
		\ol{g}(x) > g(x) > \ol{g}(x - 1) \text{ for all } x \text{ in } I := (-\iny, h^{-1}(a)).
	\end{align}
	By \refE{gIneq}, $\ol{g}'(x) = g(x + 1) - g(x) > 0$ on $\R$ so $\ol{g}$ is
	strictly increasing on $\R$.
	By \refE{gpIneq}, $\ol{g}''(x) = g'(x + 1) - g'(x) \le 0$ on $\R$, so $\ol{g}$ is concave on all of $\R$.
	If $f$ is strongly strictly
	concave then strict inequality holds in \refE{gpIneq} so $\ol{g}$ is strongly strictly concave. In either
	case, we also have $\ol{g}(-\iny) = -\iny$.
	
	Now let
	\begin{align*}
		\ol{h}(x) = \int_{-\iny}^x e^{\ol{g}(s)} \, ds.
	\end{align*}
	It follows from \refE{jgj} that
	\begin{align}\label{e:hInv3Ineq}
		\ol{h}(x) > h(x) > \ol{h}(x - 1) \text{ for all } x \text{ in } I
	\end{align}
	so that also
	\begin{align*}
		h^{-1}(x) < \ol{h}^{-1}(x) + 1
			 \text{ for all } x \text{ in } (0, a) = h(I).
	\end{align*}
	
	Letting
	\begin{align*}
		\ol{f}^t(x) := \ol{h}(t + \ol{h}^{-1}(x))
	\end{align*}
	it follows from \refT{GenSoFar} and \refC{ftConvityg} that $(\ol{f}^t)_{t \in \R}$
	is globally concave and globally acceptable and the
	corresponding function $\ol{\mu}$ is also strictly increasing and concave with
	\begin{align*}
		\int_x^{\ol{f}^tt(x)}  \frac{dr}{\ol{\mu}(r)}
			= t
	\end{align*}
	for all $t, x > 0$.
	Because $\ol{h} > h$ on $I$ and $\ol{h}^{-1}(x) > h^{-1}(x) - 1$ on $(0, a)$, we have
	\begin{align}\label{e:olf2}
		\ol{f}^2(x)
			= \ol{h}(2 + \ol{h}^{-1}(x))
			> h(2 + \ol{h}^{-1}(x))
			> h(1 + h^{-1}(x))
			= f(x)
	\end{align}
	as long as $2 + \ol{h}^{-1}(x) < h^{-1}(a)$ and $x < a$. But by \refE{hInv3Ineq},
	$\ol{h}^{-1} < h^{-1}$ on $(0, a)$ so $2 + \ol{h}^{-1}(x) < h^{-1}(a)$ will hold if
	$2 + h^{-1}(x) < h^{-1}(a)$, which in turn holds if $x < h(2 + h^{-1}(a))$. But $h$ is strictly
	increasing on all of $\R$ by \refT{GenSoFar} so this is a weaker condition than $x < a$, so
	\refE{olf2} holds on $(0, a)$.
	
	If $f$ is strongly strictly concave then so is $\ol{g}$ as observed above and hence $\ol{\mu}$
	is as well by \refC{ftConvityg}.	
	
	Now let $\mu = 2 \ol{\mu}$ and $j^t = \ol{f}^{2t}$. Then
	\begin{align*}
		\int_x^{j^t(x)}  \frac{dr}{\mu(r)}
			= \int_x^{\ol{f}^{2t}(x)}  \frac{dr}{2 \ol{\mu}(r)}
			= \frac{2t}{2}
			= t
	\end{align*}
	so by \refT{Combined}, $(j^t)_{t \in \R}$ is a globally acceptable globally concave CIG embedding
	$j = j^1$ with
	$j = \ol{f}^2 > f$. The smoothness of $\ol{f}$, $(j^t)_{t \in \R}$, and $\mu$ follow from the
	extra level of differentiability given to $\ol{g}$ and hence to $\ol{h}$ by \refE{Ug}.
\end{proof}

\begin{remark}\label{R:AnswerToQuestion1}
	Letting $\Gamma_t = \ol{f}^t$ in \refT{ConcaveCIG} gives an affirmative answer to Question 1
	of \refS{IntroPartII}.
\end{remark}

\Ignore{ 

Finally, we have the following simply corollary:
\begin{cor}\label{C:ConcaveInverseMOC}
	Fix $t_0 > 0$, let $a > 0$ be arbitrarily large,
	and let $f$ be any $C^k$, $k \ge 3$, concave acceptable MOC for a flow.
	There exists an acceptable family, $\Gamma$, of MOC with each $\Gamma_t$, $t \ge 0$,
	$C^k$ and concave,
	and a $C^{k - 1}$ strictly increasing concave MOC, $\mu$, satisfying the
	Osgood condition,
	\refE{muOsgood}, with $\mu > 0$ on $(0, \iny)$, such that
	\begin{align}\label{e:mufConcave}
        		\int_x^{\Gamma_t(x)}  \frac{dr}{\mu(r)}
			= t
	\end{align}
	 for all $x > 0$ with $\Gamma_{t_0} > f$ on $(0, a)$.
	 Furthermore, if $f$ is strongly strictly concave then $\mu$ is strongly strictly concave
	 on $(0, a)$.
\end{cor}
\begin{proof}
Simply apply \refT{ConcaveCIG} (for $t_0 = 1$) and let $\Gamma_t$ be the CIG that results.
\end{proof}
} 

\Ignore { 
\begin{remark}
	In the proof of \refT{ConcaveCIG}, the (not quite unique) generating function
	associated with $\mu$ is
	$\widetilde{h}(x) = \ol{h}(2x)$, since then $\widetilde{h}^{-1}(x) = (1/2) \ol{h}^{-1}(x)$ and
	\begin{align*}
		\widetilde{h}(t + \widetilde{h}^{-1}(x))
			&= \ol{h}(2 (t + (1/2) \ol{h}^{-1}(x)))
			= \ol{h}(2t + \ol{h}^{-1}(x)) \\
			&= \ol{f}^{2t}(x)
			= j^t(x).
	\end{align*}
\end{remark}
} 

\begin{remark}
	If we change the limits of integration in \refE{Ug} to go from $x -1$ to $x$ then we obtain
	a concave CIG embedding a concave function that is less than $f$. This leads to the
	obvious question of whether it is possible to iterate the procedure in the proof of
	\refT{ConcaveCIG}, alternately producing over or underestimates of the previous step,
	to obtain a concave $CIG$ embedding $f$ itself.
\end{remark}

\Ignore{ 
\noindent\textbf{This is a longer version of the remark above; if you see this it means you need
to comment it out before submission:}
\begin{remark}
	The transformations in the proof of \refT{ConcaveCIG} can be viewed as
	\begin{align*}
		h \mapsto g \mapsto \ol{g} \mapsto \ol{h} \mapsto \ol{\mu} \mapsto \mu.
	\end{align*}
	Here we suppress the CIGs, since $(f^t)_{t \in \R}$ is equivalent to $h$, $\ol{\mu}$ is
	equivalent to $(\ol{f})^t_{t \in \R}$, and $\mu$ is equivalent to $(j^t)_{t \in \R}$.
	
	Writing
	\begin{align*}
		\Cal{H} g(x) = \int_{-\iny}^x e^{g(s)} \, ds, \quad
		\Cal{U} h = h' \circ h^{-1}
	\end{align*}
	and
	\begin{align*}
		L g(x) = \int_{x-1}^x g(s) \, ds, \quad U g(x) = \int_x^{x+1} g(s) \, ds
	\end{align*}
	we can write our sequence of transformations as
	\begin{align*}
		h \mapsto \Cal{H}^{-1} h
			\mapsto U \Cal{H}^{-1} h
			\mapsto \Cal{H} U \Cal{H}^{-1} h
			\mapsto \Cal{U}  \Cal{H} U \Cal{H}^{-1} h
			\mapsto 2 \Cal{U}  \Cal{H} U \Cal{H}^{-1} h.
	\end{align*}
	The (not quite unique) $h$ function associated with $\mu = 2 \Cal{U}  \Cal{H} U \Cal{H}^{-1} h$ is
	$\widetilde{h}(x) = \ol{h}(2x)$, since then $\widetilde{h}^{-1}(x) = (1/2) \ol{h}^{-1}(x)$ and
	\begin{align*}
		\widetilde{h}(t + \widetilde{h}^{-1}(x))
			&= \ol{h}(2 (t + (1/2) \ol{h}^{-1}(x)))
			= \ol{h}(2t + \ol{h}^{-1}(x)) \\
			&= \ol{f}^{2t}(x)
			= j^t(x).
	\end{align*}
	Writing this as $\widetilde{h} = \Cal{S} \ol{h}$, where $\Cal{S}$ is the ``scale the argument
	by 2'' transformation, we can write the transformation from the initial generating function, $h$,
	to the final generating function as
	\begin{align*}
		h \mapsto \Cal{S} \Cal{H} U \Cal{H}^{-1} h.
	\end{align*}
	This is a similarity transformation of sorts followed by a rescaling.
	
	If we write $\Cal{F} h$ for the function $h(1 + h^{-1}(x))$ then our transformations start with
	a function $h$ (which generates the CIG, $(f^t)_{t \in \R}$) and returns the generating function,
	$\Cal{S} \Cal{H} U \Cal{H}^{-1} h$, which generates a concave CIG for which
	$\Cal{F} \Cal{S} \Cal{H} U \Cal{H}^{-1} h \ge f$.
	
	If we replace the transformation $U$ with $L$, we obtain a generating function,
	$\Cal{S} \Cal{H} L \Cal{H}^{-1} h$, which generates a concave CIG for which
	$\Cal{F} \Cal{S} \Cal{H} L \Cal{H}^{-1} h \le f$. Or, if we write
	\begin{align*}
		T_U = \Cal{S} \Cal{H} U \Cal{H}^{-1}, \quad
		T_L = \Cal{S} \Cal{H} L \Cal{H}^{-1}
	\end{align*}
	then
	\begin{align*}
		\Cal{F} T_L h \le f \le \Cal{F} T_U h.
	\end{align*}
	
	This leads to the obvious question of whether one can iterate the transformations $T_L$ and
	$T_U$ in such a way as to obtain a concave $CIG$ embedding $f$ itself.
\end{remark}
} 

In \refS{InvertingEulerMOC} we look in detail at the properties of the MOC, $\mu$, that arise when it is a bound on the modulus of continuity of the vector field associated with a solution to the Euler equations. We will find not only that $\mu$ must be concave but also that it must satisfy the additional constraint in $\refE{YudoCond}$.
\AdditionalConstraint {
This will complicate the inversion process we have discussed in this section: these complications are the subject of \refS{muInversionAdditional}.}
{
Whether this additional constraint can be accommodated in \refT{ConcaveCIG} for all $f$ is an open question.
}

\Ignore { 
or,
\begin{align*}
	f''(h(x - 1)) h'(x - 1)
		&= \frac{h''(x) h'(x - 1) - h'(x) h''(x - 1)}{(h'(x - 1))^2} \\
		&= \pr{\frac{h'(x)}{h'(x - 1)}}'
		< 0.
\end{align*}
} 

\Ignore{ 
Fix $M > 0$ arbitrarily and let
\begin{align*}
	\ol{\mu}(r)
		&= \sup \set{\abs{\mu(x) - \mu(y)} \colon x, y \in [0, M], \abs{x - y} \le r}
\end{align*}
define a map from $[0, M]$ to $[0, M]$. Then $\ol{\mu}$ is a MOC that is clearly increasing. To see that it is concave, suppose it is not. Then there exists $a, b$ in $(0, 1)$ and $r \ne s$ in $[0, \iny)$ such that
\begin{align*}
	\ol{\mu}(ar + bs)
		< a \ol{\mu}(r) + b \ol{\mu}(s).
\end{align*}
Now we can write
\begin{align*}
	\ol{\mu}(r) = \sup \set{F(x, y) \colon (x, y) \in K(r)},
\end{align*}
where $F(x, y) = \abs{\mu(x) - \mu(y)}$ is continuous on $[0, M]^2$ and
\begin{align*}
	K(r) = \set{(x, y) \in [0, M]^2 \colon \abs{x - y} \le r}
\end{align*}
is compact. Hence, for any $r$ in $[0, M]$ there exists $x_r, y_r$ in $[0, M]$ such that
\begin{align*}
	\ol{\mu}(r) = \mu(y_r) - \mu(x_r).
\end{align*}
} 

%
\newcounter{bSpeculations} \setcounter{bSpeculations}{0}

\ifthenelse{\value{bSpeculations}=0}{
    }
{ 

\textbf{The rest of this section contains some speculations, with which maybe something can eventually be made.}

Fix $h_0 > 0$ and for any $t \ge 0$ and $a > 0$ define
\begin{align*}
	\Sigma_t(a)
		= \left\{
			\begin{array}{l}
				\mu \in C^1([0, \iny)) \colon \mu(0) = 0, \, \mu(h_0) = a, \\
				\int_h^{f(h)} \frac{dx}{\mu(x)} = t \text{ for all } h > 0
			\end{array}
		\right\}
		.
\end{align*}
Define $\Sigma_t(a)^+$ to be all $\mu$ in $\Sigma_t(a)$ such that $\mu > 0$ on (0, \iny).

Observe that if $\mu, \nu$ are in $\Sigma_t(a)$ or $\Sigma_t(a)^+$ then $\mu(f(h) = f'(h) \mu(h)$ and $\nu(f(h)) = f'(h) \nu(h)$ for all $h > 0$ so, since $\mu(h_0) = \nu(h_0)$, $\mu(f^n(h_0)) = \nu(f^n(h_0))$ for all integers, $n$.

Let $\mu, \nu$ be any two elements in $\Sigma_t^+$. Then for any $h > 0$,
\begin{align*}
	\int_h^{f(h)}
		\pr{\frac{1}{\mu(x)} - \frac{1}{\nu(x)}} = 0.
\end{align*}
Then $g = 1/\mu - 1/\nu$ lies in
\begin{align*}
	\Sigma'
		= \left\{
			\begin{array}{l}
				g \in C^1([0, \iny)) \colon g(0) = g(f^n(h_0)) = 0 \; \forall n \in \Z, \\
				\int_h^{f(h)} g(x) \, dx = 0 \text{ for all } h > 0
			\end{array}
		\right\}
		.
\end{align*}

Observe that $\Sigma'$ is, in fact, a vector space over $\R$ and that
\begin{align}\label{e:etaFuncEq}
	g(f(x)) = g(x)/f'(x) \text{ for all } x > 0.
\end{align}

Solving\MarginNote{Think about the $g(0) = 0$ restriction.}for $\nu$ in terms of $\mu$ and $g $ gives
\begin{align}\label{e:nu}
	\nu = \frac{\mu}{1 - g \mu},
\end{align}
so we can write
\begin{align*}
	\Sigma_t(a)^+
		= \set{ \frac{\mu}{1 - g \mu} \colon
			g \in \Sigma', \, \mu g < 1}.
\end{align*}
These considerations also show that $\Sigma_t(a)^+$ is a path-connected topological subspace of $C^1([0, \iny)$.

We would like to know whether $\Sigma_t(a)^+$ contains any concave functions and, if so, whether that function is unique.

Given that $f$ is concave, to obtain a concave $\mu$ satisfying \refE{muf} we seek a functional on $\Sigma_t(a)^+$ whose extremum is a concave function. Taking inspiration from the $f \equiv 1$ case, which results in periodic functions, we try,
\begin{align*}
	I(\nu)
		= \int_0^{h_0} (\nu'(x))^2 \, dx.
\end{align*}

Now, fix $\mu$ in $\Sigma_t(a)^+$. (We know some such $\mu$ exists by \refT{InverseMOC}.)
If $\nu$ is any element of $\Sigma_t(a)^+$ then from \refE{nu}, for some $g$ in $\Sigma'$,
\begin{align*}
	\nu'
		= \frac{(1 - g \mu) \mu' - \mu (- g \mu' - g' \mu)}{(1 - g \mu)^2}
		= \frac{\mu' + g' \mu^2}{(1 - g \mu)^2},
\end{align*}
so
\begin{align*}
	I(\nu)
		= \int_0^{h_0} \frac{(\mu' + g' \mu^2)^2}{(1 - g \mu)^4}.
\end{align*}

\Ignore{ 
and thus
\begin{align*}
	\nu''
		&= \frac{(1 - g \mu)^2 (\mu'' + 2 g' \mu \mu' + g'' \mu^2) - 2 (\mu' + g' \mu^2)(1 - g \mu)
				(-g \mu' - g' \mu)}
			{(1 - g \mu)^4} \\
		&= \frac{(1 - g \mu) (\mu'' + 2 g' \mu \mu' + g'' \mu^2) + 2 (\mu' + g' \mu^2)
				(g \mu' + g' \mu)}
			{(1 - g \mu)^4} \\
		&= \frac{\mu'' + 2 g' \mu \mu' + g'' \mu^2 - g \mu \mu'' - 2 g g' \mu^2 \mu' - g g'' \mu^3}
			{(1 - g \mu)^3} \\
 		&\qquad+ \frac{2 g (\mu')^2 + 2 g' \mu \mu' + 2 g g' \mu^2 \mu' + 2 (g')^2 \mu^3}
			{(1 - g \mu)^3} \\
		&= \frac{\mu'' + 4 g' \mu \mu' + g'' \mu^2 - g \mu \mu'' - g g'' \mu^3
				+ 2 g (\mu')^2 + 2 (g')^2 \mu^3}
			{(1 - g \mu)^3}.
\end{align*}
} 

Initially, we proceed formally. Fixing $\mu$ in $\Sigma_t(a)^+$,
\begin{align*}
	\nu_\eps
		=  \frac{\mu}{1 - \eps g \mu}
\end{align*}
traces out a curve in $\Sigma_t(a)^+$ (for sufficiently small $\eps$) passing through $\mu$. If $\mu$ is a stationary point of $I$ then $\pdx{I(\nu_\eps)}{\eps}|_{\eps = 0} = 0$. We evaluate,
\begin{align*}
	\pdx{}{\eps} &I(\nu_\eps)|_{\eps = 0} 
		= \int_0^{h_0} \diff{}{\eps} \pr{\frac{\mu' + \eps g' \mu^2}
				{(1 - \eps g \mu)^2}}^2|_{\eps = 0} \\
		&= \int_0^{h_0} 2 (\mu')^2
				\diff{}{\eps} \pr{\frac{\mu' + \eps g' \mu^2}
				{(1 - \eps g \mu)^2}}|_{\eps = 0} \\
		&= 2 \int_0^{h_0} (\mu')^2
			\frac{(1 - \eps g \mu)^2 g' \mu^2 - 2 (\mu' + \eps g' \mu^2) (1 - \eps g \mu) (- g \mu)}
			{(1 - \eps g \mu)^4}
			|_{\eps = 0} \\
		&= 2 \int_0^{h_0} (\mu')^2 (g' \mu^2 + 2 g \mu' \mu) \\
		&= 2 \int_0^{h_0} \mu^2 (\mu')^2 g' + 4 \int_0^{h_0} \mu (\mu')^3 g.
\end{align*}

But,
\begin{align*}
	\int_0^{h_0} \mu^2 (\mu')^2 g'
		= \brac{\mu^2 (\mu')^2 g}_0^{h_0}
			- 2 \int_0^{h_0} g \mu \mu' \brac{(\mu')^2 + \mu \mu''},
\end{align*}
and, formally, the boundary term vanishes (though we need to deal with with 0 endpoint rigorously) so
\begin{align*}
	\pdx{}{\eps} &I(\nu_\eps)|_{\eps = 0} 
		= - 4 \int_0^{h_0} g \mu \mu' \brac{(\mu')^2 + \mu \mu''} + 4 \int_0^{h_0} \mu (\mu')^3 g \\
		&
		= -4 \int_0^{h_0} \mu^2 \mu' \mu'' g.
\end{align*}

If there is enough flexibility in $g$ (I am not sure there is) it follows that
\begin{align}\label{e:muC}
	\mu^2 \mu' \mu'' = C
\end{align}
for some constant, $C$, at a stationary point. Or,
\begin{align*}
	\mu' \mu''
		= \frac{C}{\mu^2}.
\end{align*}

If $\mu'$ is strictly increasing this shows that $\mu$ is either convex or concave depending on the sign of $C$. Given $f$ concave, $\mu$ cannot be convex, so it would have to be concave. \textbf{The problem, of course, is that at the extremum, we need to show that $\mu'$ does not change sign to conclude anything about the sign of $\mu''$. This might be possible via independent means, however, as it seems like the extremum \textit{must} have $\mu'$ always nonnegative.}

\textbf{Also, even if this works, for bounded vorticity, where $\mu(x) = -x \log x$ for small values of $x$ is one valid MOC, $\mu$ does not satisfy \refE{muC}. This would mean that the extremum $\mu$, if it indeed exists, is not the same as the $\mu$ we would expect.}

\Ignore{ 

If there is enough flexibility in $g$ (I am not sure there is) it follows that
\begin{align}\label{e:muC}
	\mu \mu' \brac{(\mu')^2 - \mu \mu''} = C
\end{align}
for some constant, $C$, at a stationary point. Or,
\begin{align*}
	\mu^3 \mu' \pr{\frac{\mu'}{\mu}}'
		= \mu^3 \mu' (\log \mu)''
		= C
\end{align*}
so\begin{align*}
	(\log \mu)'' = \frac{C}{\mu^3 \mu'}.
\end{align*}

If $\mu'$ is strictly increasing this shows that $\log \mu$ is either convex or concave depending on the sign of $C$.
} 

Now, in finding an extremum of $I(\nu_\eps)$, we incorporated the constraint that $\nu_eps$ lies in $\Sigma_t(A)^+$ as long as $\mu$ lies in $\Sigma_t(A)^+$. We still need to enforce the constraint that $\mu$ lies in $\Sigma_t(A)^+$, however. Using \refE{mupfx}, \refE{muC} evaluated at $f(x)$ becomes
\begin{align*}
	\mu(f(x)) &\mu'(f(x)) \brac{(\mu'(f(x)))^2 - \mu(f(x)) \mu''(f(x))} \\
		&= f'(x) \mu(x) \pr{\mu'(x) + \mu(x) \frac{f''(x)}{f'(x)}} \\
		&\qquad \brac{\pr{\mu'(x) + \mu(x) \frac{f''(x)}{f'(x)}}^2 - f'(x) \mu(x) \mu''(f(x))}.
\end{align*}
Substituting \refE{mupp}, the quantity in brackets on the right-hand side expands to
\begin{align*}
	(\mu'(x))^2 &+ 2 \mu(x) \mu'(x) \frac{f''(x)}{f'(x)} + \mu(x)^2 \pr{\frac{f''(x)}{f'(x)}}^2 \\
		&- f'(x) \mu(x) \mu''(x) - \mu(x)\mu'(x) f''(x)
			- \mu(x)^2 f'''(x)
			+ f'(x) \mu(x)^2
\end{align*}
so
\begin{align*}
	\mu(x) &\pr{f'(x) \mu'(x) + \mu(x) f''(x)} \\
		&\qquad
			\times
			\left(
			\begin{array}{l}
				\displaystyle
					(\mu'(x))^2 + 2 \frac{f''(x)}{f'(x)} \mu(x) \mu'(x) + \pr{\frac{f''(x)}{f'(x)}}^2 \mu(x)^2 \\
					\\
				- f'(x) \mu(x) \mu''(x) - f''(x) \mu(x) \mu'(x) - f'''(x) \mu(x)^2 + f'(x) \mu(x)^2
			\end{array}
			\right) \\
		&= C.
\end{align*}

} 

\Ignore{ 
Now, the same argument applies if we integrate over $[f^{-1}(h_0), h_0]$ instead of $[0, h_0]$ and would lead to the conclusion that $I(\mu)$ is extremized when $\mu^2 \mu'' = C$, for in that case we have enough flexibility in the choice of $g$ (probably) to conclude this (using $\int_{f^{-1}(h_0)}^{h_0} g = 0$). Since $\mu > 0$, this says any extremum, if it exists, is either concave or convex.

Now, if we had enough flexibility in the choice of $g$ we could conclude that $I(\mu)$ is extremized when $\mu^2 \mu'' = C$ (using $\int_{f^{-1}(h_0)}^{h_0} g = 0$). Since $\mu > 0$, this says any extremum, if it exists, is either concave or convex. But if $f$ is concave then $\mu$ cannot be convex because $\mu'(f(x)) \le \mu'(x)$, as observed following \refE{mupfx}. It would follow that $\mu$ is convex.

But it is not clear that we have sufficient flexibility in the choice of $g$ to reach this conclusion, since we require that
\begin{align*}
	\int_h^{f(h)} g(x) \, dx = 0 \text{ for all } h \text{ in } (0, f^{-1}(h_0)).
\end{align*}
} 

\Ignore{ 

\refT{InverseMOC} shows that the MOC of a flow can be arbitrarily poor at any given fixed time. These considerations also lead to the questions,
\begin{quote}
	\textbf{Question 2}: Can we invert the relation in \refE{Gammat} for all times $t > 0$ simultaneously
	to obtain a MOC, $\mu$, in terms of $\Gamma_t$?
\end{quote}

\medskip

\begin{quote}
	\textbf{Question 3}: If we add constraints to $\mu$ coming from the nature of the vector field,
	such as those implied by the vector field being a solution to the Euler equations, for what
	MOC of the flow can we invert \refE{Gammat}, even if only for a single time?
\end{quote}

\medskip

To answer Question 2  we need to better understand the properties of $(\Gamma_t)_{t \ge 0}$ as a \textit{family} of MOC in order to be able to invert \refE{Gammat}. That is the subject of \refS{BasicMOC}, and we will see that we need to assume so much about the family to obtain an inverse that obtaining the inverse is quite easy. Issues relating to the uniqueness of such inverses will lead us to consider the case where $\mu$ is concave. And that will naturally lead us to Question 3, since in the context of solutions to the Euler equations, all of the MOC involved are concave, as we show in \refS{InvertingEulerMOC}. This then leads us to Question 4, which we address in \refS{InvertingEulerToGetVectorField}.

\Ignore{ 
Question 4 we also consider only in the context of solutions to the Euler equations, \textbf{add something here.}
|} 

} 

\Ignore{ 
\bigskip
-------------- SPECULATIVE STUFF ---------------

If $\Gamma$ is an acceptable family of MOC for a flow, then \refE{muh} holds, and taking its derivative with respect to $h$ gives
\begin{align*}
	\mu'(h)
		&= \frac{1}{\Gamma_t'(h)^2}
			\brac{\Gamma_t'(h) \prt_t \Gamma_t'(h)
				- \prt_t \Gamma_t(h) \Gamma_t''(h)} \\
		&= \frac{\prt_t \Gamma_t'(h)}{\Gamma_t'(h) }
				- \frac{\prt_t \Gamma_t(h) \Gamma_t''(h)}{\Gamma_t'(h)^2}.
\end{align*}

Now suppose that $\Gamma_t = f$ at some time $t = t_0 > 0$ and that we choose $f$ to be strictly increasing and concave. Then $(\Gamma_t'(h))^{-1} = (f'(h))^{-1}$ is increasing so to insure that $\mu'$ is decreasing, thus making $\mu$ concave, it is sufficient to have $\prt_t \Gamma_t'(h)$ and $-\prt_t \Gamma_t(h) \Gamma_t''(h)$ positive and decreasing in $h$.
} 

%
%
\section{MOC of the Eulerian velocity}\label{S:InvertingEulerMOC}

\noindent We now return to the topic of \refS{YudoTheorem}, where $\mu$ is the MOC of the solution (the velocity) of the Euler equations that is derived from $\theta(p)$, the $L^p$-norms of the solution's vorticity.

To avoid trivialities, we assume throughout that $\theta$ is never zero.

We will write \refE{mu} in the form,
\begin{align}\label{e:A}
	\mu(x)
		&= \inf_{\eps \in \Cal{A}} \set{x^{1 - 2 \eps} \al(\eps)}, \quad
		\Cal{A} = (0, 1/2],
\end{align}
where $\al(\eps) = \eps^{-1} \theta(\eps^{-1})$, as in \refE{thetaal}.

Since $\theta(p) = \smallnorm{\omega^0}_{L^p}$ for some vorticity, $\omega^0$, it inherits some important properties from basic results of measure theory, as in \refL{logThetaConvex}.

\begin{lemma}\label{L:logThetaConvex}
Suppose that $\theta(p)$ is the $L^p$-norm of some function, $f$, lying in $L^p$ for all $p$ in $[p_0, \iny)$ and with $\norm{f}_{L^\iny}$ possibly finite but nonzero. Then $\varphi(p) := p \log \theta(p)$ is convex and $C^\iny$ on $(p_0, \iny)$, and, unless $\theta$ is in $L^\iny$, $\varphi(p)$ is strictly increasing for sufficiently large $p$. Also, $\log \theta(\eps^{-1})$ is convex and $\log \al$ is strictly convex.
\end{lemma}

\begin{proof}

\DetailSome{ 
We have,\MarginNote{The smoothness argument I would not include in submitted form, and perhaps none of the rest either.}
\begin{align*}
	\varphi(p)
		:= \theta(p)^p
			= p \int_0^\iny s^{p - 1} \la(s) \, d s,
\end{align*}
where $\la$ is the distribution function for $f$. The integral above is differentiable with respect to $p$ because, fixing $q$ in $(p_0, \iny)$ and letting $\eps < \min \set{1/2, q - p_0}$, for all $s \ge 1$,
\begin{align*}
	\abs{\diff{}{p} s^{p - 1} \la(s)}
		= \log s \, s^{p - 1} \la(s)
		\le s^p \la (s)
\end{align*}
for all $p$ in the interval $(q - \eps, q + \eps)$. But $s^p$ is in $L^1([1, \iny))$ since $f$ is in $L^{p + 1}(\R^2)$. For $s < 1$,
\begin{align*}
	\abs{\log s \, s^{p - 1} \la(s)} \le C_\eps s^{p_0 - 1} \la(s),
\end{align*}
which is in $L^1([0, 1))$ since $f$ is in $L^{p_0}(\R^2)$. Thus, we have bounded the derivative of the integrand on the interval $(q - \eps, q + \eps)$ by a fixed $L^1([0, \iny))$-function so we can bring the derivative under the integral sign (see, for instance, Theorem 2.27 of \cite{Folland}). This gives the differentiability of $p \log \theta(p) = \log \varphi(p)$ for $p$ in $(p_0, \iny)$. Induction gives $\theta$ in $C^\iny$.
} 

\DetailSomeElse{ 
That $\varphi(p)$ is convex is classical (see, for instance, Exercise 4(b) Chapter 3 of \cite{R1987}), and it follows from this that it is eventually strictly increasing, unless $f$ is in $L^\iny$.
}
{
That $\varphi(p)$ is $C^\iny$ and convex is classical (see, for instance, Exercise 4(b) Chapter 3 of \cite{R1987}), and it follows from this that it is eventually strictly increasing, unless $f$ is in $L^\iny$.
} 

Then
\begin{align*}
	\log \theta(\eps^{-1})
		&= \eps \pr{\eps^{-1} \log \theta(\eps^{-1})}
		= \eps \log \varphi(\eps^{-1})
\end{align*}
so
\begin{align}\label{e:logthetappvarphipp}
	\begin{split}
	(\log \theta&(\eps^{-1}))''
		= \pr{\log \varphi(\eps^{-1}) - \frac{1}{\eps} (\log \varphi)'(\eps^{-1})}' \\
		&= -\frac{1}{\eps^2} (\log \varphi)'(\eps^{-1}) + \frac{1}{\eps^3} (\log \varphi)''(\eps^{-1})
				+ \frac{1}{\eps^2} (\log \varphi)'(\eps^{-1}) \\
		&= \frac{1}{\eps^3} (\log \varphi)''(\eps^{-1}).
	\end{split}
\end{align}
Since $\log \varphi$ is convex it follows that $\log \theta(\eps^{-1})$ is convex. Then since $\log \al(\eps) = -\log \eps + \log \theta(\eps^{-1})$ and $- \log \eps$ is strictly convex, $\log \al$ is strictly convex.
\Ignore{ 
Then
\begin{align*}
	\log \al(\eps)
		&= \log (\eps^{-1} \theta(\eps^{-1}))
		= - \log \eps + \eps \pr{\eps^{-1} \log \theta(\eps^{-1})} \\
		&= - \log \eps + \eps \varphi(\eps^{-1})
\end{align*}
so
\begin{align*}
	(\log \al(\eps))''
		&= \pr{- \frac{1}{\eps} + \varphi(\eps^{-1}) - \frac{1}{\eps} \varphi'(\eps^{-1})}' \\
		&= \frac{1}{\eps^2} - \frac{1}{\eps^2} \varphi'(\eps^{-1}) + \frac{1}{\eps^2} \varphi''(\eps^{-1})
				+ \frac{1}{\eps^2} \varphi'(\eps^{-1}) \\
		&= \frac{1}{\eps^2} + \frac{1}{\eps^2} \varphi''(\eps^{-1}).
\end{align*}
Since $\varphi$ is convex it follows that $\log \al$ is strictly convex.
} 
\end{proof}

\Ignore{ 
As we observed in \refS{YudoTheorem}, $\beta_1$ is a strictly increasing concave function with $\beta_1(0) = 0$ satisfying the Osgood condition, and  $\beta_M(x)/x$ is a strictly decreasing continuous function
with $\lim_{x \to \iny} \beta_M(x)/x = 0$. All of these properties hold for $\mu$, as we show in \refT{muConcave}.
} 

The function, $\la \colon \R \to \R$ defined by
\begin{align}\label{e:la}
	\la(r) := r + \log (\mu(e^{-r}))
\end{align}
will play a large role in what follows as will the function
\begin{align}\label{e:ADef}
	A(x)
		:= \frac{x \mu'(x)}{\mu(x)}
		= x (\log \mu(x))'.
\end{align}
We first establish, in \refP{ConcavityImplications},  some properties of these functions that follow simply from $\mu$ being strictly increasing and concave (properties that we show hold for $\mu$ in \refT{muConcave}).

\begin{prop}\label{P:ConcavityImplications}
Assume that $\mu$ is a strictly increasing (strictly) concave twice continuously differentiable MOC. Then $\log \mu$ is strictly increasing and strongly strictly concave on $(0, \iny)$, $\la$ is (strictly) increasing on $\R$ with $\la' < 1$, $\mu(x)/x$ is (strictly) decreasing on $(0, \iny)$, and $\log \mu(x)/x$ is strictly increasing and strictly concave on $(0, a)$ for some $a > 0$. Also,
\begin{align}\label{e:Ax}
	0 < A(x)
		\le 1 \text{ for all } x > 0
\end{align}
with strict inequality if $\mu$ is strictly concave. Moreover, if $\mu$ satisfies the Osgood condition, \refE{muOsgood}, then
\begin{align}\label{e:limsupA}
	 \limsup_{x \to 0} A(x)
		= 1.
\end{align}
\end{prop}
\begin{proof}
That $\log \mu$ is strictly increasing and strongly strictly concave follows directly from the assumed properties of $\mu$.

\Ignore{ 
Assume that $f$ and $g$ are concave twice differentiable functions defined on appropriate domains with $f$ increasing. Then
\begin{align*}
	(f \circ g)''(x)
		&= \brac{f'(g(x)) g'(x)}'
		= f''(g(x)) (g'(x))^2 + f'(g(x)) g''(x),
\end{align*}
so $f \circ g$ is concave. If $f$ is strictly concave and $g$ strictly monotone or $g$ is strictly concave and $f$ strictly increasing then $f \circ g$ is strictly concave. Because of this, $\log \mu$ is strictly concave and is, of course, (strictly) increasing.
} 

Letting $r = - \log x$, we can write $\la(r)$ variously as
\begin{align}\label{e:larVarious}
	\la(r) = \log \pr{\frac{\mu(x(r))}{x(r)}} = \log(e^r \mu(e^{-r})) = r + \log \mu(e^{-r}).
\end{align}
Then
\begin{align}\label{e:lambdapr}
	\begin{split}
	\la'(r)
		&= \diff{}{x} \log \pr{\frac{\mu(x)}{x}} \diff{x}{r}
		= \diff{}{x} \log \pr{\frac{\mu(x)}{x}} (- e^{-r}) \\
		&= \frac{x}{\mu(x)} \frac{x \mu'(x) - \mu(x)}{x^2}  (-x)
		= \frac{1}{\mu(x)} \pr{\mu(x) - x \mu'(x)} \\
		&= 1 - \frac{x \mu'(x)}{\mu(x)}
		= 1 - A(x).
	\end{split}
\end{align}

Because $\mu$ is strictly increasing, $A(x) > 0$ and $\la' < 1$. 
Because $\mu$ is (strictly) concave, $\mu'$ is (strictly) decreasing so by the mean value theorem,
\begin{align*}
	\mu'(x) \le \frac{\mu(x) - \mu(0)}{x} = \frac{\mu(x)}{x}
\end{align*}
so $A \le 1$ and hence, and equivalently, $\la' \ge 0$ with strict inequalities when $\mu$ is strictly concave. This gives \refE{Ax} and shows that $\la$ is (strictly) increasing.

\Ignore{ 
Also,
\begin{align*}
	\la'(r)
		&=1 - A(x)
		= 1 - e^{-r} \frac{\mu'(e^{-r})}{\mu(e^{-r})}
\end{align*}
so
\begin{align*} 
	\mu'(e^{-r})
		&= e^r \mu(e^{-r}) (1 - \la'(r)).
\end{align*}
Taking another derivative gives
\begin{align*}
	\mu''(e^{-r}) (- e^{-r})
		= \mu(e^{-r}) (-\la''(r)) + \mu'(e^{-r}) (- e^{-r})(1 - \la'(r))
\end{align*}
so
\begin{align*}
	\mu''(e^{-r}) e^{-r}
		= \mu(e^{-r}) \la''(r) + \mu'(e^{-r})(1 - \la'(r)) e^{-r}.
\end{align*}
But $0 \le \la' < 1$ and $\mu' < 0$ so $\mu$ concave means that $\la$ is strongly strictly concave.

Because $\mu$ is strictly increasing, $\mu(e^{-r})$ is strictly decreasing, and because $\mu$ is concave, $\mu'$ is decreasing so $\mu'(e^{-r})$ is increasing. Thus, by \refE{mupmu1la}, $1 - \la'(r)$ must be strictly increasing; that is, $\la$ must be strictly concave.
} 

Now,
\begin{align*}
	\pr{\frac{\mu(x)}{x}}'
		&= \frac{x \mu'(x) - \mu(x)}{x^2}
		\le 0
\end{align*}
by \refE{Ax} with strict inequality when $\mu$ is strictly concave. Thus, $\mu(x)/x$ is (strictly) decreasing.

Finally, let $g(x) = \log \mu(x)/x$. Then $\log \mu(x) = x g(x)$ and
$
	(\log \mu(x))' = g(x) + x g'(x)
$
so $x g'(x) = (\log \mu(x))' - g(x)$. But $\mu(0) = 0$ so $g$ is negative on $(0, a)$ for sufficiently small $a > 0$. Hence, $g' > 0$ on $(0, a)$ since $\log \mu$ is increasing as we showed above. Then,
$
	(\log \mu(x))''
		= 2 g'(x) + x g''(x)
$
so
\begin{align*}
	x g''(x) = (\log \mu(x))'' - 2 g'(x)
\end{align*}
and we conclude that $g$ is strictly concave on $(0, a)$.

To prove \refE{limsupA}, suppose that
\begin{align*}
	0 \le \al := \limsup_{x \to 0} A(x) < 1.
\end{align*}
Then there exists $\eps > 0$ such that $A \le \al$ on $(0, \eps)$. Thus, 
$
	(\log \mu)'(x) \le \al/x
$
on $(0, \eps)$, and integrating from $x < \eps$ to $\eps$ gives
\begin{align*}
	\log \mu(\eps) - \log \mu(x) \le \al \pr{\log \eps - \log x}
\end{align*}
so
\begin{align*}
	\mu(x) \ge \frac{\mu(\eps)}{\eps^\al} x^\al.
\end{align*}
But this means that $\mu$ does not satisfy \refE{muOsgood}. Hence, \refE{limsupA} holds.
\end{proof}

\DetailSomeInline { 
\begin{remark}\label{R:ALimitMust}
\refP{ConcavityImplications} and its proof  tell us that if $\mu$ is concave then $L := \limsup_{x \to 0} A(x) \le 1$, but if the Osgood condition is to hold for $\mu$ we must have $L \ge 1$. Hence, any concave MOC satisfying the Osgood condition must have $L = 1$. As a consequence, if $\mu$ is derived from a MOC for the flow at time $t$, as in \refT{ConcaveCIG}, then we automatically have $L = 1$, since $\mu$ satisfies the Osgood condition (see \refTAnd{InverseMOC}{Combined}).
\end{remark}
} 

We show in \refT{muConcave} that when $\mu$ is given by \refE{A}, we can say more about the functions $\la$ and $A$, as well as about $\mu$ itself. The proof of this theorem relies on determining the value of $\eps$ that minimizes the expression in \refE{A}. This is natural, for as we will see in the next section, $\mu$ can be defined in terms of a Legendre transformation (see \refE{alLegendre}), and finding the minimizing $\eps$ is the usual way to calculate the Legendre transformation for strictly concave functions.

\begin{theorem}\label{T:muConcave}
	Assume that $\log \al$ is strictly convex and twice continuously differentiable.
	The function $\mu$ given by \refE{A} is continuous on $[0, \iny)$ with $\mu(0) = 0$,
	$\mu$ is strictly increasing and concave,
	and $\mu$ is twice continuously differentiable and positive on $(0, \iny)$;
	$\log \mu(x)$ and $\la(r) := r + \log (\mu(e^{-r}))$ are each strictly increasing and
	strictly concave on $(0, \iny)$ and $\R$, respectively;
	$\mu(x)/x$ is strictly decreasing with $\lim_{x \to \iny} \mu(x)/x = 0$;
	and $A$ is strictly decreasing
	with
	\begin{align}\label{e:limxmupmu}
		A(0) := \lim_{x \to 0^+} A(x) = 1.
	\end{align}
	
	Furthermore, if $p \log \theta(p)$ is convex and twice continuously differentiable
	then $\log \al$ is strictly convex and twice continuously differentiable,
	$\la$ is strongly strictly concave, and each of the following equivalent conditions hold:
	\begin{align}\label{e:YudoCond} 
		\begin{split}
		\la''(r) + (\la'(r))^2\ge 0 &\quad\text{ for all } r \text{ in } \R, \\
		x^2 \mu''(x) - x \mu'(x) + \mu(x) \ge 0 &\quad\text{ for all } x > 0, \\
		\mu''(x) \ge \pr{\frac{\mu(x)}{x}}'  &\quad\text{ for all } x > 0.
		\end{split}
	\end{align}
\end{theorem}
\begin{proof}
The expression for $\mu$ in \refE{A} shows that it is continuous on $[0, \iny)$, positive on $(0, \iny)$, and strictly increasing. The function $x \mapsto x^{1 - 2 \eps}$ is concave for all $\eps$ in $[0, 1/2]$, so for any $\gamma$ in $[0, 1]$ and $x, y$ in $(0, \iny)$,
\begin{align*}
	\mu(\gamma x &+ (1 - \gamma) y)
		= C \inf_{\eps \in \Cal{A}} \set{(\gamma x + (1 - \gamma) y)^{1-2 \eps} \al(\eps)} \\
		&\ge C \inf_{\eps \in \Cal{A}} \set{(\gamma x^{1 - 2\eps} + (1 - \gamma) y^{1-2\eps}) \al(\eps)} \\
		&\ge C \gamma \inf_{\eps \in \Cal{A}} \set{x^{1 - 2\eps} \al(\eps)}
			+ C (1 - \gamma) \inf_{\eps \in \Cal{A}} \set{y^{1 - 2 \eps} \al(\eps)} \\
		&= \gamma \mu(x) + (1 - \gamma) \mu(y),
\end{align*}
where $\Cal{A} = (0, 1/p_0]$. It follows that $\mu$ and hence $\log \mu$ is concave.
Because $\mu(x)/x = \inf_{\eps \in \Cal{A}} \set{x^{- 2 \eps} \al(\eps)}$ it strictly decreases to $0$. (Some of these facts also follow from \refP{ConcavityImplications}, given the properties of $\mu$.)

\Ignore { 
Using the final form for $\mu(x)$ in \refE{muForm2},
\begin{align*}
	\mu(e^{-r})
		&= \frac{C}{4} \inf_{\eps \in A} \set{e^{-(1 - 2 \eps)r} 4^\eps \eps^{-1} \theta(\eps^{-1})} \\
		&= \frac{C}{4} \inf_{\eps \in A} \set{e^{2 \eps r} 4^\eps \eps^{-1} \theta(\eps^{-1})}
\end{align*}
so
\begin{align}\label{e:logermu}
	\log(\mu(e^{-r}))
		&= \log C +  \inf_{\eps \in A} \set{2 \eps r + \eps \log 4 - \log \eps + \log \theta(\eps^{-1})}.
\end{align}
Since $2 \eps r$ is concave (though not strictly), $\log(\mu(e^{-r}))$ is concave.
} 

We have, as in \refE{larVarious}, and using \refE{A},
\begin{align}\label{e:lambdar}
	\la(r)
		= \log \pr{\frac{\mu(x)}{x}}
		= \inf_{\eps \in \Cal{A}} \set{- 2 \eps \log x + \log \al(\eps)}
		=  \inf_{\eps \in \Cal{A}} \set{g(r, \eps)},
\end{align}
where
\begin{align*}
	g(r, \eps) = 2 \eps r + \log \al(\eps).
\end{align*}

Since $\log \al$ is strictly convex so is $g(r, \cdot)$. Thus, $g(r, \cdot)$ always achieves it minimum at a unique $\eps = \eps(r)$ with
\begin{align}\label{e:gminCalc}
	\prt_\eps g(r, \eps)|_{\eps = \eps(r)}
		&= 2r + \al'(\eps(r))/\al(\eps(r)) = 0.
\end{align}
Moreover, the function $g$ is twice continuously differentiable in both variables, so $\eps(r)$ is continuously differentiable by the implicit function theorem.

Writing \refE{gminCalc} as
\begin{align}\label{e:logalp2r}
	(\log \al)'(\eps(r)) = - 2r,
\end{align}
since $\log \al$ is strictly convex, $(\log \al)'$ strictly increases, so $(\log \al)'(\eps)$ strictly decreases as $\eps$ decreases. But as $r$ increases, $-2 r$ strictly decreases; hence, $\eps(r)$ is a strictly decreasing function of $r$, and hence also invertible. Moreover, because $(\log \al)'$ strictly increases, we must have $\eps(r) \to 0$ as $r \to \iny$. 

Also from \refE{gminCalc},
\begin{align}\label{e:alODE}
	\al'(\eps(r)) = - 2r \al(\eps(r)).
\end{align}
But,
\begin{align*}
	\la(r) = g(r, \eps(r))
\end{align*}
so
\begin{align}\label{e:laPrimeeps}
	\begin{split}
	\la'(r)
		&= \diff{}{r} g(r, \eps(r))
		= \prt_r g(r, \eps(r)) + \prt_\eps g(r, \eps(r)) \eps'(r) \\
		&= \prt_r g(r, \eps(r))
		= 2 \eps(r)
	\end{split}
\end{align}
by \refE{gminCalc} so $\la$ is strictly increasing. Since $\eps$ is strictly decreasing it follows that $\la$ is strictly concave. And because $\eps$ is continuously differentiable, $\la$, and hence $\mu$, are twice continuously differentiable.

\Ignore{ 
Also, since $x \mapsto \la(r(x))$ is strictly decreasing by \refT{muConcave}, $\la'(r(x)) = \diff{}{x}(\la(r(x))/r'(x) = - x \diff{}{x}(\la(r(x))/r'(x) > 0$, so $\la$ is strictly increasing and $\eps > 0$.
} 

\Ignore { 
(So we can also say, treating $\eps$ as a function of $x$ instead of $r$, that
\begin{align}\label{e:epsx}
	\eps(x)
		= \frac{1}{2} - \frac{x}{2} \frac{\mu'(x)}{\mu(x)}.
\end{align}
This can be a more convenient way of calculating $\eps(r) = \eps(r(x))$, but we still need to express $\eps$ in terms of $r$, not $x$.)
} 

\Ignore{ 
Because $\mu$ is concave, $\mu'$ is decreasing so by the mean value theorem,
\begin{align*}
	\mu'(x) \le \frac{\mu(x) - \mu(0)}{x} = \frac{\mu(x)}{x}
\end{align*}
so $A(x) \le 1$ and hence $\la'(r) \ge 0$. This shows both that $\la$ is increasing and that $\eps(r) \ge 0$.
But also,
\begin{align*}
	\la(r) = \log \pr{\frac{\mu(x(r))}{x(r)}} = \log(e^r \mu(e^{-r})) = r + \log \mu(e^{-r}),
\end{align*}
so $\la$ is concave and $\eps' \le 0$.

To obtain strict concavity of $\la$, we must rule out $\la''(r) = \eps'(r) = 0$. Differentiating both sides of \refE{alODE} with respect to $r$ gives
\begin{align*}
	\al''(\eps(r)) \eps'(r) = - 2 \al(\eps(r)) - 2r \al'(\eps(r)) \eps'(r).
\end{align*}
This shows that if $\eps'(r) = 0$ then $\al(\eps(r)) = 0$. But $\al$ never vanishes because we assumed that $\theta$ never vanishes. Hence, $\la$ is strictly concave and it follows from this that $\eps$ strictly decreases with $r$ and is invertible. Also, $\eps(r) > 0$ since if $\eps(r_0) = 0$ we would have $\eps(r) = 0$ for all $r \ge r_0$ and $\eps'$ would vanish.
} 

But \refEAnd{lambdapr}{laPrimeeps} give
\begin{align*}
	\eps(r) = \frac{1}{2} \pr{1 - A(x(r))},
\end{align*}
so $A(x(r)) \to 1$ as $r \to \iny$ or $A(x) \to 1$ as $x \to 0^+$, giving \refE{limxmupmu} (this also follows from \refE{limsupA}). This also shows that $A$ is strictly decreasing.

\Ignore{ 
It follows as in \refR{ConcavityImplications} that the strict concavity of $\la$ gives the strict concavity of $\log \mu$. It is also easy to see that the concavity of the strictly increasing $\mu$ along with the strict concavity of $\log \mu$ gives the strict concavity of $\mu$.
} 

Now assume that $p \log \theta(p)$ is convex. From the proof of \refL{logThetaConvex}, we see that $\log \al$ is strictly convex and so all of the conclusions above, in particular, that $\la$ is strictly concave and $\eps$ is invertible, hold. We also have, as in the proof of \refL{logThetaConvex}, that
\begin{align}\label{e:phieps}
	\phi(\eps) := \log \theta(1/\eps) = \log \eps + \log \al(\eps)
\end{align}
is convex.

Letting $\eta \colon (0, 1/p_0] \to (0, \iny)$ be the inverse of the map, $r \mapsto \eps(r)$, we have,
using \refE{alODE},
	\begin{align}\label{e:phipp}
		\begin{split}
		\phi''(\eps)
			&= - \frac{1}{\eps^2} + \pr{\frac{\al'(\eps)}{\al(\eps)}}' 
			= - \frac{1}{\eps^2} - 2 (\eta(\eps))' \\
			&= - \frac{1}{\eps^2} - \frac{2}{\eps'(\eta(\eps))}
			=  - \frac{1}{\eps^2} - \frac{4}{\la''(\eta(\eps))}.
		\end{split}
	\end{align}
	Or, expressed in the variable $r$ and using $\eps(r) = (1/2) \la'(r)$,
	\begin{align*}
		\phi''(\eps(r))
			= - \frac{4}{(\la'(r))^2} - \frac{4}{\la''(r)}.
	\end{align*}
	By \refL{logThetaConvex}, $\phi$ is convex and we conclude that
	$1/\la''(r) + 1/(\la'(r))^2 \le 0$ so that $\la''(r) < 0$ and
	
	\begin{align*}
		\frac{\la''(r) + (\la'(r))^2}{\la''(r) \la'(r)} \le 0.
	\end{align*}
	Then since $\la''(r) \la'(r) < 0$, $\refE{YudoCond}_1$ holds.
	
	It remains to show the equivalence of the three conditions in \refE{YudoCond}.
Let
$
	r = - \log x
$
as in the proof of \refP{ConcavityImplications}. Starting with \refE{lambdapr} one can show that
\begin{align}\label{e:lambdappr}
	\la''(r)
		= x^2 \frac{\mu''(x)}{\mu(x)} - x^2 \pr{\frac{\mu'(x)}{\mu(x)}}^2 + x \frac{\mu'(x)}{\mu(x)}.
\end{align}
Then from \refEAnd{lambdapr}{lambdappr},
\begin{align*}
	\la''(r) + (\la'(r))^2
		&=  x^2 \frac{\mu''(x)}{\mu(x)} - x^2 \pr{\frac{\mu'(x)}{\mu(x)}}^2 + x \frac{\mu'(x)}{\mu(x)}
			+ \pr{1 - \frac{\mu'(x)}{\mu(x)} x}^2 \\
		&= 1 - \frac{x \mu'(x)}{\mu(x)} + \frac{x^2 \mu''(x)}{\mu(x)}
		= \frac{x^2 \mu''(x) - x \mu'(x) + \mu(x)}{\mu(x)}.
\end{align*}
This gives the equivalence of $\refE{YudoCond}_1$ and $\refE{YudoCond}_2$, and a simple calculation shows that $\refE{YudoCond}_3$ is a re-expression of $\refE{YudoCond}_2$.
(Note that integrating $\refE{YudoCond}_3$ does not contradict \refE{Ax}, because the concavity of $\mu$ means that neither $\mu(x)/x$ nor $\mu'(x)$ converges to $0$ as $x \to 0$.)
\end{proof}

\Ignore{ 
Now suppose that $e^r \mu(e^{-r})$ is not strictly concave. Then letting $f(\eps) =  - \eps \log 4 - \log \eps + \log \theta(\eps^{-1})$,
\begin{align*}
	&\inf_{\eps \in \Cal{A}} \set{2 \eps (\la r + (1 - \la) s) + f(\eps)} \\
		&\qquad= \la \inf_{\eps \in \Cal{A}} \set{2 \eps r + f(\eps)}
		 	+ (1 - \la) \inf_{\eps \in \Cal{A}} \set{2 \eps s + f(\eps)}
\end{align*}
for some $\la$ in $(0, 1)$, $r < s$. Letting $\eps_1$, $\eps_2$, and $\eps_3$ be the three (not necessarily unique) values of $\eps$ at which the three infimums above are reached, we have
\begin{align*}
	2 \eps_1 &(\la r + (1 - \la) s) + f(\eps_1) \\
		&= \la \pr{2 \eps_2 r + f(\eps_2)}
		 + (1 - \la) \pr{2 \eps_3 s + f(\eps_3)}
\end{align*}
so 
\begin{align*}
	f(\eps_1)
		&= \la \pr{2 (\eps_2 - \eps_1) r + f(\eps_2)}
		 + (1 - \la) \pr{2 (\eps_3 - \eps_1) s + f(\eps_3)} \\
		&= 2 \pr{\la (\eps_2 - \eps_1) r + (1 - \la) (\eps_3 - \eps_1) s}
			+  \la f(\eps_2) + (1 - \la) f(\eps_3).
\end{align*}
When $\eps_1 = \eps_2 = \eps_3$ this reduces to $f(\eps_1) = \la f(\eps_1) + (1 - \la) f(\eps_1)$, an equality that holds. We claim that there is no other possibility.

Now suppose that $\eps_2 = \eps_3$. Then
\begin{align*}
	f(\eps_1)
		&= 2  (\eps_2 - \eps_1) \pr{\la r + (1 - \la) s}
			+  f(\eps_2).
\end{align*}
} 

\Ignore{ 
Now, $\beta_1$ being concave is equivalent to saying that it is
absolutely continuous for any open interval, $[a, b]$ for $0 < a < b < \iny$ and that
$\beta_1'$ is defined almost everywhere and is decreasing on the set on which it is defined. (See, for instance, exercise 42 Section 3.5 of \cite{Folland}.) Therefore,
from \refE{mu},
\begin{align*}
	\mu'(x)
		&= - \frac{C}{x^2} \beta_1(x^2/4) + \frac{C}{x} \beta_1'(x^2/4) \frac{2x}{4}
		= C \pr{\frac{\beta_1'(x^2/4)}{2} - \frac{\beta_1(x^2/4)}{x^2}}
\end{align*}
so
\begin{align*}
	\mu'(\sqrt{x}/2)
		= C \pr{\frac{\beta_1'(x)}{2} - \frac{\beta_1(x)}{x}},
\end{align*}
or,
\begin{align*}
	\frac{1}{4 \sqrt{x}} \mu''(\sqrt{x}/2)
		= C \pr{\frac{\beta_1''(x)}{2} - \frac{\beta_1'(x)}{x} + \frac{\beta_1(x)}{x^2}}.
\end{align*}

Because $\mu$ is concave, it can have at most one maximum. But $\mu(0) = 0$ and $\mu$ is positive so it starts by strictly increasing from 0 at the origin. If it has a maximum at, say, $a$, then it decreases for all $x > a$. But $\mu'$ is decreasing, so this would require that $\mu$ eventually become negative, which it cannot. We conclude, therefore, that $\mu$ strictly increases, reaching no maximum value (except possibly at $ x = \iny$).
} 

\DetailSomeInline { 
\begin{remark}
As\MarginNote{Should this be in the submitted form?}a curious aside, note that two independent solutions to $x^2 \mu''(x) - x \mu'(x) + \mu(x)  = 0$ are $\mu_0(x) = x$ and $\mu_1(x) = - x \log x$, the first two examples in \refE{mumDef}. The first, $\mu_0$, will never arise from \refE{A} unless $\theta(p) \to 0$ as $p \to \iny$, while $\mu_1$ is an upper bound (for small $x$) on the $\mu$ that results from setting $\theta = constant$ in \refE{A}. In this sense, $\mu_1$ is the smallest possible $\mu$ that can result from \refE{A}.
\end{remark}
} 

\Ignore{ 
Also, assuming that $\log \al$ is strictly convex and twice continuously differentiable, then $1 - \frac{x \mu'(x)}{\mu(x)} \to 0$ as $x \to 0^+$ by \refE{limxmupmu} and a \textit{necessary} condition for $\refE{YudoCond}_1$ to hold is that
\begin{align*}
	\limsup_{x \to 0^+} \frac{x^2 \mu''(x)}{\mu(x)} \le 0.
\end{align*}
} 

\Ignore{ 
\begin{remark}\label{R:OsgoodEquivalence}
	If $\mu$ is strictly increasing and concave with $\mu(0) = 0$ then, as in the proof of
	\refP{ConcavityImplications},
	\begin{align*}
		0 < \frac{\mu'(x)}{\mu(x)} x \le 1.
	\end{align*}
	Thus,
	\begin{align*}
		\int_0^1 \frac{dx}{x \mu'(x)}
			\ge \int_0^1 \frac{dx}{\mu(x)},
	\end{align*}
	and the first integral is infinite if $\mu$ satisfies the Osgood condition. On the other hand,
	if $\mu$ satisfies \refE{limxmupmu}, as it does when $\theta(p)$ is the $L^p$-norm of
	some initial vorticity, then there exists $a > 0$ such that $x \mu'(x) \ge
	(1/2) \mu(x)$ for all $x < a$, so
	\begin{align*}
		\int_0^a \frac{dx}{\mu(x)}
			\ge \frac{1}{2} \int_0^a \frac{dx}{x \mu'(x)}.
	\end{align*}
	Thus, if
	\begin{align}\label{e:OsgoodLike}
		\int_0^1 \frac{dx}{x \mu'(x)} = \iny
	\end{align}
	then $\mu$ satisfies the Osgood condition.
	
	That is, when $\theta(p)$ is the $L^p$-norm of some initial vorticity,
	\refE{OsgoodLike} and the Osgood condition, \refE{muOsgood},  are equivalent conditions.
\end{remark}
} 

%
\section{Yudovich velocity fields are Dini-continuous}\label{S:Dini}

\noindent We say that $\mu$ is a Dini MOC if the MOC $S_\mu \colon [0, \iny) \to [0, \iny)$ defined by
\begin{align}\label{e:Smu}
	S_\mu(x)
		= \int_0^x \frac{\mu(s)}{s} \, ds
\end{align}
exists (that is, if the integral is finite for any $x > 0$ and hence for all $x > 0$ and $S_\mu$ is as in \refD{MOC}). If a function has a Dini MOC we say that the function is Dini-continuous.

The function $S_\mu$ can be used to re-express $\refE{YudoCond}_3$ as $S_\mu''(x) \le \mu''(x)$, meaning that $S_\mu$ is more strictly convex than $\mu$.

It is shown in \cite{Sueur2010} that the sequence of example Yudovich velocity fields derived from \refE{YudovichExamples} are Dini-continuous. In fact, it follows from \refP{Dini} that all Yudovich velocity fields are Dini-continuous and the MOC $\mu$ and $S_\mu$ are essentially the same, as we show in \refP{Dini}. These are perhaps the most significant properties of Yudovich velocity fields.
(Also see \refR{RemarkDini}.)

It follows trivially from \refE{Ax} that $\mu \le S_\mu$ and $\mu' \le S_\mu'$ for any concave MOC, $\mu$. (Allowing that $S_\mu$ may be infinite and in that case defining $S_\mu'$ to be infinite.) In addition, for Yudovich velocity fields, $\mu'' \ge S_\mu''$, a consequence of \refL{logThetaConvex}, \refT{muConcave}, and \refP{Dini}.

\begin{prop}\label{P:Dini}
Assume that $\mu$ is a strictly increasing (strongly) strictly concave Osgood MOC and that \refE{limxmupmu} holds. Then $\mu$ is Dini-continuous while $S_\mu$ is strictly increasing and (strongly) strictly concave, $S_\mu$ lies in the same germ as $\mu$, and $S_\mu'$ lies in the same germ as $\mu'$ at the origin.
Moreover, the equivalent of $\refE{YudoCond}$  holds for $S_\mu$ if it holds for $\mu$.
\end{prop}
\begin{proof}
Assume that \refE{limxmupmu} holds.
It follows that for some $x > 0$ we have $s \mu'(s) \ge (1/2) \mu(s)$ for all $s < x$. Hence,
\begin{align*}
	S_\mu(x)
		= \int_0^x \frac{\mu(s)}{s} \, ds
		\le 2 \int_0^x \mu'(s) ds
		= 2 \mu(x)
		< \iny.
\end{align*}
That is, $\mu$ must satisfy not only the Osgood condition, \refE{muOsgood}, but the Dini condition.
Then since $S_\mu(0)$ must be zero we can apply L'Hospital's rule to conclude that
\begin{align*} 
	\lim_{x \to 0} \frac{S_\mu(x)}{\mu(x)}
		&= \lim_{x \to 0} \frac{S_\mu'(x)}{\mu'(x)}
		= \lim_{x \to 0} \frac{\mu(x)}{x \mu'(x)}
		= \lim_{x \to 0} \frac{1}{A(x)}
		= 1.
\end{align*}
Also, $S_\mu'(x) = \mu(x)/x$ is strictly decreasing by \refP{ConcavityImplications} so $S_\mu$ is strictly concave. (And if $\mu$ is strongly strictly concave then so too is $S_\mu$.)

\Ignore{ 
If $\refE{YudoCond}_1$ holds then $\la$ is strictly concave, which is equivalent to $A$ being strictly decreasing by virtue of \refE{lambdapr}. Combined with \refE{limsupA} this gives that \refE{limxmupmu} holds and hence that all of the conclusions above continue to hold. Also, as observed above, $\refE{YudoCond}_1$ is equivalent to $S_\mu$ being more strictly convex than $\mu$.
} 

Now assume that $\refE{YudoCond}$ holds. Then the the equivalent of $\refE{YudoCond}_3$  holds for $S_\mu$ if and only if $S_{S_\mu}'' - S_\mu'' \le 0$, as was observed above for $\mu$.
Then
\begin{align*}
	S_{S_\mu}'' - &S_\mu'' \le 0
		\iff \pr{\int_0^x \frac{S_\mu(s)}{s} \, ds}'' - \pr{\int_0^x \frac{\mu(s)}{s} \, ds}'' \le 0 \\
		&\iff \pr{\frac{S_\mu(x)}{x}}' - \pr{\frac{\mu(x)}{x}}' \le 0 \\
		&\iff x S_\mu'(x) - S_\mu(x) - (x \mu'(x) - \mu(x)) \le 0 \\
		&\iff \mu(x) - S_\mu(x) - x \mu'(x) + \mu(x) \le 0 \\
		&\iff j(x) := S_\mu(x) + x \mu'(x) - 2 \mu(x) \ge 0.
\end{align*}
By \refE{limxmupmu}, $j(0) = 0$, and we have
\begin{align*}
	j'(x)
		&= \frac{\mu(x)}{x} + \mu'(x) + x \mu''(x) - 2 \mu'(x) \\
		&= \frac{x^2 \mu''(x) - x \mu'(x) + \mu(x)}{x}.
\end{align*}
Then $j'(x) \ge 0$ for all $x > 0$ if $\refE{YudoCond}_2$ holds, and it follows that $j(x) \ge 0$ for all $x > 0$.
\end{proof}

\begin{remark}\label{R:RemarkDini}

It is shown in \cite{Burch1978} by Charles Burch that if $v$ is a velocity field with a concave MOC, $\mu$, then $R v$ has a MOC, $\nu$, given by
\begin{align*}
	\nu(x) = c \pr{S_\mu(x) + x \int_x^\iny \frac{\mu(s)}{s^2} \, ds}.
\end{align*}
Here, $R$ can be a Riesz transform.
(This result appears in an earlier form as Lemma 1 of \cite{Shapiro1976} by Victor Shapiro. Also see \cite{KNV2007}.)

%
%
\Ignore{ 

In \cite{KNV2007}, this result is applied under the assumption that $\mu'(0)$ is finite (though $\mu''(0) = - \iny$), but it is still applicable, as we now show, for Yudovich velocities, where we always have $\mu'(0) = +\iny$.

For a Yudovich velocity, $v$ lies in $L^\iny([0, \iny) \times \R^2)$ so we can choose $\mu$ to be bounded. Then
\begin{align*}
	x \int_x^\iny \frac{\mu(s)}{s^2} \, ds
		&= x \int_1^\iny \frac{\mu(s)}{s^2} \, ds + x \int_x^1 \frac{\mu(s)}{s^2} \, ds \\
		&= Cx + x \int_x^1 \frac{\mu(s)}{s^2} \, ds.
\end{align*}
By \refP{ConcavityImplications}, $\mu(s)/s$ is decreasing, so
\begin{align*}
	x \int_x^1 \frac{\mu(s)}{s^2} \, ds
		\le x \frac{\mu(x)}{x}  \int_x^1 \frac{1}{s} \, ds
		= - \mu(x) \log x.
\end{align*}
Hence,
\begin{align*}
	\nu(x) \le C_0 S_\mu(x) + Cx - C \mu(x) \log x.
\end{align*}

By L'Hospital's rule and using \refE{limxmupmu},
\begin{align*}
	\lim_{x \to 0^+} (-\mu(x) \log x)
		&= -\lim_{x \to 0^+} \frac{\log x}{\mu(x)^{-1}}
		= \lim_{x \to 0^+} \frac{1/x}{\mu'(x)/\mu(x)^2} \\
		&= \lim_{x \to 0^+} \frac{\mu(x)}{A(x)}
		= 0.
\end{align*}
} 

For a Yudovich velocity, $v$ lies in $L^\iny([0, \iny) \times \R^2)$, so we can choose $\mu$ to be bounded, making $\nu(x)$ finite for all $x > 0$. Since $\nu'(x) = c \int_x^\iny \mu(s) s^{-2} \, ds$, $\nu$ is strictly increasing. That $\nu(0) = 0$ then follows directly if $\nu'(0) < \iny$ and by applying L'Hospital's rule, otherwise.
Noting that $\nu''(x) = - c \mu'(x)/x^2 < 0$, we see that $\nu$ is strictly concave. It is, in general, neither Osgood nor Dini, as we can see by looking at $\mu_2$ of \refE{mumDef}. (For bounded vorticity, however, which corresponds to $\mu_1$, $\nu$ is both Osgood and Dini.)
\end{remark}

\Ignore { 
\begin{remark}\label{R:muSandwich}
From \refE{limxmupmu} for any $\al < 1$ there exists $\eps > 0$ for which $\al <x (\log \mu)'(x) \le 1$ on $(0, \eps)$. Integrating leads to the inequality,
\begin{align*}
	\mu(\eps) \frac{x}{\eps}
		\le \mu(x)
		\le \mu(\eps) \pr{\frac{x}{\eps}}^\al.
\end{align*}
which shows that as $x \to 0$, $\mu(x)$, properly rescaled, becomes closer and closer to being linear.
\end{remark}
} 

\Ignore { 
Since
\begin{align*}
	\frac{d^2}{d r^2} \log(\mu(e^{-r}))
		&= - \diff{}{r} \brac{e^{-r} (\log \mu)'(e^{-r})} \\
		&= e^{-r} (\log \mu)'(e^{-r})
			+ e^{-2r} (\log \mu)''(e^{-r}),
\end{align*}
if $\mu$ is increasing and $\log \mu(e^{-r})$ is concave then $\log \mu$ is concave.
}

\Ignore{ 
\begin{remark}
We also have,
\begin{align*}
	\eps(r)
		&= \frac{1}{2} \la'(r)
			= \frac{1}{2} \brac{r + \log \mu(e^{-r})}'
			= \frac{1}{2} \brac{ 1 - \frac{\mu'(e^{-r})}{\mu(e^{-r})} e^{-r}} \\
		&= \frac{1}{2} \brac{1 - A(x)},	
\end{align*}
so $0 \le \eps(r) \le 1$. Of course, also, $\eps(r) \le 1/p_0$.
\end{remark}
} 

\Ignore{ 
Since $\mu(x) = x e^{\la(- \log x)}$, we have
\begin{align*}
	\mu'(x)
		&= e^{\la(- \log x)} + x \la'(- \log x) (-1/x) e^{\la(- \log x)} \\
		&= (1 - \la'(- \log x)) e^{\la(- \log x)}
\end{align*}
so $\mu$ strictly increasing implies that $\la' < 1$. 

Also,
\begin{align*}
	\mu''(x)
		&= (1 - \la'(- \log x)) \la'(- \log x)(-1/x) e^{\la(- \log x)} \\
		&\qquad
			- \la''(- \log x)(-1/x) e^{\la(- \log x)} \\
		&= \frac{1}{x} \brac{\pr{\la'(- \log x) - 1} \la'(- \log x) + \la''(- \log x)}e^{\la(- \log x)}.
\end{align*}
Thus $\la$ increasing and concave gives $\mu$ concave and adding strictness to either the increasing or concave nature of $\la$ gives $\mu$ strictly concave.
} 

\Ignore{ 
In light of \refT{muConcave}, if we wish to invert the relation in \refE{mufGamma} to obtain $\mu$ we must insure that $\mu$ has at least some of the key properties given by that lemma. But\MarginNote{Must we?}we must also insure that the function $\beta_1$ has these properties. Fortunately, this comes for free by virtue of \refL{betaFrommu}.

\begin{lemma}\label{L:betaFrommu}
	Assume that\MarginNote{Some condition don't hold, I think, as the Dini condition does not. Look into this.}
	the function $\mu$ satisfies \refE{mufGamma}
	and has all the properties stated in
	\refT{muConcave}. Then so too does the function $\beta_1$ given by inverting
	the relation in \refE{mu}.
\end{lemma}
\begin{proof}
By \refE{mu},
\begin{align}\label{e:muToBeta}
	\beta_1(x^2/4) = C x \mu(x).
\end{align}
Thus, $\beta_1$ is continuous on $[0, \iny)$, $\beta_1(0) = 0$, $\beta_1$ is strictly increasing, and
\begin{align*}
	\frac{1}{x^2/4} \beta_1(x^2/4)
		=  4 C \frac{\mu(x)}{x},
\end{align*}
which is strictly decreasing so $\beta_1(x)/x$ is strictly decreasing.

Taking the derivative of \refE{muToBeta},
\begin{align*}
	\frac{x}{2} \beta_1'(x^2/4) = C \mu(x) + C x \mu'(x),
\end{align*}
or,
\begin{align*}
	\beta_1'(x^2/4) = 2 C \pr{\frac{\mu(x)}{x}  + \mu'(x)},
\end{align*}
which is strictly positive for all $x > 0$, showing again that $\beta_1$ is strictly increasing. Also, $\mu(x)/x$ and $\mu'(x)$ are both strictly decreasing, so $\beta_1'$ is decreasing, showing that $\beta_1$ is concave. The other facts are proved similarly.
\end{proof}
} 

\Ignore{ 

Taking another derivative,
\begin{align*}
	\beta_1''(x^2/4)  \frac{x}{2}
		= 2 C \pr{\frac{- \mu(x)}{x^2} + \frac{\mu'(x)}{x}  + \mu''(x)}.
\end{align*}
But, $\mu''(x) < 0$ since $\mu$ is concave and $\mu(x)/x$ is strictly decreasing, so
\begin{align*}
	\diff{}{x} \frac{\mu(x)}{x}
		= \frac{- \mu(x)}{x^2} + \frac{\mu'(x)}{x} < 0,
\end{align*}
so we conclude that $\beta_1(x)$ is concave.
} 

%
%
\section{Inverting the defining relation for the MOC of the velocity}\label{S:InvertingEulerToGetVectorField}

\noindent 
Our intent in this section is start with a MOC, $\mu$, having all of the properties stated in \refT{muConcave} and invert the relation in \refE{A} to obtain a function $\al$ and hence $\theta$ that satisfies all of the properties of \refL{logThetaConvex}. Such a $\theta$ can then give the $L^p$-norms of a Yudovich vorticity as in \refD{YudovichVorticity}.

That $\mu$ is concave follows directly from \refE{A} and requires no special properties of $\theta$. When $\theta$ comes from the $L^p$-norms of a vorticity field, however, $p \log \theta(p)$ must be convex and $\log \al$ must be strictly convex (see \refL{logThetaConvex}). It is $\log \al$ being strictly convex that is critical to obtaining an inverse, as we can see by re-expressing \refE{A} in terms of a Legendre transformation.

\begin{definition}
	Let $f \colon I \to \R$ be a strictly convex function on an interval, $I$. We define its
	\textit{Legendre transformation}, $f^*$, by
	\begin{align*}
		f^*(x) = \sup_{\eps \in I} \set{x \eps - f(\eps)}.
	\end{align*}
	The domain of $f^*$ consists of all $x$ in $\R$ for which the supremum is finite.
\end{definition}

We can write \refE{A} in terms of a Legendre transformation. From \refE{lambdar},
\begin{align}\label{e:Leg}
	\begin{split}
	\log \pr{\frac{\mu(x)}{x}}
		&= \inf_{\eps \in \Cal{A}} \set{- 2 \eps \log x + \log \al(\eps)} \\
		&= - \sup_{\eps \in \Cal{A}} \set{2 \eps \log x - \log \al(\eps)}
		= - (\log \al)^*(2 \log x).
	\end{split}
\end{align}
Thus,
\begin{align*}
	\mu(x) = x e^{- (\log \al)^*(2 \log x)}.
\end{align*}

Because we have restricted the Legendre transformation to strictly convex functions, $f^*$ is also strictly convex, and the Legendre transformation is an involution ($(f^*)^* = f$). See, for instance, Section 14 of \cite{Anrnold1991}. Hence, letting $u = 2 \log x$, \refE{Leg} becomes
$
	(\log \al)^*(u) = - \la(- u/2).
$
Letting $\ol{\la}(s) = - \la(- s/2)$, we have
\begin{align*}
	\log \al(x)
		&= (\ol{\la})^*(x)
		= \sup_{\eps \in \R} \set{x \eps - \ol{\la}(\eps)}
		= \sup_{\eps \in \R} \set{(-x)(-\eps) - (-\la(-\eps/2))} \\
		&= \sup_{\eps \in \R} \set{(-2 x)\eps - (-\la(\eps))}
		= (-\la)^*(-2 x).
\end{align*}
Thus,
\begin{align}\label{e:alLegendre}
	\al(x) = e^{(-\la)^*(-2 x)}.
\end{align}
As long as $\la$ is strictly concave, so that $-\la$ is strictly convex, we can perform the inversion.

There are three limitations of using \refE{alLegendre} alone. First, $\la$ may be strictly concave only near the origin. Second, it is not clear what the domain of $\al$ is. In particular, we need the domain to include $0$: as we will see, $\mu$ satisfying the Osgood condition is required to insure this. Third, it is not clear from \refE{alLegendre} that $\refE{YudoCond}$ is enough to insure that $p \log \theta(p)$ is convex. For this reason, we give an explicit method for inverting \refE{A} in \refT{Invertmu}.

\begin{theorem}\label{T:Invertmu}
	Let $\mu$ be a strictly increasing $C^2$ MOC satisfying the Osgood condition,
	\refE{muOsgood}, with $\la \colon \R \to \R$
	and $\eps \colon \R \to \R$ given by
	\begin{align}\label{e:laDef}
		\la(r) = r + \log \mu(e^{-r}), \quad \eps(r) = \frac{1}{2} \la'(r).
	\end{align}
	Assume that there is a neighborhood $\Cal{N}$ of $r = \iny$ on which $\la(r)$ is strictly concave
	(or, equivalently, a neighborhood of the origin on which $A$, given by \refE{ADef}, is strictly decreasing).
	Then $\eps$ is invertible on $\Cal{N}$, and calling that inverse, $\eta$, and letting
	\begin{align}\label{e:alSolThm}
		\al(\eps)
			= C_0 \exp \left\{\la(\eta(\eps)) - 2 \eta(\eps) \eps \right\},
	\end{align}
	we have $\theta(p) = p^{-1} \al(p^{-1})$ in a neighborhood of $p = \iny$ for some $C_0 > 0$.
	The function $\log \al$ is strictly convex in a neighborhood of the origin.
	If $\refE{YudoCond}$ holds then $\log \theta(1/\eps)$ is convex in a neighborhood of the origin,
	while $p \log \theta(p)$ is convex in a neighborhood of $\iny$.
\end{theorem}
\begin{proof}
First observe that $\la$ strictly concave in a neighborhood of infinity is equivalent to $A$ being strictly decreasing in a neighborhood of the origin by virtue of \refE{lambdapr}. Combined with \refE{limsupA} this gives that \refE{limxmupmu} holds.

Examining the proof of \refT{muConcave}, the starting point now being the function $\mu$ rather than the function $\al$, we see that the first equality in \refE{lambdar} along with \refE{laPrimeeps} define $\eps$ as a function of $r = - \log x$, where $\eps$ gives the location of the infimum in the defining relation, \refE{A}, between $\al$ and $\mu$. The invertibility of $\eps = \eps(r)$ required $\la$ to be strictly concave (this also gave $\eps' < 0$) and the condition on $\mu$ in \refE{limxmupmu} insures that $\eps(\iny) = 0$; together, these two facts give the invertibility of $\eps$ on $\Cal{N}$.

\Ignore{ 
Now, $\la'(r)$ and so $\eps(r)$ will have an inverse in a neighborhood of $\iny$ if and only $\la'(r)$ is strictly monotone in such a neighborhood. Since $\mu$ is twice continuously differentiable, this means that
\begin{align*}
	\la''(r)
		&= \diff{}{r} \pr{\diff{}{x} \log \pr{\frac{\mu(x)}{x}} \diff{x}{r}}
		= - \diff{}{r} \pr{x\diff{}{x} \log \pr{\frac{\mu(x)}{x}}} \\
		&= x\diff{}{x} \log \pr{\frac{\mu(x)}{x}}
		- x \frac{d^2}{d x^2} \log \pr{\frac{\mu(x)}{x}} \diff{x}{r} \\
		&= x\diff{}{x} \log \pr{\frac{\mu(x)}{x}}
			+ x^2 \frac{d^2}{d x^2} \log \pr{\frac{\mu(x)}{x}},
\end{align*}
where we used $dx /dr = - x$.

We now obtain a condition under which $\la'(r)$ and so $\eps(r)$ are monotonic functions of $r$ in some neighborhood, $\Cal{N}$, of $\iny$ and so are invertible in $\Cal{N}$. To see this, we calculate,
\begin{align*}
	\la''(r)
		&= \la''(r(x))
		= - \frac{\mu(x) (\mu'(x) + x \mu''(x)) - x (\mu'(x))^2}{(\mu(x))^2} (- e^{-r}) \\
		&= \frac{\mu'(x)}{\mu(x)} x + \frac{\mu(x) (x \mu''(x)) - x (\mu'(x))^2}{(\mu(x))^2} x \\
		&= \frac{\mu'(x)}{\mu(x)} x + \frac{\mu''(x)}{\mu(x)} x^2
			-  \pr{\frac{\mu'(x)}{\mu(x)}}^2 x^2 \\
		&= A(x) \pr{1 - A(x)} + \frac{\mu''(x)}{\mu(x)} x^2.
\end{align*}

Now,  $0 < A(x) < 1$ and hence $0 < A(x) (1 - A(x)) < 1$, and $\mu''(x) \le 0$, $\mu$ being concave, so it is not clear that $\la''(r(x))$ is signed near $x = 0$---that is, near $r = \iny$. Hence, we add this as a condition to allow us to make the inversion of $\eps$ on some neighborhood, $\Cal{N}$. Observe that this condition would hold if as $x \to 0^+$ the expression on the right-hand side above has a limit and that limit is nonzero.
} 

Using \refE{logalp2r}, we have
\begin{align*}
	\diff{}{r} (\log \al(\eps(r)))
		= (\log \al)'(\eps(r)) \eps'(r)
			= - 2 r \eps'(r).
\end{align*}
Hence,
\begin{align*}
	\log \al(\eps(r))
		&= - 2 \int r \eps'(r) \, dr
		= - 2 \brac{r \eps(r) - \int \eps(r) \, dr} \\
		&= 2 \brac{-r \eps(r) + \int \frac{1}{2} \la'(r) \, dr}
		= \la(r) - r \la'(r) + C
\end{align*}
so
\begin{align}\label{e:alphaEps}
	\al(\eps(r))
		= C \exp \brac{\la(r) - r \la'(r)}
		= C \exp \brac{\la(r) - 2 r \eps(r)},
\end{align}
and \refE{alSolThm} is just a re-expression of \refE{alphaEps}.

Because $\al(x) = x^{-1} \theta(x^{-1})$, we have
\begin{align*}
	\log \al(x)
	 	= - \log x + \log \theta \pr{\frac{1}{x}}
		= - \log x + x \varphi \pr{\frac{1}{x}},
\end{align*}
where $\varphi(p) = p \log \theta(p)$ as in \refL{logThetaConvex}, so
\begin{align*}
	(\log \al)' (x)
		&= - \frac{1}{x} + \varphi \pr{\frac{1}{x}} - \frac{1}{x} \varphi' \pr{\frac{1}{x}}, \\
	(\log \al)''(x)
		&= \frac{1}{x^2} + \frac{1}{x^3} \varphi'' \pr{\frac{1}{x}}.
\end{align*}
Hence, taking the derivative of \refE{logalp2r} gives
\begin{align*}
	-2
		&= (\log \al)''(\eps(r)) \eps'(r)
		= \eps'(r) \brac{\frac{1}{\eps(r)^2} + \frac{1}{\eps(r)^3} \varphi'' \pr{\frac{1}{\eps(r)}}}.
\end{align*}
Since $\eps(r)$ is strictly decreasing on $\Cal{N}$, $\log \al$ is strictly convex on $\Cal{N}$.
Also,
\begin{align}\label{e:varphippInTermsOfLambda}
	\begin{split}
	\varphi'' \pr{\frac{1}{\eps(r)}}
		&= -\eps(r) \brac{2 \frac{\eps(r)^2}{\eps'(r)} + 1}
		= -\frac{\la'(r)}{2} \brac{\frac{(\la'(r))^2}{\la''(r)} + 1} \\
		&= -\frac{\la'(r)}{2 \la''(r)} \brac{(\la'(r))^2 + \la''(r)}.
	\end{split}
\end{align}
This shows that $\varphi$ is convex if $\refE{YudoCond}$ holds, as long as $\la$ is strictly concave. The convexity of $\log \theta(1/\eps)$ then follow as in the proof of \refL{logThetaConvex}.
\DetailSome{ 
Let $\eta$ be the inverse function of $\eps$ on $\Cal{N}$, so that $\eta(\eps(r)) = r$. Then we can write \refE{alODE} as
\begin{align*}
	\al'(\eps) = - 2 \eta(\eps) \al(\eps)
		\text{ or }
		(\log \al)'(\eps) = - 2 \eta(\eps).
\end{align*}
Integrating from $\eps$ to $1/p_0$ gives
\begin{align*}
	\log \al(1/p_0) - \log \al(\eps)
		= - 2 \int_\eps^{1/p_0} \eta(s) \, ds
\end{align*}
so
\begin{align}\label{e:alSol}
	\al(\eps)
		= \al(1/p_0) \exp \pr{2 \int_\eps^{1/p_0} \eta(s) \, ds}.
\end{align}
(We have left undetermined the constant, $\al(1/p_0)$, which is of little importance.)

We can write \refE{alSol} differently by making the change of variables, $r = \eta(s)$ in the integral. Then $s = \eps(r)$ so $ds = \eps'(r) \, dr$ and the integral becomes
\begin{align*}
	\int _{\eta(\eps)}^{\eta(1/p_0)} r \eps'(r) \, dr
		&= r \eps(r)|_{\eta(\eps)}^{\eta(1/p_0)}
			- \int _{\eta(\eps)}^{\eta(1/p_0)} \eps(r) \, dr,
\end{align*}
and from \refE{laPrimeeps}
\begin{align*}
	\int _{\eta(\eps)}^{\eta(1/p_0)} \eps(r) \, dr
		= \frac{1}{2} \int _{\eta(\eps)}^{\eta(1/p_0)} \la'(r) \, dr
		= \frac{1}{2} \la(r)|_{\eta(\eps)}^{\eta(1/p_0)}.
\end{align*}
Thus,
\begin{align*}
	\int _{\eta(\eps)}^{\eta(1/p_0)} r \eps'(r) \, dr
		&= {\eta(1/p_0)} / p_0 - \eta(\eps) \eps
		- \frac{1}{2} \brac{\la(\eta(1/p_0)) - \la(\eta(\eps))}
\end{align*}
and \refE{alSol} becomes
\begin{align}\label{e:alSol2}
	\begin{split}
	\al(\eps)
		&= \al(p_0^{-1}) \exp \left\{2 \eta(p_0^{-1}) /p_0 - 2 \eta(\eps) \eps
			- \la(\eta(p_0^{-1})) + \la(\eta(\eps))\right\} \\
		&= C \exp \left\{\la(\eta(\eps)) - 2 \eta(\eps) \eps \right\}.
	\end{split}
\end{align}

	To demonstrate the convexity of
	$\log \al$ and $\log \theta(1/\eps)$, observe that
	\begin{align*}
		\log \al(\eps)
			= \log C_0 + \la(\eta(\eps)) - 2 \eta(\eps) \eps
	\end{align*}
	so
	\begin{align*}
		(\log \al(\eps))''
			&= (\la'(\eta(\eps)) \eta'(\eps))' - 2 (\eta'(\eps) \eps + \eta(\eps))' \\
			&= \la''(\eta(\eps))(\eta'(\eps))^2
				+ \la'(\eta(\eps)) \eta''(\eps)
				- 2 \eps \eta''(\eps) - 4 \eta'(\eps).
	\end{align*}
	But,
	\begin{align*}
		\la'(\eta(\eps)) = 2 \eps(\eta(\eps)) = 2 \eps,
	\end{align*}
	so
	\begin{align*}
		(\log \al(\eps))''
			= \la''(\eta(\eps))(\eta'(\eps))^2
				 - 4 \eta'(\eps)
			= \eta'(\eps) \brac{\la''(\eta(\eps)) \eta'(\eps) - 4}.
	\end{align*}
	But,
	\begin{align}\label{e:lappeta}
		\la''(\eta(\eps)) \eta'(\eps)
			= (\la'(\eta(\eps))'
			= (2 \eps)'
			= 2,
	\end{align}
	so
	\begin{align}\label{e:logalppetap}
		(\log \al(\eps))''
			= \la''(\eta(\eps))(\eta'(\eps))^2
				 - 4 \eta'(\eps)
			= - 2\eta'(\eps) < 0.
	\end{align}
	Thus, $\log \al(\eps)$ is strictly convex, since $\eta$ strictly decreases with $\eps$ in a
	neighborhood of the origin,
	this being equivalent to our assumption that there is a neighborhood $\Cal{N}$ of $r = \iny$
	on which $\eps$ is invertible.
	
	Now assume that $\refE{YudoCond}_1$ holds and define $\phi(\eps) = \log \theta(1/\eps)$ as
	in \refE{phieps}. Then using \refEAnd{lappeta}{logalppetap},
	\begin{align*}
		\phi''(\eps)
			&= - \frac{1}{\eps^2} - 2 \eta'(\eps)
			= -\frac{1}{\eps^2} - \frac{4}{\la''(\eta(\eps))}.
	\end{align*}
	The convexity of $\phi$ then follows from the calculations following \refE{phipp} in the proof of
	\refT{muConcave} and the convexity of $p \log \theta(p)$ from \refE{logthetappvarphipp}.
	
	\Ignore { 
	We showed above that the neighborhood, $\Cal{N}$, exists if and only if $\la(r)$
	is either strictly concave or strictly convex in a neighborhood of $\iny$. But since $\mu$
	inverts the relation in \refE{A}, $\log \mu(e^{-r})$ is concave by \refT{muConcave}, so
	$\la(r) = r + \log \mu(e^{-r})$ is concave.
	} 
	\Ignore { 
	To rule out convexity, we calculate,
	\begin{align*}
		\frac{d^2}{d r^2} &\log(e^r \mu(e^{-r}))
			= \frac{d^2}{d r^2} \brac{r + \log \mu(e^{-r})}
			= \frac{d^2}{d r^2} \log \mu(e^{-r}) \\
			&= - \diff{}{r} \brac{e^{-r} (\log \mu)'(e^{-r})}
			= e^{-r} (\log \mu)'(e^{-r})
				+ e^{-2r} (\log \mu)''(e^{-r})
	\end{align*}
	Now, if $\log (e^r \mu(e^{-r}))$ is strictly concave then the above equality is positive.
	} 
} 
\end{proof}

\Ignore{ 
\begin{cor}\label{C:CorInvertMu}
Assume that $\mu$ is a MOC as in \refD{MOC}. Let $a > 0$. The following are equivalent:
\begin{itemize}
	\item[(i)]
		The defining relation in (11.1) can be inverted on the interval $(0, a)$ to
		obtain $\al$ (equivalently, $\theta$) from $\mu$ with $\log \al$ twice continuously
		differentiable and strictly convex on $(0, a)$.
				
	\item[(ii)]
		The function $\mu$ is  is twice continuously differentiable, strictly increasing, concave
		on $(0, a)$, and satisfies \refE{limxmupmu}.

\end{itemize}
Moreover, $\theta$ will have the required properties of a function derived from the $L^p$-norms of an underlying vorticity field if and only if $\refE{YudoCond}_1$ or, equivalently, $\refE{YudoCond}_3$, hold and $\mu$ is $C^\iny$ on $(0, \iny)$.
\end{cor}
\begin{proof}
If \refE{A} can be inverted then \refT{muConcave} gives the stated properties of $\mu$ in (ii), showing that (i) implies (ii).

Conversely, from the proof of \refT{Invertmu} it is clear that the invertibility of $\eps$ and so of the relation in \refE{muForm2} to obtain $\theta$ in terms of $\mu$ is equivalent to the strict concavity of $\la$. But the strict concavity of $\la$ follows from \refP{ConcavityImplications}, showing that (ii) implies (i).

The last sentence in the statement of the corollary follows similarly, taking advantage as well of \refL{logThetaConvex}.
\end{proof}
} 

%
%
\Ignore{ 
\begin{remark}
From the expressions for $\mu'(x)$ and $\mu''(x)$ in the proof of \refC{CorInvertMu},
\begin{align*}
	\mu''(x)
		&= (1 - \la'(- \log x)) \la'(- \log x)(-1/x) e^{\la(- \log x)} \\
		&\qquad
			- \la''(- \log x)(-1/x) e^{\la(- \log x)} \\
		&= \frac{1}{x} \pr{-\frac{1 - \la'(- \log x)} \la'(- \log x) + \la''(- \log x)}e^{\la(- \log x)}.
\end{align*}
\end{remark}
} 

\Ignore{ 
\begin{remark}\label{R:InvermuSimpler}
The main point of \refT{Invertmu} is establishing the existence of an inverse to the relation in \refE{A} and giving the properties of $\al$ and $\theta$ that result. The existence of an inverse is already almost implicit in the proof of \refT{muConcave}, once we have \refE{logalp2r}. Also, an equivalent to \refE{alSolThm}, which can be directly obtained by integrating \refE{logalp2r} and using $(\la(r) - r \la'(r))' = - r \la''(r)$, is
\begin{align}\label{e:alphaEps}
	\al(\eps(r))
		= C \exp \brac{\la(r) - r \la'(r)}
		= C \exp \brac{\la(r) - 2 r \eps(r)}.
\end{align}
Letting $\varphi(p) = p \log \theta(p)$ as in \refL{logThetaConvex}, one can also derive
\begin{align*}
	\varphi'' \pr{\frac{1}{\eps(r)}}
		&= -\frac{\la'(r)}{2 \la''(r)} \brac{(\la'(r))^2 + \la''(r)}
\end{align*}
from \refE{alSolThm}. Note that the righthand side is nonnegative if $\refE{YudoCond}$ holds, since $\la$ is strictly increasing and strictly concave.
\end{remark}
} 

\DetailSome{ 
Because $\al(x) = x^{-1} \theta(x^{-1})$, we have
\begin{align*}
	\log \al(x)
	 	= - \log x + \log \theta \pr{\frac{1}{x}}
		= - \log x + x \varphi \pr{\frac{1}{x}}
\end{align*}
so
\begin{align*}
	(\log \al)' (x)
		= - \frac{1}{x} + \varphi \pr{\frac{1}{x}} - \frac{1}{x} \varphi' \pr{\frac{1}{x}}.
\end{align*}
Hence, using \refE{logalp2r},
\begin{align*}
	-2 r
		&= (\log \al)'(\eps(r))
		= - \frac{1}{\eps(r)} + \varphi \pr{\frac{1}{\eps(r)}} - \frac{1}{\eps(r)} \varphi' \pr{\frac{1}{\eps(r)}}.
\end{align*}
Taking the derivative with respect to $r$ gives
\begin{align*}
	-2
		&= \frac{\eps'(r)}{\eps(r)^2}
			- \frac{\eps'(r)}{\eps(r)^2} \varphi' \pr{\frac{1}{\eps(r)}} 
			+ \frac{\eps'(r)}{\eps(r)^2} \varphi' \pr{\frac{1}{\eps(r)}}
			+ \frac{\eps'(r)}{\eps(r)^3} \varphi'' \pr{\frac{1}{\eps(r)}} \\
		&= \frac{\eps'(r)}{\eps(r)^2}
			+ \frac{\eps'(r)}{\eps(r)^3} \varphi'' \pr{\frac{1}{\eps(r)}}
\end{align*}
and thus
\begin{align*}
	\varphi'' \pr{\frac{1}{\eps(r)}}
		&= -\eps(r) \brac{2 \frac{\eps(r)^2}{\eps'(r)} + 1}
		= -\frac{\la'(r)}{2} \brac{\frac{(\la'(r))^2}{\la''(r)} + 1} \\
		&= -\frac{\la'(r)}{2 \la''(r)} \brac{(\la'(r))^2 + \la''(r)},
\end{align*}
which we note is nonnegative if $\refE{YudoCond}$ holds, as long as $\la$ is \textit{strictly} concave.
} 

As an example, let us apply \refT{Invertmu} to the first Yudovich vorticity, where $\theta(p) = C$, which corresponds to $\mu_1$ of \refS{GammaSingleTime}. As we know from \refS{BoundedVorticity}, $\mu(x) = - x \log x$ for sufficiently small $x$, where here and in what follows we ignore immaterial constants.

Then from \refE{laDef},
\begin{align*}
	\la(r)
		&= \log \pr{e^r (- e^{-r} \log(e^{-r})}
		= \log r
\end{align*}
and so $\eps(r) = \frac{1}{2} \la'(r) = 1/(2r)$. Thus, $\eps$ is invertible for \textit{all} $r$ with inverse $\eta(\eps) = 1/(2 \eps)$. Thus, $\eta(\eps) \eps = \frac{1}{2}$ and \refE{alSolThm} gives
\begin{align*}
	\al(\eps)
		&= C_0 \exp \set{\log(2 \cdot 1/(2 \eps)) - 2 \frac{1}{2}}
		= \frac{C}{\eps}
\end{align*}
and thus $\theta(p) = p^{-1} (C p) = C$, recovering $\theta$ to within a multiplicative constant.

The higher Yudovich examples cannot be inverted exactly using \refT{Invertmu} because $\eps(r)$ becomes a transcendental function of $r$ that cannot be inverted in closed form. The following proposition is of some use in this regard, however.

\begin{prop}
	If $\eps$ is overestimated then the expression in \refE{alSolThm} underestimates $\al$.
	That is, assume that $\ol{\eps} \colon \R \to (0, 1/2)$ is $C^1$ and strictly decreasing with
	$\ol{\eps}(\iny) = 0$ and $\ol{\eps} \ge \eps$,and let
	\begin{align}\label{e:olalr}
		\ol{\al}(\ol{\eps}(r))
		= C \exp \brac{\la(r) - 2 r \ol{\eps}(r)}.
	\end{align}
	Then $\ol{\al}(x) \le \al(x)$ for all sufficiently large $x$.
\end{prop}
\begin{proof}
	From \refE{olalr}, $\ol{\al}(\ol{\eps}(r))$ is an increasing function of $r$, so	
	\begin{align*}
		\ol{\al}(\eps(r))
			&\le \ol{\al}(\ol{\eps}(r))
			= C \exp \brac{\la(r) - 2 r \ol{\eps}(r)}
			\le C \exp \brac{\la(r) - 2 r \eps(r)} \\
			&= \al(\eps(r)).
	\end{align*}
\end{proof}

For example, suppose that for $\mu = \mu_m$, $m \ge 2$, we approximate $\eps(r)$ by $1/(2r)$, which is the exact $\eps(r)$ for the first Yudovich example. As can be easily verified, this is an overestimate of the true $\eps(r)$, and we will obtain $\ol{\al}(\eps) = \eps^{-1} \theta_m(\eps^{-1})$ as an underestimate of $\al(\eps)$. This is a kind of dual to the overestimate in \cite{Y1995} of what we call $\mu_m$ from $\theta_m$ using this same estimate for $\eps$.

\Ignore { 
We will find use for the following calculation in the next section:
\begin{align*}
	\frac{1}{1/\eps(r)} - \varphi''(1/\eps(r))
\end{align*}
} 

\Ignore{ 

Finally, we have the following corollary of \refP{Dini} and \refT{Invertmu}:

\begin{cor}\label{C:CorDini}
Let $u$ be a Yudovich velocity field lying in $\Y_\theta$ (defined in \refEAnd{Ytheta}{YNorm}) and let $\mu$ be the corresponding MOC given in \refE{mu}. Then to $S_\mu$ (defined in \refE{Smu}) there corresponds a Yudovich velocity field also lying in $\Y_\theta$ and having the same norm as $u$.
\end{cor}
\begin{proof}
	Let $S_\mu$ play the role of $\mu$ and $\beta$ the role of $\al$ in \refT{Invertmu} and let
	\begin{align*}
		\gamma(r) = r + \log S_\mu(e^{-r})
	\end{align*}
	play the role of $\la$ in that theorem.
	By \refP{Dini}, $S_\mu$ has all of the required properties to apply \refT{Invertmu} to obtain $\beta$,
	with $\beta$ satisfying
	\begin{align*}
		\beta(\gamma'(r)/2)
			= C \exp \brac{\gamma(r) - r \gamma'(r)}
	\end{align*}
	by \refE{alphaEps}. But also by \refP{Dini}, $\gamma$ is in the same germ as $\la$ and $\gamma'$
	is in the same germ as $\la'$ at $+\iny$ so $\beta$ is in the same germ as $\al$ at the origin and hence
	$\pi(p) = p^{-1} \al(p^{-1})$ is in the same germ as theta at $+\iny$. Let $\omega$ be any vorticity
	field with total mass $m = \norm{u}_{E_m}$ having \dots
\end{proof}
} 

\Ignore { 
\begin{remark}
Expressed in terms of $\la$, the Osgood condition \refE{muOsgood} is
\begin{align*}
	\int_0^\iny e^{-\la(r)} \, dr = \iny.
\end{align*}
From \refE{alphaEps},
\begin{align*}
	C e^{-\la(r)} = \frac{e^{-2 r \eps(r)}}{\al(\eps(r))}.
\end{align*}
But, \ldots
\begin{align*}
	\al(\eps(r))
		= C \exp \brac{\la(r) - r \la'(r)}
		= C \exp \brac{\la(r) - 2 r \eps(r)}.
		= \exp \brac{- r^2 \pr{\frac{\la(r)}{r}}'}.
\end{align*}

\end{remark}
} 

\Ignore{ 
Or, since $\la(r) = \log (\mu(x(r))/x(r)) = \log (- 2 \log x(r)) = \log (2 r)$ and $\eta(y) y = (2 y)^{-1} y = 1/2$ for all $y$, \refE{alSol2} gives
\begin{align*}
	\al(\eps)
		&= \al(p_0^{-1}) \exp \left\{1/2 - 1/2 - \log(p_0) + \log(\eps^{-1})\right\}
		= \al(1/p_0) (p_0 \eps)^{-1}
\end{align*}
as before.

Thus, from \refE{epsx},
\begin{align*}
	\eps(x)
		&= \frac{1}{2} - \frac{x}{2} \frac{(- x \log x)'}{(-x \log x)}
		= \frac{1}{2} - \frac{1}{2} \frac{(1 + \log x)}{\log x}
		= - \frac{1}{2 \log x}.
\end{align*}
(This agrees with our choice of $\eps_0$ in \refS{BoundedVorticity}, where $r$ played the role $x$ plays here and we were using $\beta_1$ instead of $\mu$, which accounts for the factor of $1/2$). In terms of $r$ this is $\eps(r) = 1/(2 r)$.

\Ignore{ 
Before continuing the analysis, let us apply this to the simplest Yudovich vorticity, where $\theta(p) = 1$. As we know from \refS{BoundedVorticity}, $\beta_1(x) = - x \log x$ so $\mu(x) = - 2 x \log x$ for sufficiently small $x$, where here and in what follows we ignore immaterial constants. Thus, $\la(r) = \log(- 2 \log x) = \log (2 r)$ so from\refE{laPrimeeps}, $\eps(r) = (1/2) (\log (2 r))' = 1/(2 r)$. (Thus, $\eps(r) =  -1/2 \log x$, which agrees with our choice of $\eps_0$ in \refS{BoundedVorticity}, where $r$ played the role $x$ plays here and we were using $\beta_1$ instead of $\mu$, which accounts for the factor of $1/2$).
} 

Then $\eps$ is invertible with $\eta(\eps) = 1/ (2 \eps)$ and \refE{alSol} gives
\begin{align*}
	\al(\eps)
		& = \al(1/p_0) \exp (\log (1/p_0) - \log \eps )
		= \al(1/p_0) (p_0 \eps)^{-1}.
\end{align*}
Thus,
\begin{align*}
	\theta(1/\eps)
		= \eps \al(\eps)
		= \frac{1}{p_0} \al(1/p_0),
\end{align*}
returning us to a constant $\theta$, representing bounded vorticity.
} 

\Ignore{ 
As this example shows, if we have an explicit expression for $\mu(x)$ and so for $\eps(r)$, it is easier to determine directly whether $\eps$ is invertible than it is to test the condition in \refE{ACond}. In fact, however, the condition in \refE{ACond} fails in this example, because $A(x) \to 1$ while $(\mu''(x)/\mu(x)) x^2 \to 0$ as $x \to 0^+$.
}

%
%
\Ignore{ 
\section{Some potentially useful observations}

\noindent Let $\mu(x) = x (\log x)^2$. Then $\mu$ is Dini but not Osgood. That is,
\begin{align*}
	\int_0^1 \frac{\mu(x)}{x} \, dx < \iny
\end{align*}
while
\begin{align*}
	\int_0^1 \frac{dx}{\mu(x)} < \iny.
\end{align*}
} 

\Ignore{ 

We can see \refE{A} as defining an operator, $A$, from $C((0, 1/p_0]; [0, \iny))$ to $C([0, \iny); [0, \iny))$ with $A(\al) = \mu$. Since $\al$ is derived from the $L^p$-norm, $\theta(p)$, of $\omega$, it follows that $\log \theta(p)$ is convex, $\theta$ is continuous (\textbf{in fact, I think, smooth}), and, with the one exception when $\omega$ is in $L^\iny$, $\theta(\iny) = \iny$. \textbf{Think about how the convexity of $\theta$ affects $\al$.}

For any $\mu$ in $C([0, \iny); [0, \iny))$ define
\begin{align}\label{e:B}
	a(\eps) = \sup \set{x^{2 \eps - 1} \mu(x) \colon x \text{ in } [0, \iny)}.
\end{align}
This defines an operator from, $B$, from $C([0, \iny); [0, \iny))$ to $C((0, 1/p_0]; [0, \iny))$ with $B(\mu) = a$.

Now let us assume that $\mu$ satisfies all of the properties of \refT{muConcave}. The question is, how close $B$ comes to being an inverse of $A$ for such a $\mu$.

We make some easy observations concerning $a$. First, $a$ is strictly decreasing. Second, since $\mu(x)/x \to 0$ as $x \to \iny$, the supremum in \refE{B} occurs in a right neighborhood of the origin. Third, if $\lim_{x \to 0^+}$ $\mu(x)/x = \iny$ then $a(0) = \iny$, and $a$ strictly decreases from $\iny$. Fourth, $a$ is convex, as can be seen using an argument similar to that in the proof of \refT{muConcave} showing that $\mu$ is concave.

Now suppose that $\mu$ is defined in terms of $\al$ as in \refE{A}. Then
\begin{align*}
	a(\eps)
		&= B(\mu) (\eps)
		= B A (\eps)
		= B \sup \set{x^{2 \eps - 1} \mu(x) \colon x \text{ in } [0, \iny)} \\
		&= \sup \set{x^{2 \eps - 1} \inf \set{x^{1 - 2 \delta} \al(\delta) \colon \delta
			\text{ in } (0, 1/p_0]} \colon x \text{ in } [0, \iny)} \\
		&= \sup \set{\inf \set{x^{2(\eps -\delta)} \al(\delta) \colon \delta
			\text{ in } (0, 1/p_0]} \colon x \text{ in } [0, \iny)} \\
		&\le \sup \set{\al(\eps) \colon x \text{ in } [0, \iny)}
		= \al(\eps),
\end{align*}
so $a \le \al$.

it must be that as $\eps \to 0^+$ that the supremum in \refE{B} occurs closer and closer to $x = 0$, and that

----------

Disregarding the insignificant constant $C/4$ and factor $4^{-\eps}$ in \refE{muForm2}, we can write that equation in the form
\begin{align}\label{e:mug}
	\mu(x)
		&= \inf \set{g_x(\eps) \colon \eps \text{ in } (0, 1/p_0]},
		\quad g_x(\eps) :=  x^{1 - 2 \eps} \al(\eps),
\end{align}
where $\al(\eps) = \eps^{-1} \theta(\eps^{-1})$.
Our intent is start with a MOC, $\mu$, having all of the properties stated in \refT{muConcave} and inverting (non-uniquely) the relation in \refE{mug} to obtain a function $\theta$ that is admissible in the sense of \refD{Admissible}.

The simplest possibility is that $g_x$ has a single local (and hence absolute) minimum for each value of $x$.

Then
\begin{align*}
	g_x'(\eps)
		&= \diff{}{\eps} \brac{e^{(\log x) (1 - 2 \eps)} \al(\eps)} \\
		&= -2 \log x \, x^{1 - 2 \eps} \al(\eps))
			+ x^{1 - 2 \eps} \al'(\eps)
\end{align*}
so $g_x'(\eps) = 0$ when
\begin{align*}
	\al'(\eps) = 2 \log x \, \al(\eps).
\end{align*}

Now let $\eps(x)$ be the solution to the last equality at $x$; that is,
\begin{align}\label{e:alPrimex}
	\al'(\eps(x)) = 2 \log x \, \al(\eps(x)).
\end{align}
Then $g_x'(\eps(x)) = 0$ and 
$
	\mu(x) = g_x(\eps(x)).
$

If $\eps$ is an invertible function with inverse $\eta$ then since
\begin{align*}
	\al'(\eps(x))
		&= \frac{(\al(\eps(x)))'}{\eps'(x)}
		= (\al(\eps(x)))' x'(\eps)
\end{align*}
we can write \refE{alPrimex} in the variable $\eps$ as
\begin{align}\label{e:ODEa}
	\al'(\eps) = \frac{2 \log \eta(\eps)}{\eta'(\eps)} \al(\eps).
\end{align}

Now, $\mu(x) = g_x(\eps(x))$, so
\begin{align*}
	\mu'(x)
		&= \pdx{g_x}{\eps} (\eps(x)) \eps'(x)
			+ \pdx{g_x}{x} (\eps(x))
		= \pdx{g_x}{x} (\eps(x)),
\end{align*}
since $\eps(x)$ is a local minimum of $g_x$ by assumption. But
\begin{align*}
	\pdx{g_x}{x}(\eps)
		&= (1 - 2 \eps) x^{-2 \eps} \al(\eps)	
\end{align*}
so
\begin{align*}
	\mu'(x)
		= (1 - 2 \eps(x)) x^{-2 \eps(x)} \al(\eps(x)).
\end{align*}
Hence,
\begin{align}\label{e:almueta}
	\al(\eps)
		= \frac{\eta(\eps)^{2 \eps}}{1 - 2 \eps} \mu'(\eta(\eps)).
\end{align}

Now, \refEAnd{ODEa}{almueta} give two equations between $\al$ and $\eta$. The ODE, \refE{ODEa}, is, in principle, integrable to solve for $\al(\eps)$ in terms of $\eta(\eps)$. If $F(\eps)$ is an antiderivative of $2 \log \eta(\eps)/\eta'(\eps)$ then $\al(\eps) = C e^{F(\eps)}$. Thus,
\begin{align*}
	F(\eps) = C + 2 \eps \log \eta(\eps) - \log(1 - 2 \eps) + \log \mu'(\eta(\eps))
\end{align*}
and taking a derivative with respect to $\eps$,
\begin{align*}
	2 \frac{\log \eta(\eps)}{\eta'(\eps)}
		= 2 \log \eta(\eps) - \frac{2 \eps}{1 - 2 \eps} +\frac{\mu''(\eta(\eps)) \eta'(\eps)}{\mu'(\eta(\eps))}.
\end{align*}
Multiplying both sides by $\eta'(\eps)$,
\begin{align*}
	2 \log \eta(\eps)
		= 2 \log \eta(\eps) \, \eta'(\eps)
			- \frac{2 \eps}{1 - 2 \eps} \eta'(\eps)
			+\frac{\mu''(\eta(\eps))}{\mu'(\eta(\eps))}  \eta'(\eps)^2.
\end{align*}
Since $\mu$ is strictly increasing and concave, the last term is negative. Therefore,
\begin{align*}
	\log \eta(\eps)
		< \pr{\log \eta(\eps) 
			- \frac{2 \eps}{1 - 2 \eps}} \eta'(\eps).
\end{align*}

\Ignore{ 
If $\eta(\eps) \ge 1$ then
\begin{align*}
	1 <  \pr{1 
			- \frac{2 \eps}{(1 - 2 \eps)} \log \eta(\eps)} \eta'(\eps).
\end{align*}
Now, we must have $\eps(0) = 0$ [\textbf{verify this}] and $\eps(x)$ invertible means that both $\eps(x)$ and $\eta(\eps)$ are increasing functions,  so that $\eta'(\eps) \ge 0$. It follows that
\begin{align*}
	\frac{2 \eps}{(1 - 2 \eps)}  \log \eta(\eps) \ge 1
\end{align*}
or
\begin{align*}
	\eta(\eps) \ge e^{1/(2 \eps) - 1}.
\end{align*}
But $\eta(0) = 0$ and $\eta$ is continuous, so this means that \ldots
} 

Now, we must have $\eps(0) = 0$ [\textbf{verify this}] and $\eps(x)$ invertible means that both $\eps(x)$ and $\eta(\eps)$ are increasing functions,  so that $\eta(0) = 0$ and  $\eta'(\eps) \ge 0$. Since it is the only the value of $\al(\eps)$ near $\eps = 0$ that is important, assume that  $\eps$ is small enough that $\eta(\eps) < 1$. Then
\begin{align*}
	\pr{1 - \frac{2 \eps}{(1 - 2 \eps) \log \eta(\eps)}} \eta'(\eps)
		< 1.
\end{align*}
} 

\Ignore{ 

For instance, if we let $\eps(x) = -1/\log x$ (for $x$ sufficiently small), as Yudovich did in a slightly different context, then
\begin{align*}
	\al'(\eps(x))
		&= \frac{(\al(\eps(x)))'}{\eps'(x)}
		= \frac{(2 \log x + \log 4) \al(\eps(x))}{(x (\log x)^2)^{-1}} \\
		&= (2 x (\log x)^3 + (\log 4) x (\log x)^2) \al(\eps(x)).
\end{align*}

Since $\log x = -1/\eps$ and $x = e^{-1/\eps}$, we have
\begin{align*}
	(\log \al(\eps))'
		&= \frac{\al'(\eps(x))}{\al(\eps(x))}
		= - 2 \frac{e^{-1/\eps}}{\eps^3} + (\log 4) \frac{e^{-1/\eps}}{\eps^2}.
\end{align*}

Making the change of variables, $y = -1/\eps$, we have
\begin{align*}
	\int &\brac{- 2 \frac{e^{-1/\eps}}{\eps^3} + (\log 4) \frac{e^{-1/\eps}}{\eps^2}} \, d \eps
		= 2 \int y e^y \, dy - (\log 4) \int e^y \, dy \\
		&= 2 e^y (y - 1) - (\log 4) e^y
		= e^y(2y - 2 - \log 4) \\
		&= - e^{-1/\eps} (2/\eps + 2 + \log 4).
\end{align*}
Thus,
\begin{align*}
	\log \al(\eps)
		&= - e^{-1/\eps} (2/\eps + 2 + \log 4) + C
		= x (-2 \log x + 2 + \log 4) + C
\end{align*}
so
\begin{align*}
	\al(\eps(x))
		= C x^{-2x} e^{2x} 4^x.
\end{align*}
Then,
\begin{align*}
	\mu(x)
		&= g_x(\eps(x))
		=  x^{1 - 2 \eps(x)} 4^{-\eps(x)} \al(\eps(x)))
		= C x x^{2/\log x} 4^{1/\log x} x^{-2x} e^{2x} 4^x \\
		&= C x^{1 -2x} e^{2x} 4^x
\end{align*}
} 

\AdditionalConstraint {
%
%
\section{Dealing with the additional constraint on $\mu$}\label{S:muInversionAdditional}

\noindent 
In \refS{InvertingEulerMOC} we found that $\mu$ was required to satisfy the additional constraint in $\refE{YudoCond}$, which goes beyond $\mu$ being strictly increasing and concave. With the weaker condition that $\la$ is strictly concave or, equivalently, that $A$ is strictly decreasing the inversion can still be accomplished, but $p \log \theta(p)$ need not be convex.

As noted in the first paragraph of the proof of \refT{Invertmu}, if we can obtain a $\mu$ such that $\la$ is concave near the origin then \refE{limxmupmu} holds. So we will focus on obtaining a concave $\la$ or the stronger condition in $\refE{YudoCond}$.

\Ignore{ 
By \refE{lambdapr}, $\la''(r) = e^{-r} A'(e^{-r})$, so whether or not the stronger condition in $\refE{YudoCond}$ holds, $A$ will be decreasing. Combined with \refE{limsupA}, which requires only that $\mu$ be concave, it follows that \refE{limxmupmu} would hold.
} 

In the context of attempting to extend \refT{ConcaveCIG} to account for this additional constraint on $\mu$, it is best to express them in terms of the generating function, $h$.

We assume in this section that we have started with a strongly strictly concave acceptable MOC, $f$, and that we have applied \refT{ConcaveCIG} if necessary to obtain a strictly larger concave function, $\ol{f}$, which we relabel $f$, along with a generating function $h$, the function $g = \log h'$,  a concave CIG, $(f^t)_{t \in \R}$, embedding $f$, and the strongly strictly concave MOC, $\mu$, on the vector field.

\Ignore{ 
It might appear that \refE{limxmupmu} alone constrains how poor the MOC $f$ can be, but that is not that case. For writing \refE{limxmupmu} in the form $x (\log \mu)'(x) \le 1$ it follows much as in \refR{muSandwich} that given any $\eps > 0$ we have, for all $x$ in $(\eps, f(\eps))$,
\begin{align*}
	\mu(x) \le \frac{\mu(\eps)}{\eps} x
\end{align*}
so that
\begin{align*}
	\frac{\eps}{\mu(\eps)} \frac{1}{x}
		\le \frac{1}{\mu(x)}.
\end{align*}
Thus,
\begin{align*}
	\int_\eps^{f(\eps)} \frac{\eps}{\mu(\eps)} \frac{1}{x} \, d x
		\le \int_\eps^{f(\eps)} \frac{1}{\mu(x)} \, dx
		= 1
\end{align*}
from which it follows that $\log (f(\eps)/\eps) \le \mu(\eps)/\eps$ for all $\eps > 0$. Replacing $\eps$ by $x$ and integrating gives
\begin{align*}
	\int_0^1 \log \pr{\frac{f(x)}{x}} \, dx
		\le \int_0^1 \frac{\mu(x)}{x} \, dx
		< \iny
\end{align*}
by \refR{Dini}. But for sufficiently small $a > 0$ such that $f(a) < 1$, $x < f(x) < 1$ so
\begin{align*}
	0
		&= \int_0^a \log 1 \, dx
		< \int_0^a \log \pr{\frac{f(x)}{x}} \, dx
		< \int_0^a \log \pr{\frac{1}{x}} \, dx
		< \iny.
\end{align*}
That is, for \textit{any} acceptable MOC, $\int_0^1 \log (f(x)/x) \, dx < \iny$, so \textbf{f being concave is not  restriction to the growth of f.} Well no duh!
} 

Using \refE{mupphp} and the calculation that led to \refE{loglogpIneq}, the inequality in $\refE{YudoCond}_3$ becomes
\begin{align}\label{e:logloghBound}
	0 \ge (\log h')''(x)
		&\ge \frac{h(x) \mu'(h(x)) - \mu(h(x))}{h'(x) (h(x))^2}
		= \pr{\log h}''(x).
\end{align}
Thus, from \refT{GenSoFar},
\begin{align*}
	&0 < g' = (\log h')' \le (\log h)', \\
	&(\log h)'' \le (\log h')'' = g'' < 0.
\end{align*}

By L'Hospital's rule,
\begin{align*}
	\lim_{x \to 0^+} \frac{\mu(x)}{x}
		= \lim_{x \to 0^+} \mu'(x).
\end{align*}
Here we used $\mu(0) = 0$ and the fact that the latter limit exists (though it might be infinite), as $\mu'$ is decreasing. Also, the limit cannot be zero since $\mu \not \equiv 0$.

Because $\mu$ is strictly concave, $\lim_{x \to 0^+} \mu'(x) = \iny$ unless $\mu$ is sublinear or linear near the origin. Since even unbounded vorticity leads to a superlinear near the origin, we will assume that
the two limits above are infinite. Then using \refEAnd{muhhp}{muphhphpp},
\begin{align*}
	L_1
		&:= \lim_{x \to 0^+} \mu'(x)
		=  \lim_{x \to -\iny}  \mu'(h(x))
		=  \lim_{x \to -\iny} \frac{h''(x)}{h'(x)}
		=  \lim_{x \to -\iny} (\log h')'(x), \\
	L_2
		&:= \lim_{x \to 0} \frac{\mu(x)}{x}
		= \lim_{x \to 0^+} \frac{\mu(h(x))}{h(x)}
		= \lim_{x \to -\iny} \frac{h'(x)}{h(x)}
		= \lim_{x \to -\iny} (\log h)'(x).
\end{align*}
Then since $L_1 = L_2 = \iny$, we can apply L'Hospital's rule to give
\begin{align*}
	A(0)
		&= \lim_{x \to 0} \frac{x \mu'(x)}{\mu(x)}
		= \lim_{x \to 0} \frac{\mu'(x)}{\mu(x)/x}
		= \lim_{x \to -\iny} \frac{(\log h')'(x)}{(\log h)'(x)}
		= \lim_{x \to -\iny} \frac{(\log h')''(x)}{(\log h)''(x)}
\end{align*}
\textit{if the last limit exists}.

If \refE{logloghBound} holds then
\begin{align*}
	\frac{(\log h')''(x)}{(\log h)''(x)}
		\le 1
\end{align*}

----------------

and the last limit exists then it must be 1, but even under this assumption we should not expect the last limit to exist, so we need to examine this issue further.

From the observations above, even without assuming \refE{limxmupmu}, we have
\begin{align*}
	0
		\le F
		:= g'
		= (\log h')'
		< (\log h)'
		=: G
\end{align*}
with $G(x) \to \iny$ as $x \to -\iny$. Strict inequality holds here because $\mu$ is \textit{strongly strictly} concave.

By the extended mean value theorem, for any $a, x$ in $\R$, $a < x$, there exists $\xi$ in $(a, x)$ such that
\begin{align*}
	\frac{F(x) - F(a)}{G(x) - G(a)} = \frac{F'(\xi)}{G'(\xi)}.
\end{align*}
First suppose that \refE{logloghBound} holds so that $F'(\xi)/G'(\xi) \ge 1$. \textbf{Isn't this $\le$ because $F'$ and $G'$ are negative?} Since $F(x), G(x)$ are finite, as $a$ decreases, the left hand side gets arbitrarily close to $F(a)/G(a) < 1$. From this, \refE{limxmupmu} follows (a fact we already know, as noted in
\textbf{This remark no longer exists and in any case the comment is wrong.}

-------

There are two possibilities: 1) for all $a$ there exists a $\xi < a$ such that $F'(\xi)/G'(\xi) < 1$ or 2) there is some $a_0$ in $\R$ for which $F'(\xi)/G'(\xi) \ge 1$ for all $\xi < a$.

\Ignore{ 
--------------------
We have,
\begin{align*}
	A'(x)
		&= \pr{\frac{x \mu'(x)}{\mu(x)}}'
		= \frac{\mu(x) \brac{x \mu''(x) + \mu'(x)} - x (\mu'(x))^2}{(\mu(x))^2} \\
		&= \frac{x\brac{\mu(x) \mu''(x) - (\mu'(x))^2} + \mu(x) \mu'(x)}{(\mu(x))^2} \\
		&= \frac{\mu'(x)}{\mu(x)} \brac{x \frac{\mu''(x)}{\mu'(x)} - \frac{\mu'(x)}{\mu(x)}}.
\end{align*}
Since $\mu > 0$ and $\mu' > 0$,
\begin{align*}
	A'(x) < 0
		\iff \brac{x \frac{\mu''(x)}{\mu'(x)} - \frac{\mu'(x)}{\mu(x)}} < 0.
\end{align*}
Writing the last bracketed expression as
\begin{align*}
	&\brac{\frac{\mu''(x)}{\mu'(x)} - \frac{\mu'(x)}{\mu(x)}}
			+ (x - 1) \frac{\mu''(x)}{\mu'(x)} \\
		&\qquad
		= \frac{\mu''(x) \mu(x) - (\mu'(x))^2}{(\mu(x))^2} \frac{\mu(x)}{\mu'(x)}
			+ (x - 1) \frac{\mu''(x)}{\mu'(x)} \\
		&\qquad
		= \pr{\frac{\mu'(x)}{\mu(x)}}' + (x - 1) \frac{\mu''(x)}{\mu'(x)}
		= (\log \mu)''(x) + (x - 1) \frac{\mu''(x)}{\mu'(x)} \\
		&\qquad
		= (\log \mu)''(x) + (x - 1) (\log \mu')'(x)
\end{align*}
} 

--------------------

\begin{prop}
	If
	\begin{align}\label{e:limxfppfp}
		\lim_{x \to 0} \frac{x f''(x)}{f'(x)}, \, \lim_{x \to 0} \frac{x \mu'(x)}{\mu(x)}
	\end{align}
	both exist with the first limit nonzero then
	\begin{align*}
		\lim_{x \to 0} \frac{x \mu'(x)}{\mu(x)}
			= 1.
	\end{align*}
\end{prop}
\begin{proof}
Assume that both limits in \refE{limxfppfp} exist and that
\begin{align*}
	a := 	\lim_{x \to 0} \frac{x f''(x)}{f'(x)} \ne 0.
\end{align*}
Since $f'(0) = \iny$\MarginNote{Do we point this out somewhere?} we can apply L'Hospital's rule to give
\begin{align*}
	\lim_{x \to 0} \frac{x f'(x)}{f(x)}
		= \lim_{x \to 0} \frac{x f''(x) + f'(x)}{f'(x)}
		= 1 + a.
\end{align*}
Thus, also
\begin{align*}
	\lim_{x \to 0} \frac{f(x) f''(x)}{(f'(x))^2}
		= \lim_{x \to 0} \frac{x f''(x)}{f'(x)}
			\lim_{x \to 0} \frac{f(x)}{x f'(x)}
		= \frac{a}{1 + a}.
\end{align*}
Therefore, using \refE{mupfx},
\begin{align*}
	L
		&:= \lim_{x \to 0} \frac{x \mu'(x)}{\mu(x)}
		= \lim_{x \to 0} \frac{f(x) \mu'(f(x))}{\mu(f(x))}
		= \lim_{x \to 0} \frac{f(x) (\mu'(x) + \mu(x) \frac{f''(x)}{f'(x)})}{f'(x) \mu(x)} \\
		&= \lim_{x \to 0} \brac{\frac{f(x)}{x f'(x)} \frac{x \mu'(x)}{\mu(x)}
			+ \frac{f(x) f''(x)}{(f'(x))^2}}
		= \frac{1}{1 + a} L + \frac{a}{1 + a},
\end{align*}
so $(1 + a) L = L + a$, or, $L = 1$, since $a \ne 0$.
\end{proof}

\textbf{The conclusion of the \textit{proof} of this theorem is not quite what I want, though it is not too bad. But still need a condition so that $\lim x \mu'(x)/\mu(x)$ exists}
\begin{theorem}\label{T:ConcaveCIGExtraProp}
	Let $f$ be any $C^k$ concave acceptable MOC for a flow, $k \ge 3$,
	that satisfies \refE{limxfppfp} with $-1 < a < 0$. Then the function $\mu$ in \refT{ConcaveCIG}
	can be made to satisfy
	\refE{limxmupmu}.
\end{theorem}
\begin{proof}
	Let $\mu$ be as in \refT{ConcaveCIG} and let
	\begin{align*}
		\widetilde{\mu}(x)
			= (1 + a) \mu \pr{x/(1 + a)}.
	\end{align*}
	Then
	\begin{align*}
		\lim_{x \to 0} \frac{ x \widetilde{\mu}'(x)}{\widetilde{\mu}(x)}
			= \frac{1 + a}{1 + a}
			= 1.
	\end{align*}
	Then
	\begin{align*}
		\int_x^{f^{t}(x)} \frac{ds}{\widetilde{\mu}(s)}
			= \frac{1 + a}{1 + a} \int_{x/(1 + a)}^{f^{t}(x)/(1 + a)} \frac{dy}{\mu(y)}
			= t.
	\end{align*}
	Thus, if $(\widetilde{f}^t)_{t \in \R}$ is the corresponding function $CIG$ then
	\begin{align*}
		\widetilde{f}^t(x/(1 + a))
			&= f^t(x)/(1 + a),
	\end{align*}
	or,
	\begin{align*}
		\widetilde{f}^t(y) = \frac{f^t((1 + a) y)}{1 + a}.
	\end{align*}
\end{proof}

--------------

We have,
\begin{align*}
	\la(r)
		= r + \log \mu(e^{-r})
		= - \log x + \log \mu(x)
		= \log \pr{\frac{\mu(x)}{x}}
		=: \beta(x).
\end{align*}
where $r = - \log x$, or, $x = e^{-r}$. Thus, $dx / dr = - e^{-r} = - x$ and
\begin{align*}
	\la'(r)
		&= \diff{\beta(x)}{x} \diff{x}{r}
		= - x \beta'(x).
\end{align*}
Since $\la$ is strictly increasing, $\beta$ is strictly decreasing. Taking another derivative,
\begin{align*}
	\la''(r)
		&= - \beta'(x) \diff{x}{r} - x \beta''(x) \diff{x}{r}
		= x \beta'(x) + x^2 \beta''(x)
		< 0
\end{align*}
since $\la$ is strongly strictly concave (assuming that $p \log \theta p$ is convex). Thus,
\begin{align*}
	\frac{x \beta''(x)}{\beta'(x)} < -1.
\end{align*}

Also,
\begin{align*}
	\beta(h(x))
		&= \log\pr{ \frac{\mu(h(x))}{h(x)}}
		= \log \pr{\frac{h'(x)}{h(x)}}
		= \log (\log h)'(x).
\end{align*}

-------------

From \refE{lambdappr}, $\la$ is strictly concave if and only if
\begin{align}\label{e:muppmupmumore}
	\mu''(x)
		< \frac{\mu'(x)}{\mu(x)} \brac{\mu'(x) - \frac{\mu(x)}{x}},
\end{align}
where the right-hand side is negative when $\mu$ is strictly concave, interestingly enough. In any case, we can also write this as
\begin{align*}
	\mu''(x)
		< \frac{x \mu'(x)}{\mu(x)} \pr{\frac{\mu(x)}{x}}'.
\end{align*}
Then, if $\refE{YudoCond}_3$ is satisfied, which is a stronger condition than $\la$ being strongly convex, we would have
\begin{align}\label{e:AIneq}
	\pr{\frac{\mu(x)}{x}}'
		\le \mu''(x)
		<  A(x)\pr{\frac{\mu(x)}{x}}'
		< 0,
\end{align}
where $A(x) = x \mu'(x)/\mu(x)$ as in \refE{Ax}.

This leads to the following proposition:
\begin{prop}
	If $\refE{YudoCond}$ is satisfied then
	\begin{align*}
		\lim_{x \to 0} A(x) = 1
			\iff \lim_{x \to 0} \frac{\mu''(x)}{(\mu(x)/x)'}
			\text{ exists}
	\end{align*}
	and if the latter limit exists then it must equal 1.
\end{prop}
\begin{proof}
Since $\mu''(x) \to \iny$,\MarginNote{I don't know that I point this out anywhere.}by L'Hospital's rule,
\begin{align*}
	\lim_{x \to 0} \frac{\mu''(x)}{(\mu(x)/x)'}
		= \lim_{x \to 0} \frac{x \mu'(x)}{\mu(x)}
		= \lim_{x \to 0} A(x),
\end{align*}
if the first limit exists. Then by the first inequality in \refE{AIneq}
\begin{align*}
	1
		\le \lim_{x \to 0} A(x).
\end{align*}
But $A(x) \le 1$ for all $x > 0$ so
\begin{align*}
	\lim_{x \to 0} \frac{\mu''(x)}{(\mu(x)/x)'}
		= \lim_{x \to 0} A(x)
		= 1.
\end{align*}

On the other hand, if $A(x) \to 1$ then $\mu''(x)/(\mu(x)/x)' \to 1$ by \refE{AIneq}.
\end{proof}

Or, dividing both sides of \refE{muppmupmumore} by $\mu'(x)$,
\begin{align}\label{e:laConcavemu}
	(\log \mu')'(x)
		= \frac{\mu''(x)}{\mu'(x)}
		< \frac{\mu'(x)}{\mu(x)} - \frac{1}{x}
		= (\log \mu)'(x) - \frac{1}{x}.
\end{align}
Substituting $h(x)$ for $x$, since $\log \mu(h(x)) = \log h'(x) = g(x)$, we have
\begin{align*}
	g'(x)
		= (\log \mu(h(x))'
		= (\log \mu)'(h(x)) h'(x)
\end{align*}
so the right-hand side of \refE{laConcavemu} becomes $g'(x)/h'(x) - 1/h(x)$. The left-hand side of \refE{laConcavemu} becomes
\begin{align*}
	(\log \mu')'(h(x))
		= \frac{\mu''(h(x))}{\mu'(h(x))}
		= \frac{(\log h')''(x)/h'(x)}{(\log h')'(x)}
		= \frac{g''(x)}{g'(x) h'(x)},
\end{align*}
where we used \refEAnd{mupphp}{hLogConnection}. Thus, \refE{laConcavemu} becomes
\begin{align*}
	\frac{g''(x)}{g'(x) h'(x)}
		< \frac{g'(x)}{h'(x)} - \frac{1}{h(x)},
\end{align*}
or,
\begin{align*}
	g''(x)
		&< (g'(x))^2 - \frac{h'(x)}{h(x)} g'(x)
		= (g'(x))^2 - (\log h)'(x) g'(x) \\
		&= g'(x) \pr{(g'(x) - (\log h)'(x)},
\end{align*}
or,
\begin{align*}
	(\log g')'(x) < g'(x) - (\log h)'(x).
\end{align*}

} 

\Ignore{ 
Then from \refEAnd{muhhp}{muphhphpp}, we have
\begin{align*}
	\lim_{x \to 0} \frac{x \mu'(x)}{\mu(x)}
		&= \lim_{x \to -\iny} \frac{h(x) \mu'(h(x))}{\mu(h(x))}
		= \lim_{x \to -\iny} \frac{h(x) h''(x)/h'(x)}{h'(x)} \\
		&= \lim_{x \to -\iny} \frac{h(x)}{h'(x)} \frac{h''(x)}{h'(x)}
		= \lim_{x \to -\iny} \frac{(\log h')'(x)}{(\log h)'(x)}.
\end{align*}
} 

\Ignore{ 

Of we can use the equivalent condition in $\refE{YudoCond}_1$, noting that from \refE{la},
\begin{align*}
	\la(- \log h(r))
		&= - \log h(r) + \log \mu(h(r)) \\
		&= - \log h(r) + \log h'(h^{-1}(h(r))
		= g(r) - \log h(r).
\end{align*}
Thus,
\begin{align*}
	- \la'(-\log h(r)) \frac{h'(r)}{h(r)}
		&= g'(r) - \frac{h'(r)}{h(r)},
\end{align*}
or,
\begin{align*}
	\la'(-\log h(r))
		&= 1 - \frac{g(r) h(r)}{h'(r)}.
\end{align*}
Thus,
\begin{align*}
	- \la''(-\log h(r)) \frac{h'(r)}{h(r)}
		= - \frac{h'(r) (g'(r) h(r) + g(r) h'(r)) - g(r) h(r) h''(r)}{(h'(r))^2}.
\end{align*}
Then
\begin{align*}
	&\brac{\la''(- \log h(r)) + (\la'(- \log h(r)))^2} \frac{h'(r)}{h(r)} \\
		&\qquad
		= \frac{h'(r) (g'(r) h(r) + g(r) h'(r)) - g(r) h(r) h''(r)}{(h'(r))^2} \\
			&\qquad\qquad \pr{1 - \frac{g(r) h(r)}{h'(r)}}^2 \frac{h'(r)}{h(r)}
\end{align*}
} 
{}

%
%
\section{Recovering $\omega^0$ from its $L^p$-norms}\label{S:Recoveringomega0}

\noindent It remains to invert the second map in \refE{Maps2}; namely, $\omega^0 \mapsto \theta(p)$. We should expect to invert this map neither uniquely nor exactly. Lack of uniqueness arises  because any rearrangement or sign change of $\omega^0$ yields the same $\theta$. Since, however, we are interested in square-symmetric vorticities, as in \refD{SquareSymmetric}, the lack of uniqueness is not a problem.

The inability to invert exactly is a more complex issue. To see what is involved, let $\la$ be the distribution function for $\omega^0$; that is, $\la(x) = $ measure of $\set{t \colon \smallabs{\omega^0(t)} > x}$. It is classical that
\begin{align}\label{e:Mellin}
	\theta(p)^p
		= \smallnorm{\omega^0}_{L^p}^p
			= p \int_0^\iny x^{p - 1} \la(x) \, d x
			= p \Cal{M} \la(p),
\end{align}
where $\Cal{M}$ is the Mellin transform. If $\omega^0$ lies in $L^{p_0} \cap L^p$ for all $p \ge p_0$ then $\la(p)$ decays faster than any polynomial in $p$ and it is easy to see from \refE{Mellin} (or directly from the definition of the $L^p$-norm) that $\varphi(p) := p \log \theta(p)$ is complex-analytic in the right-half plane, $\RE p > p_0$. Of necessity, then, $\varphi$ must at least be real-analytic (and real-valued) on $(p_0, \iny)$ to perform the inversion exactly, and we should not expect this to be the case.

Instead, we must look for a way to make an approximate inversion. Toward this end, we will take an approach using \refE{Mellin} that is, in a sense, a generalization of one proof of Stirling's approximation.

To motivate this approximation, we first show how to obtain an approximation for $\varphi$ from $\la$. Assume that a smooth $\la$ is given and let
\begin{align}\label{e:Ip}
	I_p
		= \int_0^\iny x^{p - 1}p \la(x) \, dx
		= \int_0^\iny e^{(p - 1) \log x - \rho(x)} \, dx,
\end{align}
where $\rho = - \log \la$.
Being a distribution function, $\la$ is decreasing, hence $\rho$ is increasing. Suppose also that $\rho$ is convex. Then $f(x) = x^{p - 1} \la(x)$  has a unique maximum at $x = x_p$, where
\begin{align}\label{e:xpDef}
	\rho'(x_p) = (p - 1)/x_p.
\end{align}
Moreover, $x_p$ must increase to $\iny$ as $p \to \iny$.

For $x$ near $x_p$, we thus have
\begin{align}\label{e:CoreApproximation}
	\begin{split}
	(p - 1) & \log x - \rho(x)
		\approx
		(p - 1) \brac{ \log x_p +\frac{1}{x_p}(x - x_p) - \frac{1}{2 x_p^2}(x - x_p)^2} \\
		&\qquad\qquad
		-\brac{\rho(x_p) + \rho'(x_p)(x - x_p) + \frac{\rho''(x_p)}{2}(x - x_p)^2} \\
		&= (p - 1) \log x_p - \rho(x_p)
			- \frac{1}{2} \brac{\frac{p - 1}{x_p^2} + \rho''(x_p)} (x - x_p)^2.
	\end{split}
\end{align}
This approximation should be a good one for all sufficiently large $p$ if
\begin{align}\label{e:rhoAssumption}
	\lim_{x \to \iny} \frac{\rho'''(x)}{\rho''(x)}
		= 0,
\end{align}
an assumption we now add.

Differentiating \refE{xpDef} with respect to $p$ and using the chain rule gives
\begin{align*}
	\rho''(x_p)
		= \frac{1}{x_p \diff{x_p}{p}} - \frac{p - 1}{x_p^2},
\end{align*}
so
\begin{align*}
	(p - 1) & \log x - \rho(x)
		\approx
		(p - 1) \log x_p - \rho(x_p)
			- \frac{(x - x_p)^2}{2 x_p \diff{x_p}{p}}.
\end{align*}

Thus,
\begin{align*}
	I_p
		&\approx
		\int_0^\iny e^{((p - 1) \log x_p - \rho(x_p))}
			\exp\pr{- \frac{1}{2 x_p \diff{x_p}{p}} (x - x_p)^2} \, dx.
\end{align*}
Assuming the Gaussian in the integrand is sufficiently sharp, we have
\begin{align*}
	I_p
		&\approx x_p^{p - 1} e^{-\rho(x_p)} \int_{-\iny}^\iny
			\exp\pr{- \frac{1}{2 x_p \diff{x_p}{p}} (x - x_p)^2} \, dx \\
		&= x_p^{p - 1} e^{-\rho(x_p)}\pr{\frac{1}{2 x_p \diff{x_p}{p}}}^{-1/2} \sqrt{\pi}
		= \sqrt{2 \pi} x_p^{p - \frac{1}{2}} e^{-\rho(x_p)} \pr{\diff{x_p}{p}}^{\frac{1}{2}}.
\end{align*}
(Even if the Gaussian is not sharp, this approximation is at most a factor of two overestimate.)

\Ignore{ 
Thus,
\begin{align*}
	I_p
		&\approx
		\int_0^\iny e^{((p - 1) \log x_p - \rho(x_p))}
			\exp\pr{- \frac{1}{2 x_p \diff{x_p}{p}} (x - x_p)^2} \, dx \\
		&\approx x_p^{p - 1} e^{-\rho(x_p)} \int_{-\iny}^\iny
			\exp\pr{- \frac{1}{2 x_p \diff{x_p}{p}} (x - x_p)^2} \, dx \\
		&= x_p^{p - 1} e^{-\rho(x_p)}\pr{\frac{1}{2 x_p \diff{x_p}{p}}}^{-1/2} \sqrt{\pi}
		= \sqrt{2 \pi} x_p^{p - \frac{1}{2}} e^{-\rho(x_p)} \diff{x_p}{p}.
\end{align*}
} 

Since $\theta(p)^p = p I_p$, we have
\begin{align}\label{e:thetaapprox}
	\begin{split}
	\theta(p)
		&\approx p^{\frac{1}{p}}\brac{\sqrt{2 \pi} x_p^{p - \frac{1}{2}} e^{-\rho(x_p)} \pr{\diff{x_p}{p}}^{1/2}}^
			{\frac{1}{p}} \\
		&= (2 \pi)^{\frac{1}{2p}} x_p^{1 - \frac{1}{2p}} e^{-\rho(x_p)/p} \brac{\diff{x_p}{p}}^{\frac{1}{2p}}.
	\end{split}
\end{align}

Also,
\begin{align*}
	\diff{\rho(x_p)}{p}
		= \rho'(x_p) \diff{x_p}{p}
		= \frac{p - 1}{x_p} \diff{x_p}{p}
		= (p - 1) \diff{}{p} \log x_p.
\end{align*}
Integrating by parts gives
\begin{align*}
	\rho(x_p)
		&= \int (p - 1) \diff{}{p} \log x_p \, dp 
		= C + (p - 1) \log x_p - \int \log x_p \, dp.
\end{align*}
Substituting this into \refE{thetaapprox} gives
\begin{align*}
	\theta(p)
		&\approx (2 \pi)^{\frac{1}{2p}}  x_p^{1 - \frac{1}{2p}}
			\exp \pr{-\frac{C}{p} - \frac{p - 1}{p} \log x_p + \frac{1}{p} \int \log x_p \, dp}
			\brac{\diff{x_p}{p}}^\frac{1}{2p} \\
		&= (2 \pi)^{\frac{1}{2p}}  x_p^{-\frac{1}{p}}
			e^{-\frac{C}{p}} \exp \pr{\frac{1}{p} \int \log x_p \, dp} \brac{\diff{x_p}{p}}^\frac{1}{2p}
\end{align*}
and hence,
\begin{align*}
	\varphi(p)
		= p \log \theta(p)
		\approx -C - \log x_p + \int \log x_p \, dp + \frac{1}{2} \log \brac{\diff{x_p}{p}}.
\end{align*}

This approximation will hold if its derivative,
\begin{align}\label{e:UglyDiffEq}
	-\diff{\log x_p}{p} + \log x_p + \frac{1}{2} \frac{\diffn{x_p}{p}{2}}{\diff{x_p}{p}}
		\approx (p \log \theta(p))' = \varphi'(p),
\end{align}
approximately holds.

\Ignore{ 
From \refE{xpDef}, $x_p \rho'(x_p) = p$. Differentiating implicitly twice leads to
\begin{align}\label{e:d2xpdxp}
	\frac{\diffn{x_p}{p}{2}}{\diff{x_p}{p}}
		= \frac{-2 + x_p \frac{\rho'''(x_p)}{\rho''(x_p)}}{\frac{\rho'(x_p)}{\rho''(x_p)} + x_p}
		\to 0
\end{align}
as $p \to \iny$ by \refE{rhoAssumption}.
} 
From \refE{xpDef}, $\log x_p = \log (p - 1) - \log \rho'(x_p)$, and differentiating gives
\begin{align*}
	\diff{\log x_p}{p}
		= \frac{1}{p - 1} - \frac{\rho''(x_p)}{\rho'(x_p)} \diff{x_p}{p}.
\end{align*}
But $x_p$ is increasing and hence so is $\log x_p$, and $\rho$ is increasing and convex, so all derivatives above are nonnegative. We conclude that $\smallabs{\diff{\log x_p}{p}} < \frac{1}{p - 1}$ and hence vanishes as $p \to \iny$. We also add the assumption that
\begin{align}\label{e:xpAssumption}
	\frac{\diffn{x_p}{p}{2}}{\diff{x_p}{p}} \to 0 \text{ as } p \to \iny.
\end{align}

Therefore, for sufficiently large $p$, we have
\begin{align*}
	x_p \approx e^{\varphi'(p)} =: \beta(p).
\end{align*}

Then \refE{xpDef} becomes
$
	\rho'(\beta(p))
		\approx p/\beta(p)
$
so that (estimating $p - 1$ by $p$)
\begin{align}\label{e:rhobetapDiff}
	\diff{}{p} \rho(\beta(p))
		= \rho'(\beta(p)) \beta'(p)
		\approx p (\log \beta)'(p)
		= p \varphi''(p).
\end{align}
The function $\rho$ increases, so the requirement that $\varphi$ be convex enters here.

Integrating from a sufficiently large $q$ to $p > q$ gives
\begin{align}\label{e:rhobetap}
	\begin{split}
	\rho(\beta(p))
		&\approx \rho(\beta(q)) + \int_q^p s \varphi''(s) \, ds \\
		&= \rho(\beta(q)) + p \varphi'(p) - q \varphi'(q)
			- \int_q^p \varphi'(p) \, dp \\
		&= C_q + p \varphi'(p) - \varphi(p).
	\end{split}
\end{align}

Now assume that \refE{rhobetap} holds exactly, and hence so does \refE{rhobetapDiff}; differentiating it implicitly gives
\begin{align*}
	\rho''(\beta(p)) &[\beta'(p)]^2
		= p \varphi'''(p) + \varphi''(p) - \rho'(\beta(p)) \beta''(p) \\
		&= p \varphi'''(p) + \varphi''(p) - \rho'(\beta(p)) \brac{[\varphi''(p)]^2 + \varphi'''(p)} e^{\varphi'(p)} \\
		&= \brac{p - \rho'(\beta(p)) e^{\varphi'(p)}} \varphi'''(p) + \varphi''(p) - \rho'(\beta(p))
			[\varphi''(p)]^2 e^{\varphi'(p)}.
\end{align*}
But,
\begin{align*} 
	\rho'(\beta(p)) e^{\varphi'(p)}
		= \frac{p \varphi''(p)}{\beta'(p)} e^{\varphi'(p)}
		= \frac{p \varphi''(p)}{\varphi''(p) e^{\varphi'(p)}} e^{\varphi'(p)}
		= p,
\end{align*}
so
\begin{align}\label{e:rhobetapDiffDiff}
	\rho''(\beta(p)) &[\beta'(p)]^2
		= \varphi''(p) - p [\varphi''(p)]^2.
\end{align}
Thus, to insure that $\rho$ is convex (which we assumed to obtain a unique solution to \refE{xpDef}) we must add the condition that $\varphi''(x) \le \frac{1}{x}$ for all sufficiently large $x$. (Then also $\abs{\varphi''(p)} =\smallabs{\diff{\log x_p}{p}} < \frac{1}{p}$, as above.) Hence, with this condition, \refE{xpDef} continues to hold (exactly).

\Ignore{ 
Then the equality in \refE{d2xpdxp} (which only required \refE{xpDef} to hold) gives
\begin{align*}
	\frac{\rho'''(x_p)}{\rho''(x_p)}
		&= \frac{2}{x_p} + \frac{x_p''}{x_p x_p'} \brac{\frac{\rho'(x_p)}{\rho''(x_p)} + x_p} \\
		&= \frac{2}{\beta(p)} + \frac{\beta''(p)}{\beta(p) \beta'(p)}
			\brac{\frac{\rho'(\beta(p))}{\rho''(\beta(p))} + \beta(p)}.
\end{align*}
But $\beta'(p) = \varphi''(p) e^{\varphi'(p)}$ and by \refEAnd{rhobetapDiffAlt}{rhobetapDiffDiff},
\begin{align*}
	\frac{\rho'(\beta(p))}{\rho''(\beta(p))}
		&= \frac{\rho'(\beta(p)) \beta'(p)}{\rho''(\beta(p)) [\beta'(p)]^2} \beta'(p)
		= \frac{p \beta'(p)}{\varphi''(p) - p [\varphi''(p)]^2} \\
		&= \frac{p \varphi''(p) e^{\varphi'(p)}}{\varphi''(p) - p [\varphi''(p)]^2}
		= \frac{p e^{\varphi'(p)}}{1 - p \varphi''(p)}.
\end{align*}
Then using $\beta''(p) = ([\varphi''(p)]^2 + \varphi'''(p)) e^{\varphi'(p)}$, we have
\begin{align*}
	\frac{\rho'''(x_p)}{\rho''(x_p)}
		&= 2 e^{-\varphi'(p)}  + \frac{[\varphi''(p)]^2 + \varphi'''(p)}{\varphi''(p)}
			\brac{\frac{p}{1 - p \varphi''(p)} + 1} \\
		&= 2 e^{-\varphi'(p)}  + \brac{\varphi''(p) + \frac{\varphi'''(p)}{\varphi''(p)}}
			\frac{p  +1 -  p \varphi''(p))}{1 - p \varphi''(p)}
\end{align*}
} 

Differentiating \refE{rhobetapDiffDiff} logarithmically gives
\begin{align*}
	\frac{\rho'''(\beta(p))}{\rho''(\beta(p))} \beta'(p)
		&= \diff{}{p} \log \brac{\varphi''(p) - p [\varphi''(p)]^2} - 2 \diff{}{p} \brac{\log \beta'(p)} \\
		&= \frac{\varphi'''(p) - 2 p \varphi'''(p) \varphi''(p) - [\varphi''(p)]^2} {\varphi''(p) - p [\varphi''(p)]^2}
			- 2 \frac{\beta''(p)}{\beta'(p)}.
\end{align*}
Thus,
\begin{align}
	\begin{split}
	&\frac{\rho'''(\beta(p))}{\rho''(\beta(p))} 
		= \frac{\frac{\varphi'''(p)}{\varphi''(p)} - 2 p \varphi'''(p) - \varphi''(p)}
			{\varphi''(p) - p [\varphi''(p)]^2}
			e^{- \varphi'(p)} \\
		&\qquad\qquad\qquad
			- 2 \frac{[\varphi''(p)]^2 + \varphi'''(p) \varphi''(p)}{[\varphi''(p)]^2} e^{- \varphi'(p)} \\
		&\qquad
		= \frac{\frac{\varphi'''(p)}{\varphi''(p)} - 2 p \varphi'''(p) - \varphi''(p)}
			{\varphi''(p) - p [\varphi''(p)]^2}
			e^{- \varphi'(p)}
			- 2 \brac{1 + \frac{\varphi'''(p)}{\varphi''(p)}} e^{- \varphi'(p)}.
	\end{split}
\end{align}
This places a condition on $\varphi$ that insures that the condition \refE{rhoAssumption} on $\rho$ holds. Or we could place the following conditions on $\varphi$, the first of which strengthens the condition imposed earlier that $\varphi''(x) \le \frac{1}{x}$ for all sufficiently large $x$:
\begin{enumerate}
	\item
		$\frac{1}{x} - \varphi''(x)$ is bounded away from zero for all sufficiently large $x$;
		
		\smallskip
		
	\item
		$\frac{x \varphi'''(x)}{\varphi''(x)} e^{- \varphi'(x)}, \;
		\frac{\varphi'''(x)}{[\varphi''(x)]^2} e^{- \varphi'(x)} \to 0$ as $x \to \iny$.
\end{enumerate}

To ensure that the assumption in \refE{xpAssumption} holds, we calculate,
\begin{align*}
	\frac{\diffn{x_p}{p}{2}}{\diff{x_p}{p}}
		&= \frac{\beta''(p)}{\beta'(p)}
		= \frac{[\varphi''(p)]^2 + \varphi'''(p)}{\varphi''(p)}
		= \varphi''(p) + \frac{\varphi'''(p)}{\varphi''(p)}.
\end{align*}
Condition (1) directly gives $\varphi''(p) \to 0$, and integrating Condition (1) gives $e^{-\varphi'(p)} >  Cp^{-1}$. It then follows from Condition (2) that $\varphi'''(p)/\varphi''(p) \to 0$.

What we have done is to give a rough derivation of the following:
\begin{quote}
If $\varphi$ is convex and satisfies the two conditions above then \refE{rhobetap} can be used to approximately determine $\rho$, and hence a square-symmetric $\omega^0$, from $\varphi$.
\end{quote}

We note that each of the examples of Yudovich in \refE{YudovichExamples} satisfy both of these conditions, and \refE{rhobetap} can be used to determine $\omega^0$ approximately. For instance, when $m = 1$, $\varphi(p) = p \log \log p$, $\varphi'(p) = \log \log p + \frac{1}{\log p}$, $p \varphi'(p) - \varphi(p) = \frac{p}{\log p}$, and $\beta(p) = \log p \, e^{1/\log p}$. For large $p$, then, we have $\beta(p) \approx \log p$ so $\beta^{-1}(x) \approx e^x$. Then from \refE{rhobetap},
\begin{align*}
	\rho(x)
		&\approx C + \beta^{-1}(x) \varphi'(\beta^{-1}(x)) - \varphi(\beta^{-1}(x))
		\approx C + \frac{\beta^{-1}(x)}{\log \beta^{-1}(x)}
		= C + \frac{e^x}{x}.
\end{align*}
Thus, $\la(x) \approx e^{-e^x/x}$, and since $e^{x/2} < e^x/x < e^x$ for all $x > 1$, this $\la$ corresponds to $\omega^0$ square-symmetric with
\begin{align*}
	\omega^0(x) = f(x_1) \log (2 \log (1/x_1)) \CharFunc_{(0, r)}
		= f(x_1) \pr{\log 2 + \log \log (1/x_1)} \CharFunc_{(0, r)}
\end{align*}
in the first quadrant for some $0 < r < e^{-1}$, where $\frac{1}{2} < f(x) < 1$. This is in agreement with \refL{lnLpnorm}.

\DetailSome{ 
From \refE{Ip} and making the change of variables, $x = \beta(q)$, we have
\begin{align*}
	I_{p - 1}
		&\approx \int_{p_0}^\iny e^{(p - 1) \log \beta(q) - \rho(\beta(q))} \beta'(q) \, dq \\
		&= \int_{p_0}^\iny e^{(p - 1) \log \beta(q) - \rho(\beta(q)) + \varphi'(q)} \varphi''(q) \, dq.
\end{align*}
In making this change of variables, the lower bound of integration would be $\beta^{-1}(0)$, but in general, $\beta$ is bounded below. Hence, we choose a large enough $p_0$ and the above formula asymptotically holds for sufficiently large $p$. But assuming that \refE{rhobetap} holds exactly,
\begin{align*}
	(p - 1) &\log \beta(q) - \rho(\beta(q)) + \varphi'(q) \\
		&= (p - 1) \varphi'(q)- (C + q \varphi'(q) - \varphi(q)) + \varphi'(q) \\
		&= -C + (p - q) \varphi'(q) + \varphi(q)
\end{align*}
so
\begin{align}\label{e:Ip1Est}
	I_{p - 1}
		&\approx C \int_{p_0}^\iny e^{(p - q) \varphi'(q) + \varphi(q)} \varphi''(q) \, dq.
\end{align}
Since $\varphi(p) = \log (\theta(p)^p) = \log(p I_{p - 1})$, we have
\begin{align}\label{e:varPhiEst}
	\varphi(p)
		&\approx C + \log p + \log \int_{p_0}^\iny e^{(p - q) \varphi'(q) + \varphi(q)} \varphi''(q) \, dq.
\end{align}

Returning to the example of Yudovich in \refE{YudovichExamples} with $m = 1$, \refE{Ip1Est} yields
\begin{align*}
	I_{p - 1}
		\approx C \int_{\log p_0}^\iny u^{p - 1} e^{-e^u/u} \pr{1 - \frac{1}{u}} \, du
		\approx C \int_{\log p_0}^\iny u^{p - 1} e^{-e^u/u} \, du,
\end{align*}
in agreement with our estimate above that $\rho(x) \approx e^x/x$.
} 

\begin{remark}
Condition (1) is fairly natural, as the need for $\varphi$ to be convex derives, ultimately, from \Holders inequality. Condition (2) arose from trying to insure that the approximation in \refE{CoreApproximation} is accurate. Assuming Condition (1) holds,
$1/\varphi''(x) > x$, so
\begin{align*}
	\frac{x \abs{\varphi'''(x)}}{\abs{\varphi''(x)}} > x^2 \abs{\varphi'''(x)}, \quad
		\frac{\varphi'''(x)}{[\varphi''(x)]^2} > x^2 \abs{\varphi'''(x)}.
\end{align*}
Integrating Condition (1) gives $e^{-\varphi'(x)} > C x^{-1}$, but for the Yudovich examples in \refE{YudovichExamples}, $e^{-\varphi'(x)}$ decreases much more slowly ($(\log x)^{-1}$ for $m = 1$). Thus, Condition (2) can be roughly viewed as saying that $\abs{\varphi'''(x)}$ strays not too far from $x^{-2}$.
Finally, both conditions can be expressed in terms of $\la$, and hence in terms of $\mu$, by using \refE{varphippInTermsOfLambda}, leading to a condition on the third derivative of $\mu$. The resulting forms of the conditions are not, however, immediately enlightening.
\end{remark}

\section*{Acknowledgements}

The author would like to thank Anna Mazzucato for useful discussions concerning \refT{InverseMOC} and Franck Sueur for a number of helpful comments on Part II. The author deeply appreciates the efforts of an anonymous referee, whose detailed comments improved this paper substantially. The author was supported in part by NSF grants DMS-0842408 and DMS-1009545 during the period of this work.

\bibliography{Refs}
\bibliographystyle{plain}

\end{document}